\documentclass[11pt]{article}
\usepackage{amsmath, amssymb, graphicx, multicol, color, epic, eepicemu, epsfig, float,enumerate,hyperref}
\newtheorem{theorem}{Theorem}[section]
\newtheorem{definition}[theorem]{Definition}

\newtheorem{example}[theorem]{Example}
\newtheorem{lemma}{Lemma}[section]
\newtheorem{proposition}{Proposition}[section]
\newtheorem{remark}[theorem]{Remark}
\newtheorem{corollary}{Corollary}[section]
\newenvironment{proof}[1][Proof]{\textbf{#1.} }{\ \rule{0.5em}{0.5em}}
\def\P{{\bf P}}
\def\V{\mathcal{V}}
\def\W{\mathcal{W}}

\begin{document}\allowdisplaybreaks
\title{{\huge \bf Adjusted Viterbi training for hidden Markov models}\\
Technical Report 07-01 \\ 
School of Mathematical Sciences,
Nottingham University, UK}
\vskip 1\baselineskip
\author{{\Large J\"uri Lember}\\\\Tartu University, Liivi 2-507, Tartu 50409, Estonia;
jyril@ut.ee\\\\
{\Large Alexey Koloydenko}$^*$\\\\
School of Mathematical Sciences. Division of Statistics.\\
University of Nottingham. University Park. Nottingham NG7 2RD, UK.\\
Tel: +44(0)115.951.4937, alexey.koloydenko@nottingham.ac.uk\\\\
}
\date{April, 24, 2006}
\maketitle
\bibliographystyle{plain}
\begin{abstract} To estimate the emission
parameters in hidden Markov models  one  commonly uses the EM
algorithm or its variation.  Our primary motivation, however, is 
the Philips speech recognition system wherein the EM algorithm is
replaced by the Viterbi training algorithm. Viterbi training is faster
and computationally less involved than EM, but it is also
biased and need not even be consistent. We
propose an alternative to the Viterbi training -- adjusted Viterbi
training -- that has the same order of computational complexity as
Viterbi training but gives more accurate estimators.
Elsewhere, we studied the adjusted Viterbi training for a
special case of mixtures, supporting 
the theory by  simulations. This paper proves  the adjusted Viterbi training 
to be also possible for  more general hidden Markov models.
\end{abstract} 
\large
\section{Introduction}
\label{sec:intro} We consider a set of  procedures to estimate the
emission parameters of a finite state hidden Markov model given
observations $x_1,\ldots,x_n$. Thus, $Y$ is a Markov chain with
(finite) state space $S$, transition matrix $\mathbb{P}=(p_{ij})$, and
initial distribution $\pi$. To every state $l\in S$ there
corresponds an emission distribution $P_l$ with density $f_l$ that
is  known up to the parametrization $f_l(x; \theta_l)$. When $Y$
reaches state $l$, an observation according to $P_l$ and
independent of everything else, is emitted.\\\\
The standard method for finding the maximum likelihood estimator
of the emission parameters $\theta_l$ is the  EM-algorithm that in
the present context is also known as the {\em Baum-Welch} or {\em
forward-backward algorithm} \cite{baumhmm, emtutorial, vanaraamat,
jelinek, tutorial, raamat}. Since the EM-algorithm can in practice
be slow and computationally expensive, one seeks reasonable
alternatives. One such alternative  is  {\em
Viterbi training} (VT). VT is used in speech recognition \cite{vanaraamat, 
philips,  raamat, Rabiner86, philips2, strom}, natural language modeling 
\cite{ochney},  image analysis \cite{gray2}, bioinformatics \cite{dna2, dna1}.
We are also motivated  by connections with constrained vector quantization \cite{gray,
gray4}. The basic idea behind VT is to replace the computationally
costly expectation (E) step of the EM-algorithm by an
appropriate maximization step with fewer and simpler
computations. In speech recognition, essentially the same training
procedure was already described by L. Rabiner {\em et al.} in
\cite{Rabiner90, Rabiner86} (see also \cite{tutorial, raamat}).
Rabiner considered this procedure as a variation of the {\em
Lloyd algorithm} used in vector quantization,  referring to
Viterbi training as the {\em segmential K-means training}. The
analogy  with the vector quantization is especially pronounced
when the underlying  chain is simply a sequence of {\em i.i.d.}
variables, observations on which are consequently  an {\em
i.i.d.} sample from a mixture distribution. For such mixture
models, VT  was also described by R. Gray {\it et al.} in
\cite{gray}, where the training algorithm was considered in the
vector quantization context under the name of {\em entropy
constrained vector quantization (ECVQ)}.
\\\\
The VT algorithm for estimation of the emission parameters of the
hidden Markov model can be described as follows. Using some initial
values for the parameters,  find  a realization
of $Y$ that maximizes the likelihood of the given
observations. Such an $n$-tuple of states is called a {\em Viterbi
alignment}. Every Viterbi alignment partitions the sample into
subsamples corresponding to the states appearing in the alignment. 
A subsample corresponding to state $l$ is regarded as an {\em i.i.d.} sample
from $P_l$ and is used to find $\hat{\mu}_l$, the maximum
likelihood estimate of $\theta_l$. These estimates are then used
to find an alignment in the next step of the training, and so on.
It can be shown that in general this procedure converges in
finitely many steps; also, it is usually much faster than  the
EM-algorithm.
\\\\
Although  VT is computationally feasible and converges fast, it
has a significant disadvantage: The obtained estimators need not
be (local) maximum likelihood estimators; moreover, they are
generally biased and inconsistent. (VT does not necessarily
increase the likelihood, it is, however, an ascent algorithm
maximizing a certain other objective function.) Despite this deficiency,
speech recognition experiments do not show any significant
degradation of the recognition performance when the EM algorithm
is replaced by  VT. There appears no other  explanation of this
phenomena but the ``curse of complexity'' of the very
speech recognition system based on HMM.\\\\
This paper considers  VT largely outside the speech
recognition context. We  regard the VT procedure merely as a
parameter estimation method, and we address the
following question: Is it possible to adjust VT in such a way that
the adjusted training still has the attractive properties of VT (fast
convergence and computational feasibility) and that the estimators
are, at the same time, ``more accurate'' than those of the
baseline procedure? In particular, we focus on a special property
of the EM algorithm that VT lacks. This property ensures that the
true parameters are asymptotically a fixed point of the algorithm.
In other words, for a sufficiently large sample, the EM algorithm
"recognizes" the true parameters and does not change them much.
VT  does not have this property; even when the initial
parameters are correct (and $n$ is arbitrarily large), an iteration of the
training procedure would in general disturb them. {\em We thus attempt
to modify VT in order to make the true parameters an asymptotic
fixed point of the resulting algorithm.} In accomplishing this task
it is crucial to understand the asymptotic behavior of $\hat P^n_l$,
the empirical measures corresponding to the subsamples obtained
from the alignment. These measures depend on the set of
parameters used by the alignment, and in order for the true
parameters to be asymptotically fixed by (adjusted) VT, the
following must hold: If $\hat P_l^n$ is obtained by the alignment with
the true parameters, and $n$ is sufficiently large, then
$\hat{\mu}_l$, the estimator obtained from $P_l^n$, must be close
to the true parameters. The latter would hold if
\begin{equation}\label{uju}
\hat P_l^n\Rightarrow P_l,\quad{\rm a.s.}
\end{equation}
 and if the estimators  $\hat{\mu}_l$ were continuous\footnote{Loosely speaking, the requirement is
that $\hat{\mu}_l$ is {\em consistent}.} at $P_l$ with respect to
the convergence in \eqref{uju}. The reason why VT does not enjoy
the desired fixed point property is, however, different and is
that \eqref{uju} need not in generally hold. Hence, in order to
improve VT in the aforementioned sense, one needs to study the
asymptotics of the measures $\hat P_l^n$. First of all, one needs to
know if there exist any limiting probability measures $Q_l$ such
that for every $l\in S$
\begin{equation}\label{ujuI}
\hat P_l^n\Rightarrow Q_l,\quad l\in S \quad{\rm a.s.}.
\end{equation}
If such limiting measures exist, then under the above continuity
assumption, the estimators $\hat{\mu}_l$ will converge to $\mu_l$,
where
\begin{equation*}\label{aziz}
\mu_l=\arg\max_{\theta_l}\int \ln f_l(x;\theta_l)Q_l(dx).
\end{equation*}
Taking now into account the difference between $\mu_l$ and the
true parameter, the appropriate adjustment of VT, so called
adjusted Viterbi training (VA) can be defined
(\S\ref{sec:viterbi_alignment_training}).
\\\\
Let us briefly introduce the main ideas of the paper. Let $X$ stand for 
the observable subprocess of our HMM. The core of the problem 
is that the alignment is not defined for infinite sequences of
observations, hence the asymptotic behavior of $\hat P_l^n$ is not
straightforward. To handle this, we introduce the notion
of {\em barrier} (\S\ref{sec:nodes_barriers}). Roughly, a
barrier is a block of observations from a predefined cylinder set
that has the following property: Alignments for contiguous
subsequences of observations enclosed by barriers  can be performed
independently of the observations outside these enclosing
barriers. A simple example of a barrier is an observation $z$ that
determines, or indicates, the underlying state: $x_u=z \Rightarrow y_u=l$, 
$u\leq n$. This happens if  $z$ can only be emitted from $l$. 
This also implies that any Viterbi alignment has to go 
through $l$ at time $u$, and in particular, 
the alignment up to $u$ does not depend on the observations after time $u$. 
If a realization had many such special $z$'s,  then
the alignment could be obtained piecewise, gluing together subalignments
each for each segment enclosed by two consecutive $z$'s. 
\\\\
Barriers are a generalization of this concept. A
barrier is characterized by containing a special observation
termed  a  {\em node} (of order $r\ge 0$). Suppose a barrier is observed
with $x_u$ being its node. The node 
guarantees the existence of state $l$ such that any alignment  
goes through $l$ at time $u$ independently of the observations 
outside the barrier.\\\\
Lemma \ref{neljas} states  (under certain assumptions) the
existence of a special path, or a block, of $Y$ states such 
that, first, the path itself occurs with a positive probability, 
and second, the (conditional) probability of it emitting a barrier 
is positive. 
Hence, by ergodicity of the full HMM process, {\em almost every} sequence of 
observations has infinitely many barriers emitted from this special block.
Next, we introduce random times $\tau_i$'s at which such nodes are emitted.
Note that $\tau_i$'s are unobservable: We do observe the barriers but without knowing 
whether or not the underlying MC is going  through that special block at the same time. 
It is, however, not difficult to see that the times $T_i=\tau_i-\tau_{i-1}$ are
{\em renewal times}, and  furthermore, the process $X$ is {\it regenerative}
with respect to the times $\tau_i$ (Proposition \ref{reg}).\\\\
Recall that almost every sequence of observations has infinitely many
barriers and that every barrier contains a node. For a generic such sequence,
let $u_i$ be the  times of its nodes. Note that $u_i$-s are observable and
that also every for all $j=1,2,\ldots$, 
$\tau_j=u_i$ for some $i\geq j$ (there may be more nodes than those emitted 
from the special block). Using these $u_i$'s as dividers, we
define infinite alignment piecewise (Definition \ref{l6pmatu}). 
Formally we have defined a mapping $v: {\cal X}^{\infty}\to {S}^{\infty}$, 
where ${\cal X}^{\infty}$ is the set of all possible observation sequences, 
and ${S}^{\infty}$ is the set of all possible state-sequences. Hence,
$V=v(X)$ is a well defined {\it alignment process}. We consider
the two-dimensional process $Z:=(X,V)$, and we note that this
process is also regenerative with respect to $\tau_i$'s. We now
define empirical measures $\hat Q^n_l$ that are based on the first $n$
elements of $Z$ (Definition \ref{emp2}).  Using the
regenerativity, it is not hard to show that there exists a limit
measure $Q_l$ such that $\hat Q_l^n\Rightarrow Q_l$, a.s. and
$\hat P_l^n\Rightarrow Q_l$ (Theorem \ref{saddam}). {\em This is the main
result of the paper.}
\\\\
To implement VA in practice, a closed form of $Q_l$ (or ${\hat\mu}_l$) as a 
function of the true parameters is necessary. The
measures $Q_l$ depend on both the transition and the
emission parameters, and  computing $Q_l$ can be very difficult.
However, in the special case of mixture models, the measures $Q_l$
are easier to find.  In \cite{AVT1}, VA is described for the
mixture case. The simulations in \cite{AVT1, AVT3} verify that VA
indeed recovers the asymptotic fixed point property. Also, since the
appropriate adjustment function does not depend on the data, 
each iteration of VA enjoys the same order of computational complexity 
(in terms of the sample size) as the baseline VT. Moreover, for commonly 
used mixtures, such as, for example mixtures of multivariate normal distributions
with unknown means and known covariances, the adjustment function
is available in a closed form (requiring integration with the
mixture densities). Depending on the dimension of the emission,
the number of components, and on the available computational
resources, one can vary the accuracy of  the adjustment.  We
reiterate that, unlike the computations of the EM algorithm,
computations of our adjustment do not involve evaluation and
subsequent summation of the mixture density at every data point.
Also, instead of calculating the measures $Q_l$ exactly, one can
easily simulate them producing in effect a stochastic version
of VA. Although simulations do require extra computations, 
the overall complexity of the stochastically adjusted VT can still 
be considerably lower than that of EM, but this, of course, requires
further investigation.
\section{Adjusted Viterbi training}
In this section, we define the adjusted Viterbi training and we
state the main question of the paper. We begin with the formal
definition of the model.
\subsection{The model}
Let $Y$ be a Markov chain with finite state space
$S=\{1,\ldots,K\}$. We assume that $Y$ is irreducible and
aperiodic with transition matrix $\mathbb{P}=(p_{ij})$ and  initial
distribution $\pi$ that is also the stationary distribution of
$Y$.
We consider a hidden Markov model (HMM), in which to every state
$l\in S$ there corresponds an {\it emission distribution} $P_l$ on
(${\cal X}, {\cal B}$). We assume ${\cal X}$ and ${\cal B}$ are
a separable metric space and the corresponding Borel
$\sigma$-algebra, respectively. Let $f_l$ be a density function of $P_l$
with respect to a certain dominating measure $\lambda$ on $({\cal X},{\cal B})$. 
Two most important concrete examples are $({\mathbb{R}^d},{\cal B})$ with
 Lebesgue measure and discrete spaces with 
the counting measure. We define support of $P_l$ as the interesection 
of all closed sets of probability 1 under $P_l$, and denote such supports by $G_l$. 
\\\\
In our model, to any realization $y_1,y_2,\ldots $ of $Y$ there
corresponds a sequence of independent random variables,
$X_1,X_2,\ldots $, where $X_n$ has the distribution $P_{y_n}$. We
do not know the realizations $y_n$ (the Markov chain $Y$ is
hidden), as we only observe the process $X=X_1,X_2,\ldots$, or,
more formally:
\begin{definition}\label{def:HMM} We say that the stochastic process $X$ is a
hidden Markov model if there is a (measurable) function $h$ such
that for each $n$,
\begin{equation}\label{def}
X_n=h(Y_n,e_n),\quad \text{where} \quad e_1,e_2,\ldots \quad
\text{ are i.i.d. and independent of  } Y.
\end{equation}
\end{definition}
Hence, the emission distribution $P_l$ is the distribution of
$h(l,e_n)$. The distribution of $X$ is completely determined by
the chain parameters $(P,\pi)$ and the emission distributions
$P_l,$ $l\in S$. Moreover, the processes  $Y$ and $X$  have the
following properties:
\begin{itemize}
    \item given $Y_n$, the observation $X_n$ is independent of $Y_m$, $m\ne n$.
    Thus, the conditional distribution of $X_n$ given $Y_1,Y_2\ldots$ depends on $Y_n$ only;
    \item the conditional distribution of $X_n$ given $Y_n$ depends  only on the state of $Y_n$ and not on $n$;
    \item given $Y_1,\ldots, Y_n$, the random variables $X_1,\ldots, X_n$ are independent.
\end{itemize}
The process $X$ is also  mixing and, therefore, ergodic.
\subsection{Viterbi alignment and training}
\label{sec:viterbi_alignment_training}
Suppose we observe  $x_1,\ldots,x_n$, the first $n$ elements of
$X$. Throughout the paper, we will also use the shorter notation $x_{1\ldots n}$. 
A central concept of the paper is  the {\em Viterbi
alignment}, which is any sequence of states $q_{1\ldots n}\in S^n$
that maximizes the likelihood of observing $x_{1\ldots n}$. In
other words, the Viterbi alignment is a maximum-likelihood
estimate  of the realization of $Y_1,\ldots, Y_n$ given
$x_1,\ldots,x_n$. In the following, the Viterbi alignment will be
referred to as the {\it alignment}. We start with the formal
definition of the alignment. First note that for any sequence
$q_{1\ldots n}\in S^n$ of states and sets $B_i\in {\cal B}$ $i=1,\ldots,n$,
$${\bf P}(X_1\in B_1,\ldots, X_n\in B_n, Y_1=q_1,\ldots, Y_n=q_n)=
{\bf P}(Y_1=q_1,\ldots,Y_n=q_n)\prod_{i=1}^n\int_{B_i}f_{q_i}d\lambda,$$
and define $\Lambda(q_1,\ldots,q_n;x_1,\ldots,x_n)$ to be the
likelihood function:
\begin{align*}
\Lambda(q_{1\ldots n}; x_{1\ldots n})\stackrel{\mathrm{def}}{=}{\bf
P}(Y_i=q_i,~i=1,\ldots,n)\prod_{i=1}^nf_{q_i}(x_i).
\end{align*}
\begin{definition}\label{def:alignment}For each  $n\ge 1$, let the set of all the
 alignments be defined as follows:
\begin{equation}\label{alignment}
  \V(x_{1\ldots n})= \{v\in S^n:~
    \forall w\in S^n~\Lambda(v;x_{1\ldots n})\ge  \Lambda(w;x_{1\ldots n})\}.
\end{equation}
Any map $v:{\cal X}^n\mapsto
\V(x_{1\ldots n})$ as well as any element $v\in \V(x_1,\ldots,x_n)$ will also be 
called  an alignment.
\end{definition}
Note that alignments require the knowledge of all the parameters of $X$: $(\pi, P)$ and 
$P_l$ $\forall l\in S$.\\\\
Throughout the paper we assume that  the sample $x_{1\ldots _n}$
is generated by an HMM with transition parameters $(\pi,\mathbb{P})$ and
with the emission distributions $f_i(x; \theta^*_l)$, where
$\theta^*=(\theta^*_1,\ldots,\theta^*_K)$ are the unknown true
parameters. We assume that the transition parameters $\mathbb{P}$ and $\pi$
are known, but the emission densities are known only up to the
parametrization $f_l(\cdot;\theta_l)$, $\theta_l\in \Theta_l$. 
A straightforward generalization to the case when $\psi=(\mathbb{P},\theta^*)$,
all of the  free parameters, are unknown, can be found in \cite{AVT4}. 
In the present case, the likelihood function $\Lambda$ as well as the set of
alignments $\V$ can be viewed as a function of $\theta$. In the
following, we shall write $\V_{\theta}$ for the set of alignments
using the parameters $\theta$. Also, unless explicitly specified,
$v_{\theta}\in \V_{\theta}$ will denote an arbitrary element of
$\V_{\theta}$.
\\\\
The classical method for computing MLE of $\theta^*$ is the
EM algorithm. However, if the dimension of $X$ is high, $n$ is big
and $f_i$'s are complex,  then  EM can be (and often is)
computationally involved.  For this reason, a shortcut, the
so-called {\it Viterbi training} is used. The Viterbi training
replaces the computationally expensive expectation (E-)step by an
appropriate  maximization step that is based on the alignment, and
is generally computationally cheaper in practice than the
expectation. We now describe the Viterbi training in the HMM case.
\begin{center}\underline{Viterbi training}\end{center}
\begin{enumerate}
\item Choose an initial value $\theta^o=(\theta^o_1,\ldots,\theta^o_K)$.
\item Given $\theta^j$, obtain alignment $$v_{\theta^j}(x_{1\ldots n})=v_{1\ldots n}$$
and partition the sample $x_1,\ldots, x_n$ into $K$ sub-samples,
where the observation $x_k$ belongs to the $l^{th}$ subsample if
and only if $v_k=l$. Equivalently, we define (at most) $K$  empirical
measures
\begin{equation}\label{emp}
\hat P_l^n(A;\theta^j,x_{1\ldots n})\stackrel{\mathrm{def}}{=}{\sum_{i=1}^n I_{A\times
l}(x_i,v_i)\over \sum_{i=1}^n I_l(v_i)},\quad A\in {\cal B}, \quad
l\in S.
\end{equation}
\item For every sub-sample find MLE given by:
\begin{equation}\label{mle}
\hat{\mu_l}^n(\theta^j,x_{1\ldots n})=\arg\max_{\theta_l\in \Theta_l}
\int \ln f_l(\theta_l,x)\hat P^n_l(dx;\theta^j,x_{1\ldots n}),
\end{equation}
and take $$\theta^{j+1}_l=\hat{\mu_l}(\theta^j,x_{1\ldots n}),\quad
l\in S.$$ If for some $l\in S$  $v_i\ne l$ for any $i=1,\ldots, n$
($l^{th}$ subsample is empty), then the empirical measure $\hat
P_l^n$ is formally undefined, in which  case we take
$\theta_l^{j+1}=\theta^j_l$. We will  be omitting this exceptional
case from now on.
\end{enumerate} 
The Viterbi training can be
interpreted as follows. Suppose that at some step $j$,
$\theta^j=\theta^*$ and hence $v_{\theta^j}$ is obtained using the
true parameters. The training is then  based on the assumption
that the alignment  $v_{1\ldots n}=v(x_{1\ldots n})$ is
correct, i.e., $v_i=Y_i$, $i=1,\ldots,n$. In this case, the
empirical measures $\hat P_l^n$, $l\in S$ would be obtained from
the i.i.d. sample generated from $P_l(\theta^*)$,  and the MLE
$\hat{\mu_l}^n(\theta^*,X_{1\ldots n})$ would be a natural estimator
to use. Clearly, under these assumptions $\hat
P^n_l(\theta^*,X_{1\ldots n})\Rightarrow P_l(\theta^*)$ a.s.
("$\Rightarrow$"  denotes the weak convergence of probability
measures) and, provided that $\{f_l(\cdot;\theta):\theta\in
\Theta_l\}$ is a $P_l$-Glivenko-Cantelli class and $\Theta_l$ is
equipped with some suitable metric, $\lim_{n\to\infty}
\hat{\mu}^n_l(\theta^*,X_{1\ldots n})= \theta^*_l$ a.s.  Hence, if $n$
is sufficiently large, then  $\hat P_l^n\approx P_l$ and
$$\theta_l^{j+1}=\hat{\mu}^n_l(\theta^*,x_{1\ldots n}) \approx
\theta^*_l=\theta^j_l,\quad \forall l$$ i.e. $\theta^j=\theta^*$
would be (approximately) a fixed point of the training algorithm.
\\\\
A weak point of the foregoing argument is that the alignment in
general is not correct even when the parameters used to find it,
are. So, generally $v_i\ne Y_i$. In particular, this implies that
the empirical measures $\hat{P}^n_l(\theta^*,x_{1\ldots n})$ are not
obtained from an i.i.d. sample from $P_l(\theta^*)$. Hence,
we have no reason to believe that $\hat
P_l^n(\theta^*,X_{1\ldots n})\Rightarrow P_l(\theta^*)$ a.s. and
$\lim_{n\to\infty}\hat{\mu}^n_l(\theta^*,X_{1\ldots n})= \theta^*_l$
a.s. Moreover, we do not even know whether the sequences of
empirical measures $\{\hat P_l^n(\theta^*,X_{1\ldots n})\}$ and MLE
estimators $\{\hat{\mu}^n_l(\theta^*,X_{1\ldots n})\}$ converge (a.s.)
at all.
\\\\
In this paper, we prove the existence of probability measures
$Q_l(\theta,\theta^*)$ (that depend on both $\theta$, the parameters used to obtain the alignments, 
as well as $\theta^*$, the true parameters used to generate the training samples),  
such that for every $l\in S$,
\begin{equation}\label{koondumineI}
\hat P_l^n(\theta^*,X_{1\ldots n})\Rightarrow Q_l(\theta^*,\theta^*),\quad
\text{a.s.}
\end{equation}
for a special choice of the alignment $v_{\theta^*}\in
\V_{\theta^*}$ used to define $\hat P_l^n(\theta^*,x_{1\ldots n})$. (In
fact, adding certain mild restrictions on $P_l$, one can eliminate
the dependence of the above result  on the particular choice of
the alignment $v_{\theta^*}\in \V_{\theta^*}$.) We will also be writing 
$Q_l(\theta)$ for $Q_l(\theta,\theta)$ whenever appropriate.

Suppose also that
the parameter space $\Theta_l$ is equipped with some metric. Then,
under certain consistency assumptions on classes  ${\cal
F}_l=\{f_l(\cdot;\theta_l):\theta_l\in \Theta_l\}$, the convergence
\begin{equation}\label{koondumineII}
\lim_{n\to\infty}{\hat \mu}_l(\theta^*,X_{1\ldots n})
=\mu_l(\theta^*)\quad \text{a.s.}
\end{equation}
can be deduced from \eqref{koondumineI}, where
\begin{equation}\label{koondumineII2}
\mu_l(\theta)\stackrel{\mathrm def}{=}\arg\max_{\theta'_l\in \Theta_l}\int
  \ln f_l(x; \theta'_l)Q_l(dx;\theta).
\end{equation}
We also show that  in general, for the baseline Viterbi training
$Q_l(\theta^*)\ne P_l(\theta^*)$, implying
$\mu_l(\theta^*)\ne \theta^*_l$. In an attempt to reduce the bias
$\theta_l^*-\mu_l(\theta^*)$,  we next propose the {\it adjusted
Viterbi training}.
\\\\
Suppose \eqref{koondumineI} and \eqref{koondumineII} hold. 
Based on \eqref{koondumineII2}, we now consider the mapping
\begin{equation}\label{mapping}
\theta \mapsto \mu_l(\theta),\quad l=1,\ldots, K,
\end{equation}
The calculation of $\mu_l(\theta)$ can be rather involved and it
may have no closed form. Nonetheless, since this function is
independent of the sample,  we can define the following correction
for the bias:
\begin{equation}\label{mapII}
\Delta_l(\theta)=\theta_l-\mu_l(\theta),\quad l=1,\ldots, K.
\end{equation}
Thus, the adjusted Viterbi training emerges as follows:
\begin{center}\underline{Adjusted Viterbi training}\end{center}
\begin{enumerate}
\item Choose an initial value $\theta^0=(\theta^0_1,\ldots,\theta^0_K)$.
\item Given $\theta^j$, perform the alignment and define $K$  empirical measures 
$\hat P_l^n(\theta^j,\theta^*)$ as in \eqref{emp}.
\item For every $\hat P_l^n(\theta^j,x_{1\ldots n})$,
find $\hat{\mu}^n_l(\theta^j,x_{1\ldots n})$ as in \eqref{mle}.
\item For each $l$, define
$$\theta^{j+1}_l=\hat{\mu}^n_l(\theta^j,x_{1\ldots n})+\Delta_l(\theta^j),$$
where $\Delta_l$ as in \eqref{mapII}.
\end{enumerate}
Note that, as desired, for a sufficiently large $n$, the adjusted
training algorithm  has $\theta^*$ as its (approximately) fixed
point: Indeed, suppose $\theta^j=\theta^*$, then 
$\hat{\mu}^n_l(\theta^j,x_{1\ldots n})=\hat{\mu}^n_l(\theta^*,x_{1\ldots n})$. Recalling
\eqref{koondumineII}, it then follows that
$\hat{\mu}^n_l(\theta^*,x_{1\ldots n})\approx
\mu_l(\theta^*)=\mu_l(\theta^j)$, for all $l\in S$. Hence,
\begin{equation}
  \label{eq:approx}
\theta_l^{j+1}=\hat{\mu}_l(\theta^*,x_{1\ldots n})+\Delta_l(\theta^*)
\approx \mu_l(\theta^*)+\Delta_l(\theta^*)=
\theta_l^*=\theta^j,\quad l\in S.
\end{equation}
In \cite{AVT1}, we considered i.i.d. sequence $X_1,X_2,\ldots$,
where $X_1$ has a mixture distribution, i.e. the density of $X_1$
is $\sum_{i=1}^K p_i f_i$. Here $p_i>0$ are the mixture weights.
Such a sequence is  an HMM with the transition matrix satisfying
$p_{ij}=p_j$ $\forall i,j$. In this particular case, the alignment
and the measures $Q_l$ are easy to find. Indeed, for any set of
parameters $\theta=(\theta_1,\ldots \theta_K)$, the alignment
$v_{\theta}$ can be obtained via a {\it Voronoi partition}
${\mathcal S}(\theta)=\{S_1(\theta),\ldots,S_K(\theta)\}$, where
\begin{align}
S_1(\theta)&=\{x: p_1f_1(x;\theta_1)\geq p_jf_j(x;\theta_j),\quad \forall j\in S\} \label{eq:S1}\\
S_l(\theta)&=\{x: p_lf_l(x;\theta_l)\geq p_jf_j(x;\theta_j),\quad
\forall j\in S\}\backslash (S_1\cup\ldots\cup S_{l-1}),\quad
l=2,\ldots, K.\label{eq:S}
\end{align}
Now, the alignment can be defined point-wise as follows:
$v_{\theta}(x_1,\ldots,x_n)=v_{\theta}(x_1)\cdots
v_{\theta}(x_n),$ where $v_{\theta}(x)=l$ if
and only if $x\in S_l(\theta)$.\\
The convergence \eqref{koondumineI} now follows immediately from
the strong law of large numbers as
$\hat{P}^n_l(\theta^*,X_{1\ldots n})\Rightarrow Q_l(\theta^*)$ a.s.,
where
$$q_l(x;\theta^*)\propto f(x;\theta^*)I_{S_l(\theta^*)}=
(\sum_ip_if_i(x;\theta^*))I_{S_l(\theta^*)},\quad l=1,\ldots, K$$
are the densities of respective $Q_l(\theta^*)$.\\\\
Thus, in  the special case of mixtures, the adjustments $\Delta_l$
are easy to calculate and the adjusted Viterbi training is easy to
implement. Simulations in \cite{AVT1} have largely supported the
expected gain in estimation accuracy due to the adjustment $\Delta$ 
with a small extra cost for computing $\Delta$. Indeed, this
extra computation does not affect the algorithm's overall computational complexity as
a function of the sample size, since $\Delta$ depends on the training sample only through 
$\theta^j$, the current value of the parameter. 
\\\\
Due to the time-dependence in the general HMM, the convergence
\eqref{koondumineI} does not  follow immediately from the law of
large numbers. However, the very concept of the adjusted Viterbi
training is based on the existence of the $Q_l$-measures. Thus, in
order to generalize this concept to an arbitrary HMM,
one has to begin with the existence of the $Q_l$-measures, which is
{\em the objective of this paper}.
\section{Nodes and barriers}
\label{sec:nodes_barriers} In this section, we present some
preliminaries that will allow us to prove the convergences
\eqref{koondumineI} and \eqref{koondumineII}. We choose to
introduce the necessary concepts gradually, building up the
general notions on special cases that we find more intuitive and
insightful. For a comprehensive introduction  to HMM's and related
topics we refer to \cite{vanaraamat, tutorial, raamat}, and  an
overview of the basic concepts related to HMM's follows below in
\S\ref{sec:va_nodes}. We then proceed to the notion
of {\em infinite (Viterbi) alignment} (\S\ref{sec:infal}),
developing on the way several auxiliary notions such as {\em nodes} and {\em
barriers}.
\\\\
Throughout the rest of this section, we will be writing $f_l$ and
$\V$ for   $f_l(\cdot;\theta^*_l)$, the true emission
distributions, and $\V_{\theta^*}$, the set of alignments with the
true parameters, respectively.
\subsection{Nodes}\label{sec:va_nodes}
\subsubsection{Preliminaries}
Let $1\le u_1<u_2<\ldots < u_k\le n$. Given any sequence
$a=(a_1,\ldots,a_n)$, write  $a_{u_1\ldots u_k}$ for $(a_{u_1},\ldots,a_{u_k})$
and define also the following objects:
\begin{align*}
  S^{l_1\ldots l_k}_{u_1\ldots u_k}(n)&\stackrel{\mathrm{def}}{=}\{v\in S^n: v_{u_1\ldots u_k}=(l_1,\ldots,l_k)\}.
\end{align*}
Next, given observations $x_{1\ldots n}$,
let us introduce the set of constrained likelihood maximizers defined below:
\begin{equation*}
  \W^l_u(x_{1 \ldots n})=\{v\in S^l_u(n):  \forall w\in S^l_u(n)~\Lambda(v; x_{1 \ldots n})\ge  \Lambda(w; x_{1\ldots n})\}.
\end{equation*}
Next, define the {\it scores}
\begin{equation}\label{eq:delta}
\delta_u(l)\stackrel{\mathrm{def}}{=}\max_{q\in S^l_u(u)}\Lambda(q; x_{1\ldots u}),
\end{equation}
and notice the trivial case: $\delta_l(1)=\pi_lf_l(x_1)$. Then, we
have the following recursion (see, for example, \cite{raamat}):
\begin{equation}\label{rec}
\delta_{u+1}(j)=\max_{l\in S}(\delta_u(l)p_{lj})f_j(x_{u+1}).
\end{equation}
The Viterbi training as well as the Viterbi  alignment inherit
their names from the {\it Viterbi algorithm},  which is a dynamic
programming algorithm for finding $v\in$ $\V(x_{1\ldots n})$. In
fact, due to potential non-uniqueness of such $v$, the Viterbi
algorithm requires a selection rule as part of its specification.
However, for our purposes we will often be manipulating by
$\V(x_{1\ldots n})$ as opposed to by individual  $v$'s, in which
case we will also be identifying the entire $\V(x_{1\ldots n})$
with the output of the  algorithm. This algorithm is based on
recursion \eqref{rec} and on the following relations:
\begin{align}\label{vaike}
 t(u,j)&=\{l\in S: \forall i\in S~\delta_u(l)p_{lj}\ge \delta_u(i)p_{ij}\}, \quad u=1,\ldots,n-1,\\
\label{viimane} \V(x_{1\ldots n})&=\{v\in S^n:\delta_n(v_n)\ge
\delta_n(i)~
                    \forall i\in S, v_u\in t(u,v_{u+1})~1\le u<n\}.
\end{align}
It can also be shown that
\begin{equation}
  \label{eq:constrainedmle}
  \W^l_n(x_{1 \ldots n})=\{v\in S_n^l(n):~v_u\in t(u,v_{u+1})~u=1,\ldots,n-1\}.
\end{equation}
We shall also need the following notation:
\begin{align*}
  \V^{l_1\ldots l_k}_{u_1\ldots u_k}(x_{1\ldots n})&=\{v\in \V(x_{1 \ldots n}): v_{u_iu_{i+1}\ldots u_k}=(l_1,\ldots,l_k)\}.
\end{align*}
and will use subscript $(l)$ to refer to alignments obtained using $(p_{li})_{i\in S}$ (instead of $\pi$) 
as the initial distribution. Thus $\V_{(l)}(x_{1\ldots n})$ stands for the set of all such alignments, and
\begin{align*}
  \V_{(l)u_1\ldots u_k}^{\,\,\,\,\,\,l_1\ldots l_k}(x_{1\ldots n})&=
  \{v\in \V_{(l)}(x_{1 \ldots n}): v_{u_iu_{i+1}\ldots
  u_k}=(l_1,\ldots,l_k)\}.
\end{align*}
Similarly,  $\W_{(l)u_1\ldots u_k}^{\,\,\,\,\,\,l_1\ldots l_k}(x_{1\ldots n})$ will be referring
to the constrained alignments obtained using $(p_{li})_{i\in S}$ 
as the initial distribution.
%
The following Proposition and Corollary reveal more structure of
the alignments.
\begin{proposition}\label{joop0} Let $1\le u\le n$, then
\begin{align}
\W^l_u(x_{1\ldots n})&=\W_u^l(x_{1\ldots u})\times \V_{(l)}(x_{u+1\ldots n}),\label{joop1}\\
\V^l_u(x_{1\ldots n})\neq \emptyset &\Rightarrow\V^l_u(x_{1\ldots
n})=\W^l_u(x_{1\ldots n}).\label{joop2}
\end{align}
\end{proposition}
\begin{proof}
The Markov property implies: for any $q=(q_1,\ldots,q_n)$.
$$\Lambda(q; x_{1\ldots n})=\Lambda(q_{1\ldots u}; x_{1\ldots u})\cdot \Lambda(q_{u+1\ldots n};x_{u+1 \ldots n}|q_u),$$
where
$$\Lambda(q_{u+1\ldots n}; x_{u+1\ldots n}|l)={\bf P}(Y_{u+1\ldots n}=q_{u+1\ldots n}|Y_u=l)\prod_{i=u+1}^nf_{q_i}(x_i).$$
Thus, \eqref{joop1} follows from the  equivalence between
maximizing  $\Lambda(q; x_{1\ldots n})$ over $S^l_u(n)$ on one
hand, and maximizing $\Lambda(q_{1\ldots u}; x_{1\ldots u})$ and
$\Lambda(q_{u+1\ldots n}; x_{u\ldots n}|l)$ over $S^{n-u}$ and
$S^l_u(n)$, respectively and independently, on the other.
\eqref{joop2} follows immediately from the definitions of the
involved sets.
\end{proof}
\begin{corollary}\label{joop3}
\begin{equation}\label{eq:decomp}
\V^l_u(x_{1\ldots n})\neq \emptyset \text{ and } \V^l_u(x_{1\ldots
u})\neq \emptyset \Rightarrow
 \V_u^l(x_{1 \ldots n})= \V^l_u(x_{1\ldots u})\times \V_{(l)}(x_{u+1\ldots n}).
\end{equation}
\end{corollary}
\begin{proof}
The hypotheses of \eqref{eq:decomp} together with \eqref{joop2}
imply $\V^l_u(x_{1\ldots n})=\W^l_u(x_{1\ldots n})$ and
$\V^l_u(x_{1\ldots u})=\W^l_u(x_{1\ldots u})$.  The latter
statements and \eqref{joop1} yield the claim.
\end{proof}
\subsubsection{Nodes and alignment}
We aim at extending the notion of  alignment for infinite HMM's.
In order to fulfil this objective, we investigate properties of
finite alignments (e.g. Propositions \ref{joop0}, and
\ref{nodeprop}) and identify necessary ingredients (e.g. ``node'',
 and ``barrier'') for the development of the extended theory.
We start with the notion of nodes:
\begin{definition}\label{def:node} 
For $1\le u< n$, we call $x_u$ an $l$-node if
\begin{equation}\label{krit}
\delta_u(l)p_{lj}\geq  \delta_u(i)p_{ij},\quad \forall i,j\in S.
\end{equation}
We also say that $x_u$ is a node if it is an $l$-node for some
$l\in S$.
\end{definition}
\begin{figure}
\begin{center}
\setlength{\unitlength}{0.00068in}
\begingroup\makeatletter\ifx\SetFigFont\undefined%
\gdef\SetFigFont#1#2#3#4#5{%
  \reset@font\fontsize{#1}{#2pt}%
  \fontfamily{#3}\fontseries{#4}\fontshape{#5}%
  \selectfont}%
\fi\endgroup%
{\renewcommand{\dashlinestretch}{30}
\begin{picture}(9205,3808)(0,-10)
\put(450,2503){\ellipse{150}{150}}
\put(450,1603){\ellipse{150}{150}}
\put(450,703){\ellipse{150}{150}}
\put(4950,1603){\ellipse{300}{300}}
\put(1350,2503){\ellipse{150}{150}}
\put(1350,1603){\ellipse{150}{150}}
\put(1350,703){\ellipse{150}{150}}
\put(2250,2503){\ellipse{150}{150}}
\put(2250,1603){\ellipse{150}{150}}
\put(2250,703){\ellipse{150}{150}}
\put(4050,2503){\ellipse{150}{150}}
\put(4050,1603){\ellipse{150}{150}}
\put(4050,703){\ellipse{150}{150}}
\put(4950,703){\ellipse{150}{150}}
\put(4950,2503){\ellipse{150}{150}}
\put(5850,703){\ellipse{150}{150}}
\put(5850,1603){\ellipse{150}{150}}
\put(5850,2503){\ellipse{150}{150}}
\put(7650,703){\ellipse{150}{150}}
\put(7650,1603){\ellipse{150}{150}}
\put(7650,2503){\ellipse{150}{150}}
\put(8550,1603){\ellipse{150}{150}}
\put(8550,2503){\ellipse{150}{150}}
\put(8550,703){\ellipse{150}{150}} \path(450,2803)(8850,2803)
\path(525,1678)(1350,2503) \dashline{60.000}(4950,1603)(5850,1603)
\dashline{60.000}(7650,2503)(8550,2503)
\dashline{60.000}(7650,703)(8550,703) \path(7350,1003)(7650,703)
\path(7650,703)(8550,1603) \path(4050,1603)(4950,1603)
\path(2250,1603)(2550,1603) \path(450,403)(8850,403)
\dashline{60.000}(4050,1603)(4950,2503)
\drawline(4950,2503)(4950,2503) \path(4950,1603)(5850,2503)
\dashline{60.000}(3750,703)(4050,703) \path(1350,2503)(2250,1603)
\path(5850,2503)(6150,2203) \thicklines
\dottedline{135}(2850,3103)(3450,3103) \thinlines
\drawline(4650,3103)(4650,3103)
\dashline{60.000}(495,763)(1350,1603)
\dashline{60.000}(1350,1603)(2250,703)
\dashline{60.000}(1350,2503)(2250,2503)
\dashline{60.000}(2250,703)(2550,703)
\dashline{60.000}(450,2503)(1350,703) \thicklines
\dottedline{135}(2880,1603)(3480,1603)
\dottedline{135}(6690,3088)(7290,3088)
\dottedline{135}(6660,1903)(7260,1903) \thinlines
\dashline{60.000}(3750,2203)(4050,2503)
\dashline{60.000}(4950,703)(4050,2503)
\dashline{60.000}(3750,1603)(4050,1603)
\dashline{60.000}(4950,1603)(5850,703)
\put(600,
2428){\makebox(0,0)[lb]{\smash{{{\SetFigFont{12}{14.4}{\rmdefault}{\mddefault}{\updefault}$\delta_1(1)$}}}}}
\put(1200,2623){\makebox(0,0)[lb]{\smash{{{\SetFigFont{12}{14.4}{\rmdefault}{\mddefault}{\updefault}$\delta_2(1)$}}}}}
\put(2100,2623){\makebox(0,0)[lb]{\smash{{{\SetFigFont{12}{14.4}{\rmdefault}{\mddefault}{\updefault}$\delta_3(1)$}}}}}
\put(600,
1528){\makebox(0,0)[lb]{\smash{{{\SetFigFont{12}{14.4}{\rmdefault}{\mddefault}{\updefault}$\delta_1(2)$}}}}}
\put(1200,1723){\makebox(0,0)[lb]{\smash{{{\SetFigFont{12}{14.4}{\rmdefault}{\mddefault}{\updefault}$\delta_2(2)$}}}}}
\put(2100,1723){\makebox(0,0)[lb]{\smash{{{\SetFigFont{12}{14.4}{\rmdefault}{\mddefault}{\updefault}$\delta_3(2)$}}}}}
\put(600,
628){\makebox(0,0)[lb]{\smash{{{\SetFigFont{12}{14.4}{\rmdefault}{\mddefault}{\updefault}$\delta_1(3)$}}}}}
\put(1200,823){\makebox(0,0)[lb]{\smash{{{\SetFigFont{12}{14.4}{\rmdefault}{\mddefault}{\updefault}$\delta_2(3)$}}}}}
\put(2100,823){\makebox(0,0)[lb]{\smash{{{\SetFigFont{12}{14.4}{\rmdefault}{\mddefault}{\updefault}$\delta_3(3)$}}}}}
\put(0,2428){\makebox(0,0)[lb]{\smash{{{\SetFigFont{12}{14.4}{\rmdefault}{\mddefault}{\updefault}1}}}}}
\put(0,1528){\makebox(0,0)[lb]{\smash{{{\SetFigFont{12}{14.4}{\rmdefault}{\mddefault}{\updefault}2}}}}}
\put(0,628){\makebox(0,0)[lb]{\smash{{{\SetFigFont{12}{14.4}{\rmdefault}{\mddefault}{\updefault}3}}}}}
\put(8775,2428){\makebox(0,0)[lb]{\smash{{{\SetFigFont{12}{14.4}{\rmdefault}{\mddefault}{\updefault}$\delta_n(1)$}}}}}
\put(8775,1528){\makebox(0,0)[lb]{\smash{{{\SetFigFont{12}{14.4}{\rmdefault}{\mddefault}{\updefault}$\delta_n(2)$}}}}}
\put(8775,628){\makebox(0,0)[lb]{\smash{{{\SetFigFont{12}{14.4}{\rmdefault}{\mddefault}{\updefault}$\delta_n(3)$}}}}}
\put(0,3703){\makebox(0,0)[lb]{\smash{{{\SetFigFont{12}{14.4}{\rmdefault}{\mddefault}{\updefault}s}}}}}
\put(0,3478){\makebox(0,0)[lb]{\smash{{{\SetFigFont{12}{14.4}{\rmdefault}{\mddefault}{\updefault}t}}}}}
\put(0,3253){\makebox(0,0)[lb]{\smash{{{\SetFigFont{12}{14.4}{\rmdefault}{\mddefault}{\updefault}a}}}}}
\put(0,3028){\makebox(0,0)[lb]{\smash{{{\SetFigFont{12}{14.4}{\rmdefault}{\mddefault}{\updefault}t}}}}}
\put(0,2803){\makebox(0,0)[lb]{\smash{{{\SetFigFont{12}{14.4}{\rmdefault}{\mddefault}{\updefault}e}}}}}
\put(180,
103){\makebox(0,0)[lb]{\smash{{{\SetFigFont{12}{14.4}{\rmdefault}{\mddefault}{\updefault}$v_1=2$}}}}}
\put(1065,103){\makebox(0,0)[lb]{\smash{{{\SetFigFont{12}{14.4}{\rmdefault}{\mddefault}{\updefault}$v_2=1$}}}}}
\put(1980,103){\makebox(0,0)[lb]{\smash{{{\SetFigFont{12}{14.4}{\rmdefault}{\mddefault}{\updefault}$v_3=2$}}}}}
\put(5475,103){\makebox(0,0)[lb]{\smash{{{\SetFigFont{12}{14.4}{\rmdefault}{\mddefault}{\updefault}$v_{u+1}=1$}}}}}
\put(3600,103){\makebox(0,0)[lb]{\smash{{{\SetFigFont{12}{14.4}{\rmdefault}{\mddefault}{\updefault}$v_{u-1}=2$}}}}}
\put(7350,103){\makebox(0,0)[lb]{\smash{{{\SetFigFont{12}{14.4}{\rmdefault}{\mddefault}{\updefault}$v_{n-1}=3$}}}}}
\put(8400,103){\makebox(0,0)[lb]{\smash{{{\SetFigFont{12}{14.4}{\rmdefault}{\mddefault}{\updefault}$v_n=2$}}}}}
\put(4650,103){\makebox(0,0)[lb]{\smash{{{\SetFigFont{12}{14.4}{\rmdefault}{\mddefault}{\updefault}$v_u=2$}}}}}
\put(400,
3103){\makebox(0,0)[lb]{\smash{{{\SetFigFont{12}{14.4}{\rmdefault}{\mddefault}{\updefault}$x_1$}}}}}
\put(1300,3103){\makebox(0,0)[lb]{\smash{{{\SetFigFont{12}{14.4}{\rmdefault}{\mddefault}{\updefault}$x_2$}}}}}
\put(2200,3103){\makebox(0,0)[lb]{\smash{{{\SetFigFont{12}{14.4}{\rmdefault}{\mddefault}{\updefault}$x_3$}}}}}
\put(4000,3103){\makebox(0,0)[lb]{\smash{{{\SetFigFont{12}{14.4}{\rmdefault}{\mddefault}{\updefault}$x_{u-1}$}}}}}
\put(4900,3103){\makebox(0,0)[lb]{\smash{{{\SetFigFont{12}{14.4}{\rmdefault}{\mddefault}{\updefault}$x_u$}}}}}
\put(5800,3103){\makebox(0,0)[lb]{\smash{{{\SetFigFont{12}{14.4}{\rmdefault}{\mddefault}{\updefault}$x_{u+1}$}}}}}
\put(7600,3103){\makebox(0,0)[lb]{\smash{{{\SetFigFont{12}{14.4}{\rmdefault}{\mddefault}{\updefault}$x_{n-1}$}}}}}
\put(8500,3103){\makebox(0,0)[lb]{\smash{{{\SetFigFont{12}{14.4}{\rmdefault}{\mddefault}{\updefault}$x_n$}}}}}
\end{picture} }
\end{center}
\caption{An example of the Viterbi algorithm in action. The solid
line corresponds to the final alignment $v_{1\ldots n}$. The
dashed  links are of the form $(k,l)-(k+1,j)$ with $l\in t(k,j)$
and are not part of the final alignment.  E.g., $(1,3)-(2,2)-(3,3)$ is
because $3\in t(1,2)$, $2\in t(2,3)$. The observation $x_u$ is a
$2$-node, since we have $2\in t(u,j)$  $\forall j\in S$.  We also
see that $v_{1\ldots u}$ is {\em fixed}.}
\label{fig:viterbi_algorithm}
\end{figure}
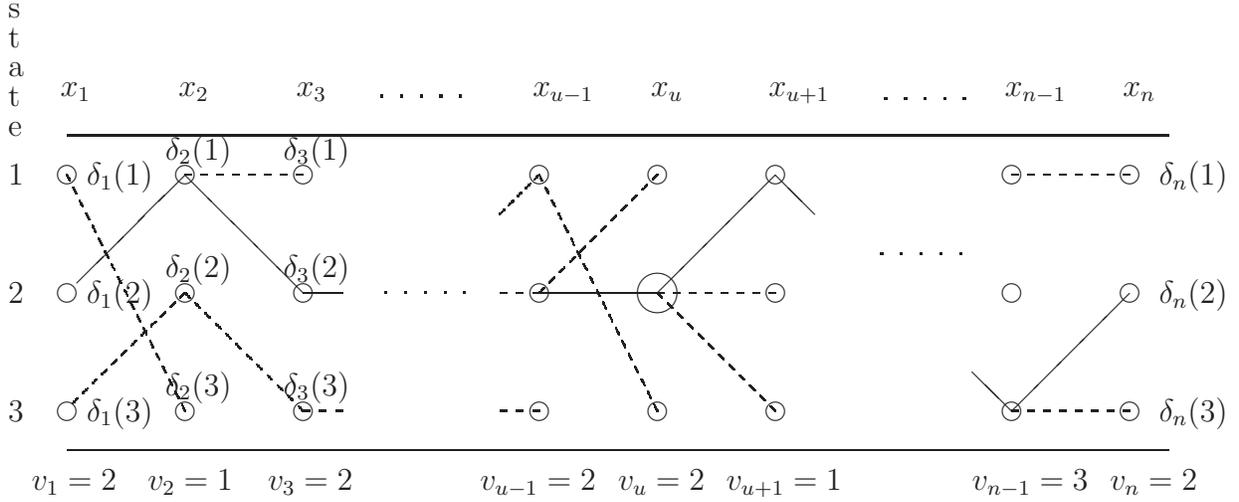
Figure \ref{fig:viterbi_algorithm} illustrates the newly
introduced notion.
\begin{proposition}\label{nodeprop}
\begin{eqnarray}
x_u \text{ is an}~l\text{-node} & \iff & l\in t(u,j)~\forall j\in S, \label{eq:node-t}\\
                               & \Rightarrow & \V_u^l(x_{1\ldots u})\neq \emptyset,
                               \label{eq:node-Va}\\
                               &\Rightarrow & \forall v\in \V(x_{1\ldots n}), \forall
             v^*\in \V_u^l(x_{1\ldots u})~(v^*,v_{u+1\ldots n})\in \V_u^l(x_{1\ldots n}),
             \label{eq:l-tie}\\
             & \Rightarrow & \V_u^l(x_{1\ldots n})\neq \emptyset, \label{eq:node-V}\\
             & \Rightarrow &\text{Right hand side of~}\eqref{eq:decomp}\label{joop}.
\end{eqnarray}
Whether $x_u$ is a node does not depend on $x_i$, $i>u$.
\end{proposition}
\begin{proof}
The final statement follows immediately from Definition
\ref{def:node} and \eqref{eq:delta}, and \eqref{eq:node-t} also
follows immediately from Definition \ref{def:node} and
\eqref{vaike}. Summing both sides of \eqref{krit} over $j\in S$,
we obtain
\begin{equation}\label{eq:lastl}
\delta_u(l)\geq  \delta_u(i),\quad \forall i\in S,
\end{equation}
hence, \eqref{eq:node-Va} holds by \eqref{viimane}.  Note that
\eqref{eq:l-tie}  means that any alignment $v\in\V(x_{1\ldots n})$
can be modified by setting $v_u=l$ and taking $v^*_i\in
t(i,v_{i+1})$ for $i=u-1,u-2,\ldots,1$, and the modified string
remains an alignment, i.e. belongs to  $\V(x_{1\ldots n})$.  Such
a modification is evidently always possible, i.e.,
$(v^*,v_{u+1\ldots n})$ is well-defined since $\V_u^l(x_{1\ldots
u})\neq \emptyset$. For $u=n$ this holds trivially, for $u<n$ this
follows from \eqref{eq:node-t} (as the latter implies $l\in
t(u,v_{u+1})$ for any value of $v_{u+1}$), and \eqref{viimane}.
Also, \eqref{eq:l-tie} implies \eqref{eq:node-V}. Finally, given
\eqref{eq:node-Va} and \eqref{eq:node-V}, Corollary \ref{joop3}
yields \eqref{joop}.
\end{proof}
\begin{remark}\label{fixed}
Note that a modification of $v\in\V(x_{1\ldots x_n})$ possibly
required to enforce $v_u=l$ when $x_u$ is an $l$-node (see proof
of \eqref{eq:l-tie} above) depends only on $x_1,\ldots,x_{u-1}$.
Thus,  if $x_u$ is an $l$-node and if $v^*\in \V_u^l(x_{1\ldots
x_u})$, then for any $n > u$ and any $x_{u+1},\ldots,x_n$
\eqref{eq:l-tie} always guarantees an alignment $v\in\V(x_{1\ldots
n})$ with $v_{1\ldots u}=v^*$, in which case we can call
$v^*$ {\em fixed}, meaning that $v^*$ can be kept as the substring
of the first $u$ components for any alignment based on the
extended observations.
\end{remark}
The fact that  $v\in \V(x_{1\ldots n})$ in general does not imply
$v_{1\ldots u} \in \V(x_{1\ldots u})$ complicates the structure
of the alignments and furthermore emphasizes the significance of
nodes in view of \eqref{joop} and Remark \ref{fixed}.
\begin{corollary}\label{good}
Suppose the observations $x_1,\ldots,x_n$ are such that  for some
$1\le u_1<u_2<\cdots <u_k\le n$, the observations $x_{u_i}$ are
$l_i$-nodes, $i=1,\ldots,k-1$. Then
\begin{eqnarray}\label{pieces} \nonumber
\emptyset \ne \V_{u_1u_2\cdots u_k}^{l_1l_2\cdots l_k}(x_{1\ldots
n})=&&\\
=\V_{u_1}^{l_1}(x_{1\ldots u_1})\times
\V_{(l_1)u_2}^{\,\,\,\,\,\,\,\,l_2}(x_{u_1+1\ldots u_2})\times\cdots\times
\V_{(l_{k-1})u_k}^{\,\,\,\,\,\,\,\,\,\,\,\,\,\,l_k}(x_{u_{k-1}+1\ldots u_k})\times
\V_{(l_k)}(x_{u_k+1\ldots n}).&&
\end{eqnarray}
\end{corollary}
\begin{proof}
By (\ref{eq:node-Va}),
$$\V_{u_i}^{l_i}(x_{1\ldots u_i})\ne \emptyset,\quad
i=1,\ldots,k.$$ By (\ref{eq:node-V})
$$\V_{u_k}^{l_k}(x_{1\ldots n})\ne \emptyset,\quad
\V_{u_i}^{l_i}(x_{1\ldots u_{i+1}})\ne \emptyset \quad
i=1,\ldots,k-1.$$
From (\ref{eq:l-tie}), it now follows
$$\V_{u_{i}u_{i+1}}^{l_{i}l_{i+1}}(x_{1\ldots u_{i+1}})\ne \emptyset,\quad
i=2,\ldots k-1.$$ Now use \eqref{eq:decomp} to decompose
$$\V_{u_k}^{l_k}(x_{1\ldots n})=\V_{u_k}^{l_k}(x_{1\ldots
u_k})\times \V_{(l_k)}(x_{u_k+1\ldots n}).$$ Use \eqref{eq:decomp}
again to decompose
$$\V_{u_{k-1}u_k}^{l_{k-1}l_k}(x_{1\ldots u_k})=\V_{u_{k-1}}^{l_{k-1}}(x_{1\ldots
u_{k-1}})\times \V_{(l_{k-1})u_k}^{\,\,\,\,\,\,\,\,\,\,\,\,\,\,\,l_k}(x_{u_{k-1}+1\ldots
u_k}).$$ Proceeding this way, we obtain \eqref{pieces}.
\end{proof}\\\\
Corollary \ref{good}  guarantees the existence of an alignment
$v(x_{1\ldots n})$ that can be constructed {\it piecewise}, i.e.
\begin{equation}\label{piecesI}
(v_1,\ldots,v_{k+1})\in \V(x_{1\ldots n}),
\end{equation}
 where
\begin{align*}
&v_1\in \V_{u_1}^{l_1}(x_{1\ldots u_1}),\,v_2\in
\V_{(l_1)u_2}^{\,\,\,\,\,\,\,\,l_2}(x_{u_1+1\ldots u_2}),\ldots, v_k\in
\V_{(l_{k-1})u_k}^{\,\,\,\,\,\,\,\,\,\,\,\,\,\,\,\,l_k}(x_{u_{k-1}+1 \ldots u_k}),v_{k+1}\in
\V_{(l_{k})}(x_{u_{k}+1\ldots u_n}).
\end{align*}
\subsubsection{Proper alignment}
If  the sets $\V_{(l_{i-1})u_i}^{\,\,\,\,\,\,\,\,\,\,\,\,\,\,l_i}(x_{u_{i-1}+1\ldots u_i})$,
$i=2,\ldots, k$ as well as $\V_{(l_k)}(x_{u_k+1\ldots n})$ have
a single element each, then the concatenation \eqref{piecesI} is unique.
Otherwise,  a single $v_i$ will need to be selected from 
$\V_{(l_{i-1})u_i}^{\,\,\,\,\,\,\,\,\,\,\,\,\,\,l_i}(x_{u_{i-1}+1\ldots u_i})$.
Thus, suppose that   
$(x_{u_{i-1}+1\ldots u_i})=(x_{u_{j-1}+1\ldots u_j})$, 
and $l_i=l_j$ for some $j\ne i$. Ignoring the fact that the actual probability of
such realizations may well be zero in most cases, for technical reasons we are nonetheless 
going to be general and require that 
the selection from any $\V_{(q)u+\Delta}^{\,\,\,\,\,\,\,\,\,l}(x_{u+1\ldots u+\Delta})$ for which 
$x_u$ and $x_{u+\Delta}$ are $q$ and $l$ nodes, respectively, be made independently of
$u$.
To achieve this, we impose the following (formally even more restrictive) condition on admissible
selection schemes $\{w^{ql}(x_{1\ldots m}): \mathcal{X}^m \to
\W_{(q)m}^{\,\,\,\,\,\,\,l}(x_{1\ldots m}),~m=1,\ldots,n,~q,l\in S\}$: 
\begin{equation}\label{kooskolas}
\forall q,~\forall l\in S,~\forall m\leq n,~\forall x_{1 \ldots n}\in \mathcal{X}^n:~
w_{1\ldots n}=w^{ql}(x_{1 \ldots n}) \Rightarrow  w_{1\ldots m}=w^{qw_m}(x_{1 \ldots m}).
\end{equation}
The condition \eqref{kooskolas} above simply states that the ties are broken
consistently.
\begin{definition}\label{proper}The alignment \eqref{piecesI} based on 
$l_1$,\ldots,$l_k$ nodes $x_{u_1},\ldots,x_{u_k}$ is called proper if 
 for every $i=2,\ldots,k-1$ $$v_i=w^{l_il_{i+1}}(x_{u_i+1\ldots u_{i+1}}),$$
where $\{w^{ql}(x_{1\ldots m}): \mathcal{X}^m \to
\W_{(q)m}^{\,\,\,\,\,\,\,l}(x_{1\ldots m}),~m=1,\ldots,n,~q,l\in S\}$ is some selection scheme satisfying \eqref{kooskolas}.
\end{definition}
 Clearly, there may be many such selection schemes and the following discussion 
is valid for all of them (provided the choice is fixed throughout).
One such selection scheme is based on taking maxima under the reverse lexicographic order 
on $S^m$ (for any positive integer $m$). According to this order $\prec$, for $a, b\in  S^m$, $a\prec b$ if and only
if for some $i$, $1\le i<m$, $a_i<b_i$ and $a_j=b_j$ for $j=i+1,\ldots,m$.  (Clearly, 
if neither $a\prec b$ nor $b\prec a$, then $a_j=b_j$ for $j=1,\ldots,m$, in which
case $a$ and $b$ are defined equal for this order.)  It is immediate to verify that
\eqref{kooskolas} holds for 
\begin{equation}\label{eq:selectmax}
w^{ql}(x_{1\ldots m})\stackrel{\mathrm def}{=} 
\max\nolimits_\prec\W_{(q)m}^{\,\,\,\,\,\,\,l}(x_{1\ldots m}),~1\le m\le n,~q,l\in S.
\end{equation}
For the sake of concreteness, we are going to refer to this particular selection
scheme as {\em the selection} and base all proper alignments on it. 
Also, since Definition \ref{proper} does not concern 
the initial or terminal components of the concatenated alignment  \eqref{piecesI}, 
we extend the selection (again, purely for the sake of concreteness of the presentation) 
to 
the initial and terminal components
of the concatenated alignment \eqref{piecesI}. Thus, to specify the initial component
we have $w^{\pi l}(x_{1\ldots m})\stackrel{\mathrm def}{=} 
\max\nolimits_\prec\W_m^l(x_{1\ldots m}),~1\le m\le n,$ for all $l\in S$  and for all $\pi$,
probability mass functions on $S$. To be concise, we will write $\vee W$ for the selected element of $W$ for  any $W\subset S^m$ (where $W$ generally depends on $x_{1\ldots m}$).  
In particular, the final component is then specified via $\vee\V_{(l)}(x_{1\ldots m})$.
\begin{example}\label{naidesegu} Consider an i.i.d. sequence
$X_1,X_2,\ldots$, where $X_1$ has a mixture distribution, i.e. the
density of $X_1$ is $\sum_{i=1}^K p_i f_i$. Here $p_i>0$ are the
mixture weights. Such a sequence is  an HMM with the transition
matrix satisfying $p_{ij}=p_j$ $\forall i,j$. In this case, an
observation   $x_u$ is an $l$-node if
$$\delta_u(l)\geq \delta_u(i),\quad \forall i.$$
In particular, this means that every observation is an $l$-node
for some $l\in S$.  Then  \eqref{rec} becomes
$$\delta_{u+1}(i)=\max_j(\delta_u(j))p_{i}f_i(x_{u+1})\propto p_{i}f_i(x_{u+1}),\quad \forall i$$
and
\begin{equation}\label{segu}
\delta_u(l)\geq \delta_u(i),\quad \forall i\quad \iff \quad
p_lf_l(x_u)\geq p_if_i(x_u),\quad \forall i.
\end{equation}
Thus, in a mixture-model, any observation $x_u$ is a node, more
precisely it is an $l$-node for any $l=\arg\max_j$
$(p_jf_j(x_u))$. For this model, the alignment can naturally be concatenated
point-wise: $v(x_{1\ldots n})=(v(x_1),\ldots,v(x_n)),$ where
\begin{equation}\label{yproper}
v(x)=\arg\max_ip_if_i(x).
\end{equation}
 The alignment will be proper if ties in \eqref{yproper} are broken
 consistently, which is, for example, the case when using the selection \eqref{eq:selectmax}.
\end{example}
\subsection{$r^{th}$-order nodes}
The concept of nodes is both important
and rich, but the existence of a node can also be restrictive in
the following sense:  Suppose $x_{1\ldots u}$ is such that
$\delta_u(i)>0$ for every $i$. In this case, \eqref{krit} is
equivalent to
\begin{align*}
\delta_u(l)&\geq \max_{i}\Bigl(\max_j \Bigl({p_{ij}\over
p_{lj}}\Bigr)\delta_u(i)\Bigr)
\end{align*}
and actually implies $p_{lj}>0$ for every $j\in S$. 
Hence, one cannot guarantee the existence of an $l$-node for an arbitrary
emission distribution since an ergodic $\mathbb{P}$ in general can have a zero in every row, 
violating the above positivity constraint on the $l^{th}$ row of $\mathbb{P}$.
We now generalize the notion of nodes in order to eliminate the
aforementioned positivity constraint and to still enjoy the
desirable  properties of nodes.
We need some additional definitions: For each $u\geq 1$ and $r\ge
1$, let
\begin{equation}\label{eq:pr}
p^{(r)}_{ij}(u)=\max_{q_{1\ldots r}\in
S^r}p_{iq_1}f_{q_1}(x_{u+1})
p_{q_1q_2}f_{q_2}(x_{u+2})p_{q_2q_3}\ldots
p_{q_{r-1}q_r}f_{q_r}(x_{u+r})p_{q_r j}.
\end{equation}
Also,  for each $u\geq 1$  define $p^{(0)}_{ij}(u)=p_{ij}$, and
notice
\begin{equation}\label{eq:prrecurse}
p^{(r)}_{ij}(u)=\max_{q\in S}p^{(r'-1)}_{iq}(u)f_q(x_{u+1})p^{(r-r')}_{q j}(u+1),
~\text{for all}~r'=1,2,\ldots,r.
\end{equation}
The recursion \eqref{rec} then generalizes to
\begin{equation}
  \label{eq:prrecurse2}
 \delta_{u+1}(j)=\max_{i\in S} \bigl(\delta_{u-r}(i) p^{(r)}_{ij}(u-r)\bigr)f_j(x_{u+1}),~ r<u. 
\end{equation}
For $r\ge 1$ and $u+r\le n$  define
\begin{align}\label{eq:tr}
t^{(r)}(u,j)&=\{l\in S: \forall i\in S~\delta_u(l)p^{(r-1)}_{lj}\ge \delta_u(i)p^{(r-1)}_{ij}\},  \\
\nonumber t^{(r)}(u,J)&=\{t^{(r)}(u,j):~j\in J\},\quad J\subset S.
\end{align}
 It can be verified that for $1\le q, r$, $q+r\le n-u$
\begin{equation}\label{eq:trecurse}
t^{(r+q)}(u,j)=t^{(q)}(u,t^{(r)}(u+q,j)),
\end{equation}
where $t^{(1)}(u,j)$ coincides with $t(u,j)$ \eqref{viimane}.
Thus, $l_1\in t^{(q)}(u,t^{(r)}(u+q,j))$ in \eqref{eq:trecurse}
implies the existence of  $l_2\in t^{(r)}(u+q,j)$ such that
$l_1\in t^{(q)}(u,l_2)$.  In short,
\begin{equation*}
  t^{(q)}(u,t^{(r)}(u+q,j))=\cup_{l\in t^{(r)}(u+q,j)}t^{(q)}(u,l).
\end{equation*}
Note that with this new notation, \eqref{viimane} and \eqref{eq:constrainedmle} can be rewritten respectively as follows:
\begin{align}\label{viimane2}
\V(x_1,\ldots,x_n)&=\{v\in S^n:\delta_{n}(v_n)\ge \delta_n(i)~
                    \forall i\in S, v_u\in t^{(n-u)}(u,v_n)~1\le u<n\}\\
\W_u^l(x_1,\ldots,x_n)&=\{v\in S^l_n(n):~v_u\in t^{(n-u)}(u,l)~1\le u<n\}\label{viimane2a}
\end{align}
We now  generalize the concept of the node:
\begin{definition}\label{rnode}
Let  $1\le r<n$, $u\le n-r$
and let $l\in S$.  We call $x_u$ an $l$-node of order $r$ if
\begin{equation}\label{r-krit}
\delta_u(l)p^{(r)}_{lj}(u)\geq \delta_u(i)p^{(r)}_{ij}(u),\quad
\forall i,j\in S.
\end{equation}
We also say that $x_u$ is a node of order $r$ if it is an
$l$-node of order $r$ for some $l\in S$.
\end{definition}Note that a $0^{th}$-order node is  just a node.
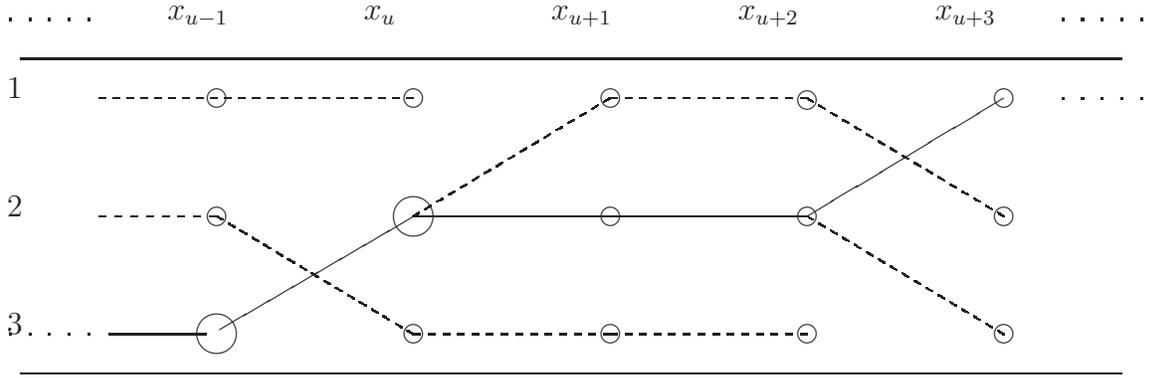
\begin{figure}
\begin{center}
\setlength{\unitlength}{0.00068in}
\begingroup\makeatletter\ifx\SetFigFont\undefined%
\gdef\SetFigFont#1#2#3#4#5{%
  \reset@font\fontsize{#1}{#2pt}%
  \fontfamily{#3}\fontseries{#4}\fontshape{#5}%
  \selectfont}%
\fi\endgroup%
{\renewcommand{\dashlinestretch}{30}
\begin{picture}(8669,2883)(0,-10)
\put(0,2112){\makebox(0,0)[lb]{\smash{{{\SetFigFont{12}{14.4}{\rmdefault}{\mddefault}{\updefault}1}}}}}
\put(0,1212){\makebox(0,0)[lb]{\smash{{{\SetFigFont{12}{14.4}{\rmdefault}{\mddefault}{\updefault}2}}}}}
\put(0,312){\makebox(0,0)[lb]{\smash{{{\SetFigFont{12}{14.4}{\rmdefault}{\mddefault}{\updefault}3}}}}}
\put(1597,1212){\ellipse{150}{150}}
\put(1597,2112){\ellipse{150}{150}}
\put(1597,312){\ellipse{300}{300}}
\put(3097,1212){\ellipse{300}{300}}
\put(3097,312){\ellipse{150}{150}}
\put(3097,2112){\ellipse{150}{150}}
\put(4597,312){\ellipse{150}{150}}
\put(4597,1212){\ellipse{150}{150}}
\put(4597,2112){\ellipse{150}{150}}
\put(6097,312){\ellipse{150}{150}}
\put(6097,1212){\ellipse{150}{150}}
\put(6097,2097){\ellipse{150}{150}}
\put(7597,312){\ellipse{150}{150}}
\put(7597,1212){\ellipse{150}{150}}
\put(7597,2112){\ellipse{150}{150}} \path(97,2412)(8497,2412)
\path(97,12)(8497,12) \path(772,312)(1522,312)
\path(3097,1212)(6097,1212) \path(6097,1212)(7597,2112)
\path(1620,350)(3097,1212)
\dashline{60.000}(1597,1212)(3097,312)(4597,312)(6097,312)
\dashline{60.000}(697,1212)(1597,1212)
\dashline{60.000}(3097,1212)(4597,2112)(6097,2112)(7597,1212)
\thicklines \dottedline{135}(8047,2712)(8647,2712)
\dottedline{135}(22,312)(622,312) \thinlines
\dashline{60.000}(6097,1212)(7597,312) \thicklines
\dottedline{135}(8047,2112)(8647,2112)
\dottedline{135}(22,2712)(622,2712) \thinlines
\dashline{60.000}(697,2112)(3097,2112)
\put(1222,2712){\makebox(0,0)[lb]{\smash{{{\SetFigFont{12}{14.4}{\rmdefault}{\mddefault}{\updefault}$x_{u-1}$}}}}}
\put(2722,2712){\makebox(0,0)[lb]{\smash{{{\SetFigFont{12}{14.4}{\rmdefault}{\mddefault}{\updefault}$x_u$}}}}}
\put(4147,2712){\makebox(0,0)[lb]{\smash{{{\SetFigFont{12}{14.4}{\rmdefault}{\mddefault}{\updefault}$x_{u+1}$}}}}}
\put(5572,2712){\makebox(0,0)[lb]{\smash{{{\SetFigFont{12}{14.4}{\rmdefault}{\mddefault}{\updefault}$x_{u+2}$}}}}}
\put(7072,2712){\makebox(0,0)[lb]{\smash{{{\SetFigFont{12}{14.4}{\rmdefault}{\mddefault}{\updefault}$x_{u+3}$}}}}}
\end{picture}   }
\end{center}
\caption{In this example, $x_u$ is a $2^d$ order 2-node, $x_{u-1}$
is a $3^d$-order  3-node. Thus, for given $x_{1\ldots n}$, the alignment includes $v_u=2$.
However, unlike in the case of ordinary nodes (of order 0), $x_{u+1}$ can now
destroy the property of $x_u$ being the (second order) node. 
}
\label{fig:viterbi_algorithm2}
\end{figure}
One immediately obtains the following properties of the
(generalized) nodes:
\begin{proposition}\label{rnode_propty}
Let $0\le r$, $1\le q$ such that $r+q \le n-u$, then
\begin{enumerate}
\item If $x_u$ is an $r^{th}$-order $l$-node, then it is also an $l$-node of order $r+q$.
\item If $x_{u+q}$ is an $r^{th}$-order $l$-node, then $x_u$ is an $(r+q)^{th}$-order $l'$-node for any $l'\in  t^{(q)}(u,l)$.
\end{enumerate}
\end{proposition}
Next, we generalize Proposition \ref{nodeprop}:
\begin{proposition}\label{nodepropr}
\begin{align}
x_u \text{ is an}~l\text{-node of order }r & \iff
l\in t^{(r+1)}(u,j)~\forall j\in S, \label{eq:node-tr}\\
\nonumber u+r<n, x_u \text{ is an}~l\text{-node of order }r
&\Rightarrow  \forall v\in \V(x_{1\ldots n}),
             \forall v^*\in \W_u^l(x_{1\ldots u})\\ \exists v'\in \W_{u~u+r+1}^{l~v_{u+r+1}}(x_{1\ldots u+r+1}):
            &v^*=v'_{1\ldots u}, (v',v_{u+r+1\ldots n})\in \V_u^l(x_{1\ldots n}),\label{eq:l-tier}\\
            & \Rightarrow \V_u^l(x_{1\ldots n})\neq \emptyset, \label{eq:node-Vr}\\
            & \Rightarrow  \V_u^l(x_{1\ldots n})=\W^l_u(x_{1\ldots u})
            \times \V_{(l)}(x_{u+1\ldots n})\label{joopr}.
\end{align}
Finding $v'_{u+1\ldots u+r}$ and $v^*\in
\W_u^l(x_{1\ldots u})$ in \eqref{eq:l-tier} for given $v\in \V(x_{1\ldots n})$  
does not require knowledge of any of $x_{u+r+1\ldots n}$. Finally, whether $x_u$
is an $l$-node of order $r$ depends on $x_1,\ldots, x_{u+r}$
only, i.e. it does not depend on any $x_i$ for $i>{u+r}$.
\end{proposition}
\begin{proof}
The final statement  follows immediately from Definition
\ref{rnode} and relations \eqref{eq:delta} and \eqref{eq:pr}.
\eqref{eq:node-tr} also follows immediately from Definition
\ref{rnode} and \eqref{eq:tr}. In order to see \eqref{eq:l-tier},
note that applying  \eqref{eq:trecurse} with $q=1$ to $l\in
t^{(r+1)}(u,v_{u+r+1})$ once gives us $\tilde v_1\in
t^{(r)}(u+1,v_{u+r+1})$.  Applying then \eqref{eq:trecurse} with
$q=1$ to $\tilde v_i\in t^{(r-i+1)}(u+i,v_{u+r+1})$ successively for
$i=2,\ldots,r$ proves the existence of the entire $\tilde v_{1\ldots r}\in S^r$ such
that $l\in \tilde t(u,v'_1)$, $\tilde v'_1\in t(u+1,\tilde v_2)$, $\ldots$,
$\tilde v_{r-1}\in t(u+r-1,\tilde v_r)$, $\tilde v_r\in t(u,v_{u+r+1})$. Thus, recalling
\eqref{viimane2a}, $\tilde v=v'_{u+1\ldots u+r}$ for some $v'\in \W_{u~u+r+1}^{l~v_{u+r+1}}(x_{1\ldots u+r+1})$.  
Since $v^*_i\in t(i,v^*_{i+1})$ for $i=1,\ldots,u-1$ ($v^*\in
\W_u^l(x_{1\ldots u})$ and \eqref{eq:constrainedmle}), and $v_i\in
t(i,v_{i+1})$ for $i=u+r+1,\ldots,n-1$ and $\delta_{n}(v_n)\ge
\delta_n(j)$ $\forall j\in S$ ($v\in \V(x_{1\ldots n})$ and
\eqref{viimane}), one gets $(v^*,v',v_{u+r+1\ldots n})\in \V_u^l(x_{1\ldots n})$.  
Evidently, $v'$ above involves no $x_i$ for $i>u+r$.
Thus, unlike in \eqref{eq:l-tie},  in addition to  setting $v_u=l$
and taking $v^*_i\in t(i,v_{i+1})$ for $i=u-1,u-2,\ldots,1$ we may
have to ``realign'' $u+1^{st}$, $\ldots$, $u+r^{th}$ components in
order  for the modified string to remain in  $\V(x_{1\ldots n})$.
Moreover, $v^*$ need not  belong to $\V(x_{1\ldots u})$.
Clearly, \eqref{eq:l-tier} implies \eqref{eq:node-Vr}. Finally,
given \eqref{eq:node-Vr}, Proposition \ref{joop0} yields
\eqref{joopr}.
\end{proof}
\begin{corollary}\label{r-prep}
For any fixed $s\in S$, Proposition~\ref{nodepropr} remains valid after replacing 
$\pi$ by $(p_{si})_{i\in S}$, wherever appropriate.  In particular, 
\begin{eqnarray*}
u+r<n, x_u \text{ is an}~l\text{-node of order }r
& \Rightarrow &\emptyset \neq \V_{(s)u}^{\,\,\,\,\,\,\,l}(x_{1\ldots n})=\\
&&=\W^{\,\,\,\,\,\,\,l}_{(s)u}(x_{1\ldots u})  \times \V_{(l)}(x_{u+1\ldots n}).
\end{eqnarray*}
\end{corollary}
\begin{corollary}\label{r-good}
Let $u_i+r_i<u_{i+1}$ $i=1,\ldots,k-1$, and $u_k+r_k<n$, and 
suppose $x_{1\ldots n}$ is such that  the observations $x_{u_i}$
are $l_i$-nodes of order $r_i$, for $i=1,\ldots,k$. Then
\begin{eqnarray}\label{r-pieces}
\nonumber  \emptyset \ne \V_{u_1u_2\cdots u_k}^{l_1l_2\cdots l_k}(x_{1\ldots n})=&&\\
=\W_{u_1}^{l_1}(x_{1\ldots u_1})\times \W_{(l_1)u_2}^{\,\,\,\,\,\,\,\,\,l_2}(x_{u_1+1\ldots u_2})\times\cdots\times
\W_{(l_{k-1})u_k}^{\,\,\,\,\,\,\,\,\,\,\,\,\,\,l_k}(x_{u_{k-1}+1\ldots u_k})\times
\V_{(l_k)}(x_{u_k+1\ldots n}).&&
\end{eqnarray}
\end{corollary}
\begin{proof}
By \eqref{eq:node-Vr}, we have
$$\V_{u_i}^{l_i}(x_{1\ldots n})\ne \emptyset \quad i=1,\ldots,k.$$ Hence,
$$\emptyset \ne \V_{u_1u_2\cdots u_k}^{l_1l_2\cdots l_k}(x_{1\ldots n}).$$
By \eqref{joopr}, 
$$\V_{u_1u_2\cdots u_k}^{l_1l_2\cdots l_k}(x_{1\ldots n})=\W_{u_1}^{l_1}(x_{1\ldots u_1})\times
\V_{(l_1)u_2\cdots u_k}^{\,\,\,\,\,\,\,\,\,l_2\cdots l_k}(x_{u_1+1\ldots n}).$$
Apply Corollary~\ref{r-prep} to get 
$$\V_{(l_1)u_2\cdots u_k}^{\,\,\,\,\,\,\,\,\,l_2\cdots l_k}(x_{u_1+1\ldots n})=\W_{(l_1)u_2}^{\,\,\,\,\,\,\,\,\,l_2}(x_{u_1+1\ldots u_2})\times
\V_{(l_2)u_3\cdots u_k}^{\,\,\,\,\,\,\,\,\,l_3\cdots l_k}(x_{u_2+1\ldots n}),$$ and repeat similarly to get
$$\V_{(l_i)u_{i+1}\cdots u_k}^{\,\,\,\,\,\,\,\,\,l_{i+1}\cdots l_k}(x_{u_i+1\ldots n})=
\W_{(l_i)u_{i+1}}^{\,\,\,\,\,\,\,\,\,l_{i+1}}(x_{u_i+1\ldots u_{i+1}})\times
\V_{(l_{i+1})u_{i+2}\cdots u_k}^{\,\,\,\,\,\,\,\,\,\,\,\,\,\,l_{i+2}\cdots l_k}(x_{u_{i+1}+1\ldots n})$$ 
for $i=2,\ldots,k-1$, yielding the desired result.
\end{proof}
\\\\
Thus, the assumptions of Proposition~\ref{nodepropr} and 
Corollary \ref{r-good} establish the existence of piecewise alignments 
\begin{equation}\label{r-piecesI}
v=(v_1,\ldots,v_{k+1})\in \V(x_{1\ldots n}),
\end{equation}
 where
$v_1\in \W_{u_1}^{l_1}(x_{1\ldots u_1}),\,v_2\in
\W_{(l_1)u_2}^{\,\,\,\,\,\,\,\,\,l_2}(x_{u_1+1\ldots u_2}),\ldots$, $v_k\in
\W_{(l_{k-1})u_k}^{\,\,\,\,\,\,\,\,\,\,\,\,\,\,l_k}(x_{u_{k-1}+1 \ldots u_k})$, $v_{k+1}\in
\V_{(l_{k})}(x_{u_{k}+1\ldots n})$.
Moreover, for every $i=1,\ldots,k$, the vectors
$w(i)\stackrel{\mathrm def}{=}(v_1,\ldots,v_i)$ satisfy
$w(i)\in \W_{u_i}^{l_i}(x_{1\ldots u_i})$ and 
$w(i)_{1\ldots u_{i-1}}=w(i-1), \quad i=2,\ldots,k$.
Since $w(i)$ does not depend on $x_{u_i+r_i+1},\ldots,
x_n$ and as long as $x_1,\ldots, x_{u_i+r_i}$ are such that $x_{u_i}$ is
a  node of order-$r_i$, an alignment $v(x_{1\ldots n})$ can always be found such 
that $v_{1\ldots u_i}=w(i)$.
\\
\begin{definition}\label{piecewise} Any alignment of the form in \eqref{r-piecesI} is called
a piecewise alignment based on  nodes
$x_{u_1},\ldots,x_{u_k}$ of orders $r_1,\ldots,r_k$, respectively. 
\end{definition}
Recall that we have previously  fixed the selection scheme $\vee$
\eqref{eq:selectmax}.
Based on this selection scheme, we will concern ourselves in \S\ref{sec:infal} with proper (Definition~\ref{proper})
piecewise (Definition~\ref{piecewise}) alignments (that are based on nodes of possibly non-zero orders) formally
defined as follows:
\begin{definition} \label{l6plik}
  \begin{eqnarray*}
    \label{eq:finitealign}
v(x_{1\ldots n})&\stackrel{\mathrm{def}}{=}&(\vee\W_{u_1}^{l_1}(x_{1\ldots u_1}),
 \vee\W_{(l_1)u_2}^{\,\,\,\,\,\,\,\,\,l_2}(x_{u_1+1\ldots u_2}),\ldots,\\
&& \vee\W_{(l_{k-1})u_k}^{\,\,\,\,\,\,\,\,\,\,\,\,\,\,\,\,l_k}(x_{u_{k-1}+1\ldots u_k}),\vee \V_{(l_k)}(x_{u_k+1\ldots n}))
\in \V_{u_1\ldots u_k}^{l_1\ldots l_k}(x_{1\ldots n}),  
  \end{eqnarray*}
for $k>0$, and 
$v(x_{1\ldots n})\stackrel{\mathrm{def}}{=}\vee \V(x_{1\ldots n})$ for $k=0$.
\end{definition}
To summarize the above, recall that by defining nodes (of various orders) we aim at 
extending alignments at infinitum, and we would like to do this for as wide a class of 
HMM's with irreducible and aperiodic hidden layers as possible. Requiring  $l$-nodes of 
order 0 immediately restricts the transition probabilities  by requiring $p_{lj}>0$ for 
$\forall j\in S$.  However, this restriction disappears with the introduction of
nodes of order $r$ for sufficiently large $r$. Indeed, suppose that 
$\forall u$ $0<u\le n$, we have $\delta_u(j)>0$ $\forall j\in S$ (which in particular implies 
$f_j(x_u)>0$ $\forall j\in S$ $\forall u$ $0<u\le n$). Then, $x_u$ being 
an $l$-node of order $r$, and irreducibility of the underlying chain, imply
$p^{(r)}_{lj}(u)>0$ $\forall j\in S$. The latter in turn implies  that  $r_{lj}>0$ 
for every $j\in S$, where $r_{lj}$ is the $lj^{\mathrm{th}}$ entry of $\mathbb{P}^r$.
Thus, having an $l$-node of order $r$ for {\bf some} $r$ does not impose 
any restriction on $\mathbb{P}$: by virtue of irreducibility and aperiodicity of $\mathbb{P}$, 
there always exists  $r_0$ such that $P$ has all of its entries positive for 
every $r\ge r_0$.\\
\subsection{Barriers}
By Corollary \ref{r-good}, $x_{u}$ being a node of order $r$ fixes the
alignment up to $u$ for any possible continuation of
$x_{1 \ldots u+r}$. However, changing the value of an
observation before $x_{u+r+1}$, say $x_1$ or $x_{u+r},$ can
prevent $x_{u}$ from being the node. Moreover, in general nothing
guarantees that for an arbitrary prefix $x'_{1\ldots w}\in \mathcal{X}^w$, 
$w+u$-th element of $(x'_{1\ldots w},x_{1\ldots u+r})$ would be a node of order $r$. 
On the other hand, a block of observations $x^b_{1\ldots k}\in \mathcal{X}^k$ ($k\ge r$) can be such that
for any $w>0$ and for any $x'_{1\ldots w}\in \mathcal{X}^w$, $w+k-r$-th element of $(x'_{1\ldots w},x^b_{1\ldots k})$
is a node of order $r$. $x^b_{1\ldots k}$ in that case will be
called {\it a barrier}. 
\begin{definition}\label{def:barrier}
A block of observations $x^b_{1\ldots k}\in \mathcal{X}^k$ ($k\ge r$) is called an $l$-barrier of order $r$ and length $k$
if for any $w>0$ and for any $x'_{1\ldots w}\in \mathcal{X}^w$, $w+k-r$-th element of $(x'_{1\ldots w},x^b_{1\ldots k})$
is an $l$-node of order $r$.
\end{definition}
\subsection{Existence of barriers }
\label{sec:barriers}
In this section, we state  the main technical result of the
paper.
For each $i\in S$, we denote by $G_i=\cap_{G\text{-closed},~P_i(G)=1}G$, 
the support of $P_i$.
\begin{definition}We define a subset $C\subset S$ to be {\it a cluster}, if, simultaneously,
$$\min_{j\in C}P_j(\cap _{i\in C}G_i\cap \{x\in \mathcal{X}:~f_i(x)>0\})>0,\quad {\rm and}\quad   P_j(\cap _{i\in C}G_i)=0\quad \forall j\not \in C.$$
\end{definition}
(Note that $C$ is well-defined that is, if the first condition is satisfied with one choice of density functions 
$f_i$, it will certainly be satisfied with any other choice of densities 
$g_i$ since $\lambda(\{x\in\mathcal{X}:~f_i\neq g_i\})=0$ for all $i\in S$.)
%
Hence, a cluster is a maximal subset of states such that the
corresponding emission distributions have a ''detectable''
intersection of their supports $G_C=\cap _{i\in C}G_i$. 
Clusters need not necessarily be disjoint and  a
cluster can consist of a single state. In this latter case such a state is not
hidden: Any emission from this state reveals it. 
If $K=2$, then, for an HMM, there is only one cluster
(otherwise the underlying Markov chain would not be hidden as all
observations reveal their states).
In many cases in practise there is only one cluster, that is $S$.
\\
A proof of Lemma~\ref{neljas} below is given
in Appendix \ref{sec:proofneljas}.
\begin{lemma}\label{neljas} Assume that for each state $l\in S$,
\begin{align}\label{lll}
& P_l\left(x:f_l(x)\max_{j}\{p_{jl}\}> \max_{i,i\ne
l}\{f_i(x)\max_{j}\{p_{ji}\}\}\right)>0.
\end{align}
Moreover, assume that there exist a cluster $C\subset S$ and a finite
integer $m<\infty$ such that the $m$-th power of the
sub-stochastic matrix $\mathbb{Q}=(p_{ij})_{i,j\in C}$ has all of its
entries  non-zero.  Then, for some integers $M$ and $r$, $M>r\geq 0$, 
there exist a set 
$B=B_1\times \cdots \times B_M \subset {\cal
X}^M$, an $M$-tuple of states $q_{1\ldots M}\in S^M$, and a state
$l\in S$, such that every vector $y\in B$ is
an $l$-barrier of order $r$ , $q_{M-r}=l$ and
\begin{equation*}
{\bf P}\Bigl((X_1,\ldots,X_M)\in {\cal
Y}\Big|Y_1=q_1,\ldots,Y_M=q_M\Bigl)>0, \quad {\bf
P}(Y_1=q_1,\ldots,Y_M=q_M)>0.
\end{equation*}
\end{lemma}
Lemma \ref{neljas} implies that $\P\bigl( (X_1,\ldots,X_M)\in
B\bigr)>0$. Also, since every element of $B$ is a 
barrier of order $r$, the ergodicity of $X$ therefore guarantees  a.e. realization of $X$ 
to contain infinitely many  $l$-barriers of order $r$. Hence, a.e.  realization of $X$ 
also has infinitely many  $l$-nodes of order $r$. 
\subsubsection{Separated barriers}
If we were to apply Corollary \ref{r-good} to a realization with infinitely
many $l$-nodes of order $r$, we would first need to ensure 
that $u_{i+1}>u_i+r$ for $i=1,2,\ldots$, where $u_i$'s are the locations of the nodes.
Obviously, one can easily select a subsequence of those nodes to enforce this condition.
For certain technical reasons related to the construction of the infinite alignment process
(\S\ref{sec:alignmentprocess}), we, however, choose first to define special barriers for which 
the above ''separation'' condition is always satisfied. Then, we give a formal statement 
(Lemma \ref{separated} below) guaranteeing
that these separated barriers  occur also infinitely often.
Let $B\subset {\cal X}^{M}$ and $M$ and $r$ be as in Lemma
\ref{neljas}. Assume that for some $l\in S$ and some $j>0$
$x_{j\ldots j+M-1}\in B$, i.e. $x_{j\ldots j+M-1}$ is an
$l$-barrier of order $r$, and $x_{j+M-r-1}$ is an
$l$-node of order $r$. However, it might happen that for some $i$, $j\le i\le
j+r,$ $x_{i\ldots i+M-1}$ is also in $B$. Then
$x_{i+M-r-1}$ is another node of order $r$. In this case,
${i+M-r-1}-(j+M-r-1)\leq r$ and Corollary \ref{r-good} can not be
used  (in the presence of ties) with these two nodes simultaneously.
\begin{definition}\label{def:sepbar}
Let $B^* \subset \mathcal{X}^N$ such that all its elements are
$l$-barriers of order $r$ for some $l\in S$ and $r\le N$.
We say that $x^b_{1\ldots N}\in B^*$ is 
separated 
(relative to $B^*$) if for any $w$, $1\le w\le r$, and for any 
$x'_{1\ldots w}\in \mathcal{X}^w$ the concatenated block 
$(x'_{1\ldots w},x^b_{1\ldots N-w})\not \in B^*$.
\end{definition}
Thus, roughly, a barrier is separated, if it is at least $r+1$ steps apart
from any preceding $B^*$ barrier.  

Suppose $B\subset {\cal X}^M$ is such that 
every $x^b_{1\ldots M}\in B$ is a barrier. The
barriers from $B$ need not in general be separated.  However, 
it can be possible to extend these barriers to 
make them separated relative to 
their own set $B^*$. For example, suppose further 
that there exists $x\in {\cal X}$ such that no 
$y\in B$ contains $x$, i.e. 
$x^b_i\ne x$ $i=1,\ldots, M$. All the elements of 
$B^*\stackrel{\mathrm{def}}{=}\{x\}\times B$ are evidently barriers,
and moreover, they are now also separated (relative to  $B^*$).

This will be used in  Appendix \S\ref{sec:separatedproof} to prove Lemma~\ref{separated} given below, 
and which states that under the assumptions of  Lemma \ref{neljas},  separated barriers are also guaranteed
to occur. In other words, a.e. realization of $X$ has infinitely many separated barriers.
\begin{lemma}\label{separated}
Suppose  the assumptions of Lemma \ref{neljas} are satisfied. 
Then, for some integers $M$ and $r$, $M>r\geq 0$, there exist a set 
$B=B_1\times \cdots \times B_M \subset {\cal
X}^M$, an $M$-tuple of states $q_{1\ldots M}\in S^M$, and a state
$l\in S$, such that every vector $y\in B$ is
a separated (relative to $B$) $l$-barrier of order $r$, $q_{M-r}=l$ and
\begin{equation*}
{\bf P}\Bigl((X_1,\ldots,X_M)\in B\Big|Y_1=q_1,\ldots,Y_M=q_M\Bigl)>0, 
\quad {\bf P}(Y_1=q_1,\ldots,Y_M=q_M)>0.
\end{equation*}
\end{lemma}
\subsubsection{Counterexamples}
The condition on $C$ in Lemma \ref{neljas} might seem technical
and even unnecessary. We next give an example of an HMM where
the cluster condition is not fulfilled and no barriers can occur.
Then, we will modify the example (Examples \ref{ex:2.4} \ref{ex:2.5}) 
to enforce the cluster condition and consequently gain barriers. 
\begin{example}\label{ex:2.3} Let $K=4$ and consider an ergodic Markov chain
with transition matrix
$$\mathbb{P}=
\left(%
\begin{array}{cccc}
  {1\over 2} & 0 & 0 &  {1\over 2} \\
  0 &  {1\over 2} &  {1\over 2} & 0 \\
   {1\over 2} & 0 &  {1\over 2} & 0 \\
  0 &  {1\over 2} & 0 &  {1\over 2}\\
\end{array}%
\right).
$$
Let the emission distributions be such that \eqref{lll} is
satisfied and $G_1=G_2$ and $G_3=G_4$ and $G_1\cap G_3=\emptyset$.
Hence, in this case there are two disjoint clusters $C_1=\{1,2\}$,
$C_2=\{3,4\}$. The matrices $\mathbb{Q}_i$  corresponding to $C_i$,
$i=1,2$ are
$$\mathbb{Q}_1=\mathbb{Q}_2=
\left(%
\begin{array}{cc}
   {1\over 2} & 0 \\
  0 &  {1\over 2} \\
\end{array}%
\right).$$ Evidently, the cluster assumption of Lemma \ref{neljas} is not
satisfied. Note also that the alignment cannot change (in one step) its state to 
the other one of the same cluster. Due to the disjoint supports, any observation  indicates
the corresponding cluster. Hence any sequence of observations can
be regarded as a sequence of  blocks emitted from
alternating clusters. However, the alignment inside each block stays
constant.\\\\
In order to see that no $x_u$ can be a node (of any order) for $1\le u<n$,
recall  $t(u,j)$~\eqref{vaike} and $t(u,j)^{(r)}$~\eqref{eq:trecurse}, 
and Proposition~\ref{nodepropr}. Specifically, note that in this setting
for any $j\in S$ $t(u,j)$ contains exactly one element, hence for any $r\ge 1$,
$t(u,j)^{(r)}$ defines a function from $S$ to $S$.  Now, it is easy to see that
depending on $x_u$, $t(u,j)$ belongs to a single cluster $C(x_u)$ for all $j\in S$. 
In particular, there are $i,j\in C'\subset S$ for some cluster $C'$ such that $i\neq j$.  Given 
this particular transition matrix, evidently $t(u,i)\neq t(u,j)$. Hence, $x_u$ cannot
be a (zero order) node (by~\eqref{eq:node-tr}). Now, starting with $u+1$ (instead of $u$),
the same argument establishes that for some $i,j\in S$, $t(u+1,i)\neq t(u+1,j)$ but are 
in one cluster.  Applying the same argument again but now to $t(u+1,i)$ and $t(u+1,j)$, we get that 
$t(u,t(u+1,i))\neq t(u,t(u+1,j))$, i.e. $t^{(2)}(u,i)\neq t^{(2)}(u,j)$. Consequently $x_u$ cannot
be a first order node~\eqref{eq:node-tr}; and so forth and so on recursively for  any $r$ such 
that $0\le r<n-u$.
\end{example}
\begin{example}\label{ex:2.4} Let us modify the HMM in Example \ref{ex:2.3} to ensure
the assumptions of Lemma \ref{neljas} hold. At first, let us
change the transition matrix. Let $0<\epsilon<{1\over 2}$ and
consider the Markov chain  $Y$ with transition matrix
$$
\left(%
\begin{array}{cccc}
  {1\over 2}-\epsilon & \epsilon & 0 &  {1\over 2} \\
  \epsilon &  {1\over 2}-\epsilon &  {1\over 2} & 0 \\
   {1\over 2} & 0 &  {1\over 2} & 0 \\
  0 &  {1\over 2} & 0 &  {1\over 2}\\
\end{array}%
\right).
$$
Let the emission distributions be as in the previous example. In
this case, the cluster $C_1$ satisfies the assumption of Lemma
\ref{neljas}. As previously, every observation indicates its
cluster. Unlike in the previous example, nodes are now possible.
To be concrete, let $\epsilon=1/4$, $f_1(x)=\exp(-x)_{x\ge 0}$, 
$f_2(x)=2\exp(-2x)_{x\ge 0}$, and $f_3(x)=\exp(x)_{x\le 0}$, 
$f_4(x)=2\exp(2x)_{x\le 0}$. It can then be verified that, for example,
if $x_1=1$, $x_2=1$ then $x_1$ is a $1$-node of order 2. Indeed, in
that case any element of 
$B=(0,+\infty)\times (\log(2), +\infty)\times (0,+\infty)$ is
a $1$-barrier of order 2.
\end{example}
\begin{example}\label{ex:2.5} Another way to modify the HMM in  Example \ref{ex:2.3}
to enforce the assumptions of Lemma \ref{neljas} is to change
the emission probabilities. Assume that the supports $G_i$,
$i=1,\ldots,4$ are such that $P_j(\cap_{i=1}^4 G_i)>0$ for all $j\in S$, and
\eqref{lll} holds. Now, the model has only one cluster that is $S=\{1,\ldots, 4\}$. 
Since the matrix $\mathbb{P}^2$ has all its entries positive, the conditions 
of Lemma \ref{neljas} are now satisfied.  A barrier can now be constructed. For example, the
following block of observations,
\begin{equation}\label{barrier}
z_1,z_2,z_3,x_1,\ldots,x_{k}, z'_1,z'_2,z'_3,
\end{equation}
 where $z_i,z'_i\in \cap_{j=1}^4 G_j$, $i=1,2,3$, $x_i\in {\cal X}$, $i=1,\ldots k$
 and $k$ is sufficiently large, is a barrier (see proof of Lemma \ref{neljas}).
  The construction of barriers in this case
 is possible because of the observations $z_i$ and $z'_i$.
 These observations  can be emitted from any state (i.e. from any distribution 
$P_i$, $i=1,\ldots,4$) and therefore  do not indicate any proper subsets of $S$.
 They play a role of a buffer allowing a change in
 the alignment from a given state to any other state (in 3 steps).
 The HMM in Example \ref{ex:2.3} does not have $r$-order nodes, because such
 buffers do not arise.  The cluster assumption in Lemma \ref{neljas}
 makes these buffers possible.
\end{example}

\def\Z{\tilde Z}
\section{Alignment process}\label{sec:alignmentprocess}
Let $x_{1\infty}=x_1,x_2,\ldots $ be a realization of $X$. If
for some $r<\infty$ $x_{1\infty}$ contains infinitely many
$r$-order nodes, then Corollary \ref{r-good} paves the way for 
defining an infinite alignment for  $x_{1\infty}$.
\subsection{Preliminaries}\label{sec:prelimproc}
Throughout this Section, we work under the assumptions of Lemma
\ref{neljas}. Let $M\ge 0$, $B\subset \mathcal{X}^M$, $r\ge 0$, 
and $l\in S$, $q=q_{1\ldots M}\in S^M$ as promised by Lemma~\ref{separated}.  
Then, for every $n\ge 1$, $$\P\Bigl((Y_n,\ldots,Y_{n+M-1})=q\Bigr)>0,\quad
\P\Bigl((X_n,\ldots,X_{n+M-1})\in B \Big|(Y_n,\ldots,Y_{n+M-1})=q\Bigr)>0$$ 
hence every $x_{n\ldots n+M-1}\in B$ is  a separated (relative to $B$) 
$l$-barrier of order $r$.\\
Denote $\P\Bigl((X_n,\ldots,X_{n+M-1})\in {\cal
Y}|(Y_n,\ldots,Y_{n+M-1})=q\Bigr)$ by $\gamma^*$. Thus, $\gamma^*>0$, and define  
\begin{equation}\label{vector}
U_n\stackrel{\mathrm{def}}{=}(X_n,\ldots,X_{n+M-1}),\quad
D_n\stackrel{\mathrm{def}}{=}(Y_n,\ldots,Y_{n+M-1}).\end{equation}
Let ${\cal F}_n\stackrel{\mathrm{def}}{=}\sigma(Y_1,\ldots,Y_n,X_1,\ldots,X_{n})$.
Define stopping times
$\nu_0,\nu_1,\nu_2,\ldots $, $R_0,R_1,R_2,\ldots,$ and
$\vartheta_0, \vartheta_1,\vartheta_2,\ldots,$ of the filtration
 $\{\mathcal{F}_{n+M-1}\}_{n=1}^{\infty}$ as follows:
\begin{align}\label{stop-nu}
&\nu_0\stackrel{\mathrm{def}}{=}\min\{n\geq 1: U_n\in B,D_n=q \},\quad
\nu_i\stackrel{\mathrm{def}}{=}\min\{n>\nu_{i-1}: U_n\in B,D_n=q\};\\
&\nonumber \\
\label{stop-u}
 &\vartheta_0\stackrel{\mathrm{def}}{=}\min\{n\geq 1: U_n\in B\},\quad
 \vartheta_i\stackrel{\mathrm{def}}{=}\min\{n>\vartheta_{i-1}:  U_n\in B\};\\
&\nonumber  \\
\label{stop-R} &R_0\stackrel{\mathrm{def}}{=}\min\{n\ge 1:\,\, D_n=q\},\quad
R_i\stackrel{\mathrm{def}}{=}\min\{n>R_{i-1}:\,\, D_n=q\}.
\end{align}
We use the convention $\min\emptyset =0$ and $\max\emptyset=-1$.
Note the difference between  $\nu$ and $R$ and $\vartheta$: The
stopping times $\vartheta$ are observable via
$X$ process alone; the stopping times $R$ are observable via the $Y$ process alone; 
the sopping times $\nu$ already  require knowledge of
the full two-dimensional process $(X,Y)$.
Clearly $\vartheta_i\leq \nu_i$, and $R_i\leq \nu_i$.\\\\
From \eqref{stop-R}, it follows that the random variables 
$R_0, (R_1-R_0), (R_2-R_1), \ldots $ are independent and $(R_1-R_0),
(R_2-R_1), \ldots $ are identically distributed.
 The same evidently holds for the random variables
 $\nu_0, (\nu_1-\nu_0), (\nu_2-\nu_1),\ldots $.
\begin{proposition}\label{kesk} For any  initial distribution $\pi'$ 
of $Y,$ we have $E_{\pi'}\nu_0<\infty$ and $E_{\pi'}(\nu_1-\nu_0)<\infty$.

\end{proposition}
Proposition \ref{kesk} above is an intuitive result; a proof is provided in Appendix \S\ref{subsec:keskproof}.
To every $\nu_i,\,\,\, i=0,1,\ldots$ there corresponds an 
$l$-barrier of order $r$. This barrier extends over the interval
$[\nu_i,\nu_i+M-1]$. By Definition \ref{def:barrier},
$X_{\tau_i}$ is an $l$-node of order $r$, where
\begin{equation}\label{renewal}
\tau_i\stackrel{\mathrm{def}}{=}\nu_i+(M-1)-r,\quad i=0,1,\ldots
\end{equation}
Define 
\begin{eqnarray}
  \label{eq:Ti}
T_0\stackrel{\mathrm{def}}{=}\tau_0, & T_i\stackrel{\mathrm{def}}{=}\tau_i-\tau_{i-1}=\nu_i-\nu_{i-1},~
i=1,2,\ldots.  
\end{eqnarray}
Proposition \ref{kesk} says that $E_{\pi'}T_1<\infty$, $E_{\pi'}T_0<\infty$, where $\pi'$ is any
initial distribution of $Y$.
Thus, $T_i$, $i=0,1,\ldots$ correspond to a delayed renewal process (for a general reference see, for example, 
\cite{grimmet-stirzaker}).\\\\
Let $u_0,u_1,u_2,\ldots$ be the locations of the order $r$ $l$-nodes corresponding
to the stopping times $\vartheta$, i.e.
\begin{equation}\label{uh-uu}
u_i=\vartheta_{i}+(M-1)-r,\quad i=0,1,2,\ldots.
\end{equation}
Clearly, every $\tau_i$ is also $u_j$ for some $j\ge i$. Also, since the
barriers are separated,  $u_i>u_{i-1}+r$.
\subsection{Alignments}\label{sec:infal}
We next specify the alignments
$v(x_{1\ldots n})\in \V(x_{1\ldots n})$ and
define $v(x_{1\ldots \infty})$ as well as the  measures $\hat P_l^n$ corresponding to $v(x_{1\ldots n})$.
 
Let $k(x_{1\ldots n})$ be the number of $x_{u_0},~x_{u_1},\ldots,$ $x_{u_{k(x_{1\ldots n})-1}}$, {\em all} $l$ nodes of order $r$
such that $u_i>u_{i-1}+r$ for $i=1,\ldots,k(x_{1\ldots n})-1$, and $u_{k(x_{1\ldots n})-1}+r<n$. Recall (Definition~\ref{l6plik})
that based on the selection $\vee$ \eqref{eq:selectmax}, we single out the following proper piecewise alignment:

  \begin{eqnarray*}
v(x_{1\ldots n})&=&(\vee\W_{u_0}^l(x_{1\ldots u_0}),
 \vee\W_{(l)u_1}^{\,\,\,\,\,\,\,l}(x_{u_0+1\ldots u_1}),\ldots,\\
&& \vee\W_{(l)u_{k-1}}^{\,\,\,\,\,\,\,l}(x_{u_{k-2}+1\ldots u_{k-1}}),\vee \V_{(l)}(x_{u_{k-1}+1\ldots n}))
\in \V_{u_0\ldots u_{k-1}}^{l\ldots l}(x_{1\ldots n}),  
  \end{eqnarray*}
for $k=k(x_{1\ldots n})>0$, and 
$v(x_{1\ldots n})=\vee \V(x_{1\ldots n})$ for $k=0$.
Corollary \ref{r-good} makes it possible to define the
{\it infinite proper piecewise alignment} that will be consistent with
Definition \ref{l6plik} (in the sense of \eqref{eq:limit} below). Namely,
we state
\begin{definition}\label{l6pmatu}
\begin{eqnarray*}\label{tykid}
v(x_{1\ldots \infty})&\stackrel{\mathrm{def}}{=}&(\vee\W_{u_0}^l(x_{1\ldots u_0}),
 \vee\W_{(l)u_1}^{\,\,\,\,\,\,\,l}(x_{u_0+1\ldots u_1}),\ldots,)
\end{eqnarray*}
for all $x_{1\infty}$ that contain infinitely many $x_{u_0}$, $x_{u_1}$, \ldots,
$l$-nodes of order $r$, which is the case a.s. (Lemmas \ref{neljas} and \ref{separated}).
(For all other realizations, let us adopt 
$v(x_{1\ldots \infty})\stackrel{\mathrm{def}}{=}(\vee\W_{u_0}^l(x_{1\ldots u_0}),
 \vee\W_{(l)u_1}^{\,\,\,\,\,\,\,l}(x_{u_0+1\ldots u_1}),\ldots,
\vee\W_{(l)u_{k-1}}^{\,\,\,\,\,\,\,l}(x_{u_{k-2}+1\ldots u_{k-1}}),1,1,\ldots)$, where $k$ is
the total number of $l$ nodes of order $r$ in the given realization.)


\end{definition}
 Note that for every $x_{u_i}$ observed in $(x_1,\ldots,x_n)$
\begin{equation}\label{eq:limit} 
v(x_1^{\infty})_{1\cdots u_i}=v(x_1,\ldots,x_n)_{1\cdots u_i}.
\end{equation}
 Let us now formally define the empirical measures $\hat P_l^n$ which are central
to this theory:
 \begin{definition}\label{emp-measures}
Let $=V'_{1\ldots n}=v(X_1,\ldots,X_n)$  (where $v$ is as in Definition
\ref{l6plik}). For each state $l\in S$ that appears in $V'_1$, $V'_2$,
\ldots $V'_n$ define the {\it empirical $l$-measure} 
\begin{equation*}
\hat P_l^n(A, X_{1\ldots n})\stackrel{\mathrm{def}}{=}{\sum_{i=1}^n I_{A\times l}(X_i,V'_i)\over \sum_{i=1}^n
I_l(V'_i)},\quad A\in {\cal B}.
\end{equation*}
For other $l\in S$ (i.e. such that $l\ne V'_i$ for $i=1,\ldots, n$),
define $\hat P_l^n$ to be an arbitrarily chosen (probability) measure $P^*$.
\end{definition}
The infinite alignment allows us to define the {\it alignment
process}:
\begin{definition}
The encoded process $V\stackrel{\mathrm{def}}{=}v(X)$ will be
called the {\em alignment process}.
\end{definition}
(Of course, the definition of $V$ above is sensible only because $X$ has infinitely many $u_i$-s a.s..)
We shall also consider the 2-dimensional process 
$$Z\stackrel{\mathrm{def}}{=}(X,V).$$
Using $Z$, we define a related quantity $Q_l^n$ as follows: Let
$V_1,\ldots, V_n$ be the first $n$ elements of the alignment
process. In general
$$v(x_1^{\infty})_{1\ldots n}\ne
v(x_1,\ldots,x_n),$$ hence $V'_i$ need not equal $V_i$. For every $l\in S$, we
define
\begin{equation*}\label{emp2}
\hat Q_l^n(A, Z_{1\ldots n})\stackrel{\mathrm{def}}{=}{\sum_{i=1}^n I_{A\times l}(X_i,V_i)\over \sum_{i=1}^n
I_l(V_i)}= {\sum_{i=1}^n I_{A\times l}(Z_i)\over \sum_{i=1}^n
I_l(V)}, \quad A\in {\cal B} .
\end{equation*}
(As in Definition \ref{emp-measures},  if $l\ne V_i$, $i=1,\ldots, n$, then $\hat Q_l^n\stackrel{\mathrm{def}}{=}P^*$.)
\subsection{Regenerativity}
To prove our main theorem, we use the fact that $Z$ is a
regenerative process (for a general reference, see, for example, \cite{bremaud}):
\begin{proposition}\label{reg}
The processes $V$, $X$, and  $Z$ are regenerative with respect to the sequence
of stopping times $\tau_i$.
\end{proposition}
A proof is given in Appendix \S\ref{subsec:reg}.
\\\\
Recall  (\S\ref{sec:prelimproc}) $B$, 
the set of separated $l$-barriers of order $r$, and the corresponding
state sequence $q$.
Let
$$P^r_{q_i}\propto P_{q_i}I_{B_i},\quad i=1,\ldots,M.$$
Thus, $P^r_{q_i}$ is the measure $P^r_{q_i}$ conditioned on
$B_i$, $i$-th component of $B$.  Recall also that $q_{M-r}=l$.\\\\
Define new processes
\begin{align}\label{def:new-M}
&Y^r\stackrel{\mathrm{def}}{=}(Y^r_i)_{i=1}^{\infty},\quad\text{where}\quad
Y^r_1=q_{M-r+1},\ldots, Y^r_{r}=q_{M},\quad\text{and}\quad Y^r_{r+1},Y^r_{r+2},\ldots \\
&\nonumber \text{ is an }~S~\text{-valued Markov chain with transition probability matrix } \mathbb{P}
\text{ and initial}\\ 
& \nonumber\text{distribution }(p_{q_N\,j})_{j\in S};\\
&\nonumber  X^r\stackrel{\mathrm{def}}{=}(X^r_i)_{i=1}^{\infty}\quad \text{is a modified HMM
with}~Y^r~\text{as its underlying Markov chain and}\\
&\nonumber P_{Y^r_i}~\text{if}~i>r,~\text{and}~P^r_{q_{N-r+i}}~\text{if}~1\le i\le r,~\text{as its emission  distributions;}\\
&\label{def:new-V}
V^r\stackrel{\mathrm{def}}{=}(V^r_i)_{i=1}^{\infty}\stackrel{\mathrm{def}}{=}v(X^r),~\text{where  } 
v~\text{is as in Definition \ref{l6pmatu};}\\
&  Z^r\stackrel{\mathrm{def}}{=}(X^r,V^r)\label{def:new-Z}.
\end{align}
Note that the process $X^r$ is not exactly an HMM as defined in
Definition \ref{def:HMM} because the first $r$-emissions are generated from
distributions that differ from the distributions of the subsequent emissions. However, 
conditioned on the underlying Markov Chain $Y^r$, all emissions are still independent. Also
note that in the definition of $V^r$, the alignment is still
based on the original HMM $X$, i.e. the definition of $v(x_1,\ldots,x_n)$
relies on the  distributions $P_{q_1}$, $P_{q_2}$,\ldots, $P_{q_n}$ (given
$Y_{1\ldots n}=q_{1\ldots n}$).

Finally, note that for $r=0$, the process $Y^0$ is essentially our original
Markov chain except for the initial distribution that is now $(p_{l\,j})_{j\in S}$ instead of $\pi$.
Similarly, $X^0$ is the HMM in the sense of Definition
\ref{def:HMM} with $Y^0$ as its underlying Markov chain.
Therefore, $Z^0$ is the  process $Z$ with $(p_{l\,j})_{j\in S}$ as the initial distribution of 
its $Y$-component.\\\\
Finally we define  analogues of $\nu_0$ and $\tau_0$:
\begin{align}\nonumber
&\nu_0^{r}\stackrel{\mathrm{def}}{=}\min\Big \{n\geq 1: (Y^r_{n},\ldots,
Y^r_{n+M-1})=q,\quad (X^r_{n},\ldots, X^r_{n+M-1})\in B\bigl)\Big\}\\
&\label{r-tau}\tau_0^r\stackrel{\mathrm{def}}{=}\nu_0^r +(M-1)-r.
\end{align}
Note that the  random variable $\tau_0^r$ has the same law as $T_i$~\eqref{eq:Ti}, 
$i\geq 1$.
%
Since the barriers from $B$ are separated  
(Definition \ref{def:sepbar}, Lemma~\ref{separated}), then
$\nu_0^r>r$. This means that the laws of $\nu_0^r$,  $\tau_0^r$,
 $\nu_0+r$, and $\tau_0+r$ would all be equal if the initial distribution of $Y$ were $(p_{q_M\,l})_{l\in S}$. 
Recalling that any initial distribution 
$\pi'$ of  $Y$  yields $E_{\pi'}(\nu_0)<\infty$ (Proposition \ref{kesk}), we obtain
\begin{equation}\label{lopp0}
ET_1=E\tau_0^r=E_{q_M}(\nu_0+(M-1)-r+r)<\infty.
\end{equation}
The above observations will allow us to prove (see Appendix \S\ref{subsec:proofofsaddam}) 
the following theorem which is {\em the main result of the paper}:
\begin{theorem}\label{saddam}
If $X$ satisfies the assumptions of Lemma \ref{neljas}, then there
exist probability measures $Q_l$, $l\in S$, such that
\begin{equation*} \hat P_l^n\Rightarrow Q_l,\quad \text{a.s.},\quad
\hat Q_l^n\Rightarrow Q_l,\quad \text{a.s.}
\end{equation*}
 and for each $A\in {\cal B}$,
\begin{equation}\label{jaotus}
Q_l(A)={\sum_{i=1}^{\infty}{\bf P}(Z^r_i \in A\times l,
\tau^r_0\geq i)\over \sum_{i=1}^{\infty} {\bf P}(V^r_i=l,\tau^r_0\geq i)}.
\end{equation}
where  $V^r$, $Z^r$, and $\tau^r_0$ are defined in \eqref{def:new-V}, \eqref{def:new-Z}, and
\eqref{r-tau}, respectively.
\end{theorem}
\begin{corollary}
Suppose $X$ satisfies the assumptions of Lemma \ref{neljas} with
$r=0$.  Then, for each $l\in S$ (\ref{jaotus}) takes form
\begin{equation}\label{0-jaotus}
Q_{l}(A)={\sum_{j=1}^{\infty}{\bf P}_l(Z_j \in A\times i,
\tau_0\geq j)\over \sum_{j=1}^{\infty} {\bf P}_l(V_j=i,\tau_0\geq
j)},
\end{equation}
where $\mathbf{P}_l$ corresponds to the $Y$ process initialized with 
$(p_{l\,j})_{j\in S}$  instead of $\pi$. \end{corollary}

\section{Appendix}
\subsection{Proof of Lemma \ref{neljas}}\label{sec:proofneljas}
\begin{proof}The proof below is a rather direct construction which is, however, technically
involved. In order to facilitate the exposition of this proof, we have divided it into 18 
short parts as outlined below:
\begin{enumerate}[I.]
\item - \S\ref{subsec:I} - Maximal probability transitions $p^*_i$ and maximal likelihood ratio $A$.
\item Construction of
\begin{enumerate}[]
\item (\S\ref{subsec:Xl}) auxiliary subsets $\mathcal{X}_l\in \mathcal{X}$, \eqref{tarn};
\item (\S\ref{subsec:Z}) a special set $Z\subset\mathcal{X}$, \eqref{eq:defineZ1}, \eqref{eq:defineZ2};
\item (\S\ref{subsec:abs}) auxiliary sequences $\mathbf{s}$ \eqref{tsykkel}, $\mathbf{a}$ \eqref{side}, and $\mathbf{b}$ 
       \eqref{eq:defineb} of  states in $S$;
\item (\S\ref{subsec:k})  $k$, the number of $\mathbf{s}$ cycles inside the $s$-path;
\item (\S\ref{subsec:spath}) the $s$-path \eqref{path}, a prototype of the required sequence $q_{1\ldots M}$;
\item (\S\ref{subsec:barrier}) the required barrier \eqref{blokk}.
\end{enumerate}
\item Proving the barrier construction \eqref{blokk}:
\begin{enumerate}[]
\item (\S\ref{subsec:alpha-eta}) $\alpha, \beta, \gamma, \eta$-notation for commonly used maximal partial likelihoods;
\item (\S\ref{subsec:betabound}) a bound \eqref{betas}  on $\beta$;
\item (\S\ref{subsec:llbound}) bounds \eqref{krim}, \eqref{tsiv}, \eqref{sugar}, and \eqref{p} on common likelihood ratios;
\item (\S\ref{subsec:gammabound})$\gamma_j\le const\times  \gamma_1$;
\item (\S\ref{subsec:gammaboundmore}) Further bounds \eqref{liim}, \eqref{pulk} on likelihoods;
\item (\S\ref{subsec:etabound}) $\eta_j\le const\times  \eta_1$;
\item (\S\ref{subsec:etaone}) a special representation of $\eta_1$ \eqref{on};
\item (\S\ref{subsec:deltaone}) an implication of  \eqref{kohus} and \eqref{on} for $\delta_1(x_{lL})$;
\item \hspace{1cm}  $x_{kL}$ is a $(kL+m+P)$-order 1-node:
 \begin{enumerate}[]
      \item (\S\ref{subsec:onenode}) proof
      \item (\S\ref{subsec:auxproof35}) proof of an auxiliary inequality \eqref{lill}.
  \end{enumerate}
\end{enumerate}
\item (\S\ref{subsec:M}) Completion of the $s$-path to $q_{1\ldots M}$ and conclusion.
\end{enumerate}

\subsubsection{Maximal probability transitions $p^*_i$ and maximal likelihood ratio $A$.}
\label{subsec:I}
Let
\begin{eqnarray}
  \label{eq:maxtransandratio}
p^*_i=\max_{j\in S}\{p_{j\,i}\},\,i\in S, &\text{and}&  A=\max_{i\in S}\max_{j\in S}\Big\{{p^*_i\over p_{ji}}:p_{ji}>0\Big\}
\end{eqnarray}
be defined as above.

\subsubsection{$\mathcal{X}_l\subset \mathcal{X}$.}
\label{subsec:Xl}
It follows from the assumption \eqref{lll} and finiteness of $S$ that there exists an 
$\epsilon>0$ such that for all $l\in S$
\begin{equation}\label{tarn}
P_l({\cal X}_l)>0,\quad\text{where}~{\cal X}_l\stackrel{\mathrm{def}}{=} \Big\{x\in\mathcal{X}: \max_{i,i\ne
l}\{p^*_i f_i(x)\}< (1-\epsilon)  p^*_lf_l(x)\Bigr\}.
\end{equation}
(Note that $p^*_l>0$ for all $l\in S$ by irreducibility of $Y$.) Also note that  
the sets  ${\cal X}_l,l\in S$ are disjoint and have positive reference measure 
$\lambda({\cal X}_l)>0$.
\subsubsection{$\mathcal{Z}\subset\mathcal{X}$ and $\delta-K$ bounds on cluster densities $f_i$, $i\in C$}
\label{subsec:Z}
Let $C$ be a cluster as in the assumptions of the Lemma with the corresponding
sub-stochastic matrix $\mathbb{Q}$. The existence of $C$ implies the
existence of a set $\hat{{\cal Z}}\subset G_C(=\cap_{i\in C }G_i)$ and $\delta>0$,
 such that $\lambda (\hat{{\cal Z}})>0$, and $\forall z\in
{\hat{\cal Z}}$, the following statements hold:
\begin{enumerate}[(i)]
    \item $\min_{i\in C} f_i(z)>\delta$;
    \item $\max_{j\not \in C} f_j(z)=0$.
\end{enumerate}
Indeed, if no $\hat{{\cal Z}}$ and $\delta>0$ existed with property (i), 
we would have \\
$\lambda \left( \cap_{i\in C }(G_i \cap\{z\in\mathcal{X}:~f_i(z)>0\}) \right)=0$,
contradicting the first defining property of cluster: 
$P_j(G_C\cap _{i\in C}\{x\in \mathcal{X}:~f_i(x)>0\})>0$ (with any $j\in C$).
Now, if $\hat{{\cal Z}}$ did not satisfy (ii), 
we would remove from it $\hat{{\cal Z}}\cap \cup_{i\not\in C}\{z\in\mathcal{X}:~f_i(z)>0\}$
as this would not reduce its $\lambda$ measure. This is due to the second condition in the
definition of cluster which implies $\lambda(G_C \cap \{z\in\mathcal{X}:~f_i(z)>0\})=0$ for all $i\not\in C$.

Evidently, $K>0$ can be chosen sufficiently large to make
$\lambda (\{z\in\mathcal{X}:~f_i(z)\ge K\})$ arbitrarily small, and in particular, to
guarantee that 
$\lambda (\{z\in\mathcal{X}:~f_i(z)\ge K\})<\frac{\lambda (\hat{{\cal Z}})}{|C|}$.
Clearly then, redefining 
$\hat{{\cal Z}}\stackrel{\mathrm{def}}{=}
\hat{{\cal Z}}\cap\{z\in\mathcal{X}:~f_i(z)<K,~i\in C\}$ preserves 
$\lambda (\hat{{\cal Z}})>0$.  
Next, consider
\begin{equation}\label{moot}
\lambda(\hat{{\cal Z}}\backslash (\cup_{l\in S}{\cal X}_l)).
\end{equation}
If \eqref{moot} is positive, then define
\begin{equation}
  \label{eq:defineZ1}
{\cal Z}\stackrel{\mathrm{def}}{=}\hat{{\cal Z}}\backslash (\cup_{l\in S}{\cal X}_l).
\end{equation}
If \eqref{moot} is zero, then there must be  $s\in C$ such that
$$\lambda(\hat{{\cal Z}}\cap {\cal X}_s)>0$$
and in this case, let
\begin{equation}
  \label{eq:defineZ2}
{\cal Z}\stackrel{\mathrm{def}}{=}\hat{{\cal Z}}\cap {\cal X}_s.
\end{equation}
Such $s\in S$ must clearly exist since $\lambda(\hat{{\cal Z}})>0$ but
$\lambda(\hat{{\cal Z}}\backslash (\cup_{l\in S}{\cal X}_l))=0$. 
To see that $s$ must necessarily be in the cluster $C$, 
note  
 $\forall s\not\in C$, $f_s(z)=0$ $\forall z\in{\hat{\cal Z}}$, which implies 
${\hat{\cal Z}}\cap {\cal X}_s =\emptyset$.
\subsubsection{Sequences $\mathbf{s}$,  $\mathbf{a}$, and $\mathbf{b}$ of  states in $S$}
\label{subsec:abs} Let us define an auxiliary sequence of states $q_1$, $q_2$, and so on, as follows:
If \eqref{moot} is zero, that is, if  
${\cal Z}=\hat{{\cal Z}}\cap {\cal X}_s$ for some $s\in C$, then define $q_1=s$,
otherwise let $q_1$ be an arbitrary state in $C$.
Let $q_2$ be a state with maximal probability of transition to $q_1$, i.e.: 
$p_{q_2\,q_1}=p^*_{q_1}$ (see \eqref{eq:maxtransandratio} for the $p^*$ notation). 
Suppose $q_2\ne q_1$. Then find $q_3$ with $p_{q_3\,q_2}=p^*_{q_2}$.
If $q_3\not\in \{q_1,q_2\}$, find $q_4:~p_{q_4\,q_3}=p^*_{q_3}$, and so on. 
Let $U$ be the first index such that $q_U\in \{q_1,\ldots,q_{U-1}\}$, 
that is, $q_U=q_T$ for some $T<U$.  
This means that there exists a sequence of states  $\{q_T,\ldots, q_U\}$ such that
 \begin{itemize}
    \item $q_T=q_{U}$
    \item $q_{T+i}=\arg\max_{j} p_{jq_{T+i-1}},\quad i=1, \ldots, U-T.$
\end{itemize}
To simplify the notation and without loss of generality, assume $q_U=1$. Reorder and rename the
states as follows:
\begin{align}\label{tsykkel}
&s_1:=q_{U-1},\, s_2:=q_{U-2},\ldots,
s_i:=q_{U-i},\ldots,s_L:=q_T=1\quad i=1,\ldots,L\stackrel{\mathrm{def}}{=}U-T,\\
&\label{side} a_1:=q_{T-1},\,a_2:=q_{T-2},\ldots,a_P:=q_1,\quad
P\stackrel{\mathrm{def}}{=}T-1.
\end{align}
Hence,
$$\{q_1,\ldots,q_{T-1},q_T,q_{T+1},\ldots,q_{U-1},q_U\}=\{a_P,\ldots,a_1,1,s_{L-1},\ldots,s_1,1\}.$$
Note that if $T=1$, then $P=0$ and
$\{q_1,\ldots,\ldots,q_{U-1},q_U\}=\{1,s_{L-1},\ldots,s_1,1\}.$
We have thus introduced special sequences $\mathbf{a}=(a_1,a_2,\ldots,a_P)$ and
$\mathbf{s}=(s_1,s_2,\ldots,s_{L-1},1)$.
Clearly,
\begin{align}\label{klass}
p_{s_{i-1}\,s_i}&=p^*_{s_i},\quad i=2,\ldots, L,\quad p^*_{s_1}=p_{1\,s_1}\\
p_{a_{i-1}\,a_i}&=p^*_{a_i},\quad i=2,\ldots, P,\quad p^*_{a_1}=s_L=1.\nonumber
\end{align}
Next, we are going to exhibit $\mathbf{b}=(b_1,\ldots b_R)$, another auxiliary sequence for some $R\ge 1$, 
characterized as follows:
\begin{enumerate}[(i)]\label{eq:defineb}
    \item $b_R=1$;
    \item there exists $b_0\in C$ such that
    $p_{b_0\,b_1}p_{b_1\,b_2}\cdots
    p_{b_{R-1}\,b_R}>0$
    \item if $R>1$, then $b_{i-1}\ne b_i$ for every $i=1,\ldots,R$.
\end{enumerate}
Thus, the path $b_1,\ldots b_{R-1},b_R$ connects cluster $C$ to state 1 
in $R$ steps. Let us also require that $R$ be minimal such. Clearly 
such $\mathbf{b}$ and $b_0$ do exist due to irreducibility of $Y$. Specifically,
for any two states in $S$ in general, and for any state in $C$ and state $1$ in particular,
there exists a (finite) transition path of a positive probability. Note also that minimality 
of $R$ guarantees (iii) (in the special case of $R=1$ it may happen that $b_1=1\in S$ 
and $p_{1\,1}>0$, in which case $b_0$ can be taken to be also $1$). 
\subsubsection{$k$, the number of $\mathbf{s}$ cycles inside the $s$-path}
\label{subsec:k}
Let $\mathbb{Q}^{m}$ be the $m$-th power of  the
sub-stochastic matrix $\mathbb{Q}=(p_{ij})_{i,j\in C}$; let $q_{ij}$ be the entries of $\mathbb{Q}^{m}$. 
By the assumption, $q_{ij}>0$ for every $i,j\in C$. This means that for every $i,j\in C$, there exists a
path from $i$ to $j$ of length $m$  that has a positive probability.
Let $q^*_{ij}$ be the probability of a maximum probability path
from $i$ to $j$. In other words, for every $i,j\in C$, there exist
states $w_1,\ldots,w_{m-1}\in C$ such that
\begin{equation}\label{maxprobpath}
p_{iw_1}p_{w_1w_2}\cdots
p_{w_{m-1}w_{m-1}}p_{w_{m-1}j}=q^*_{ij}>0.\end{equation} Denote
by $q$ 
\begin{equation}\label{eq:qpositive}
\min_{i,j\in C}q^*_{i\,j}>0.
\end{equation}
Next, choose $k$ sufficiently large for the following to hold:
\begin{equation}
  \label{eq:largek}
 (1-\epsilon)^{k-1}<q^2({\delta\over K})^{2m}A^{-R}, 
\end{equation}
where $A$ and $\epsilon$  are as in \eqref{eq:maxtransandratio} and \eqref{tarn}, respectively, 
and $\delta$ and $K$ are introduced in \S\ref{subsec:Z}.
\subsubsection{The $s$-path}
\label{subsec:spath}
We now fix the state-sequence
\begin{equation}\label{path}
b_0,b_1,\ldots, b_R,s_1,s_2,\ldots, s_{2Lk},a_1,\ldots,a_P,
\end{equation}
where $s_{Lj+i}=s_i$, $j=1,\ldots,2k-1$, $i=1,\ldots,L$,
(and in particular $s_{Lj}=1$, $j=1,\ldots,2k$). The sequence \eqref{path} will be called
the {\em $s$-path}. The $s$-path is a concatenation  of $2k$ $\mathbf{s}$ cycles
$s_1,\ldots,s_L$, the beginning  and the end of which are connected to the cluster $C$ 
via positive probability paths $\mathbf{b}$ and $\mathbf{a}$, respectively 
(recall that $a_P=q_1\in C$ and $b_R=1$ by construction). Additionally, the
$b_R,s_1,s_2,\ldots, s_{2Lk},a_1,\ldots,a_P$-segment of the $s$-path
\eqref{path} has the important property \eqref{klass}, i.e. every consecutive transition along 
this segment occurs with the maximal transition probability given
its destination state. 
(However, $\mathbf{b}$, the beginning of the $s$-path, need  not satisfy
this property.) The $s$-path comes very close to being the sequence $q_{1\ldots M}$ required by the Lemma 
and will be completed to $q_{1\ldots M}$ in \S\ref{subsec:M}. 
In fact, the  idea of the Lemma and its proof is to exhibit (a cylinder subset of) observations such that once
emitted along the $s$-path, these observations would trap the Viterbi backtracking so that the latter winds up 
on the $s$-path. 
That will guarantee that an observation corresponding to the beginning of the $s$-path, is, as desired, a node.
\subsubsection{The barrier}
\label{subsec:barrier}
Consider the following sequence of observations
\begin{equation}\label{blokk}
z_0,z_1,\ldots,z_{m},x'_1,\ldots,x'_{R-1},x_0,x_1,\ldots,x_{2Lk},x^{''}_1,\ldots,x^{''}_P,z'_1,\ldots,z'_{m},
\end{equation}
 where $z_0,z_i,z'_i\in {\cal Z},~i=1,\ldots,m;$ $x'_i\in  {\cal X}_{b_i},~i=1,\ldots,R-1;$ and  
$$x_0\in {\cal X}_1,\quad x_{i+Lj}\in {\cal X}_{s_i},~j=1,\ldots,2k-1,~i=1,\ldots,L;
\quad x^{''}_i\in {\cal X}_{a_i},i=1,\ldots,P.$$
From this point on throughout \S\ref{subsec:onenode}, we shall be proving that $x_{Lk}$ is a  
1-node of order $(kL+m+P)$, and, therefore, that \eqref{blokk} is a 1-barrier of order $(kL+m+P)$.\\\\
Let $u\ge 2Lk+2m+1+P+R$ and let $x_1,\ldots,x_u$ be any sequence of observations terminating in
the $2Lk+2m+1+P+R$ observation long sequence of \eqref{blokk}. 
\subsubsection{$\alpha$, $\beta$, $\gamma$, $\eta$}
\label{subsec:alpha-eta}
Recall the definition of the scores $\delta_u(i)$ \eqref{eq:delta} and the maximum partial
likelihoods $p^{(r)}_{i\,j}(u)$ \eqref{eq:pr}. 
Now, we need to abbreviate some of the notation as follows. Namely, we 
denote by $\delta_i(x_l)$ (resp. $\delta_i(z_l)$) the scores corresponding to the observation $x_l$
(resp. $z_l$). Similarly, we denote by $p^{(r)}_{i\,j}(x_l)$ (resp.
$p^{(r)}_{i\,j}(z_l)$) the maximum partial likelihoods 
corresponding to the observation $x_l$ (resp. $z_l$). Formally, for any $i,j\in S$
and appropriate $r\ge 0$, the abbreviated notation is as follows:
\begin{align}\label{eq:shortscores}
\delta_i(x_l)&:=\delta_{u-P-m-2kL+l}(i),~p_{ij}^{(r)}(x_l):=p_{ij}^{(r)}(u-P-m-2kL+l),~0\le l\le 2kL;\\
\nonumber
p_{ij}^{(r)}(x'_l)&:=p_{ij}^{(r)}(u-P-m-2kL-R+l),~1\le l\le R-1;\\
\nonumber
\delta_i(z_l)&:=\delta_{u-2Lk-2m-P-R+l}(i),~p_{ij}^{(r)}(z_l):=p_{ij}^{(r)}(u- 2Lk-2m-P-R+l),0\le l\le m;\\
\nonumber
\delta_i(z'_l)&:=\delta_{u-m+l}(i),~p_{ij}^{(r)}(z'_l):=p_{ij}^{(r)}(u-m+l),~1\le l\le m. \\
\end{align}
Also, we will be frequently using the scores corresponding to $z_0$,
$x'_1$, $x_{Lk}$, and $x_{2Lk}$, hence the following further abbreviations:
$$\alpha_i:=\delta_i(z_0),\quad \beta_i:=\delta_i(z_m),\quad
\gamma_i:=\delta_i(x_0),\quad \eta_i:=\delta_i(x_{Lk}).$$
Note  that $\forall j\not \in C$, $f(z_0)=f_j(z'_l)=f_j(z_l)=0$, $l=1,\ldots,m$
by construction of ${\cal Z}$ (\S\ref{subsec:Z}). Hence,
$\alpha_j=\beta_j=0$ $\forall j\not\in C$, and a more general implication is
that for every $j\in S$
\begin{align}\label{kreeka1}
\beta_j&=\max_{i\in
C}\alpha_ip_{ij}^{(m-1)}(z_0)f_j(z_m)=\alpha_{i_{\beta}(j)}p_{i_{\beta}(j)\,j}^{(m-1)}(z_0)f_j(z_m)~
\text{for some }i_{\beta}(j)\in C;\\
\label{kreeka2} \gamma_j&=\max_{i\in C}
\beta_ip_{ij}^{(R-1)}(z_m)f_j(x_0)=\beta_{i_{\gamma}(j)}p_{i_{\gamma}(j)\,j}^{(R-1)}(z_m)f_j(x_0)~
\text{for some }i_{\gamma}(j) \in C.
\end{align}
Also note the following representation of $\eta_j$ in terms of $\gamma$ that we will use:
\begin{equation}
\label{kreeka3} \eta_j=\max_{i\in S}
\gamma_ip_{i\,j}^{(kL-1)}(x_0)f_j(x_{kL})=\gamma_{i_{\eta}(j)}p_{{i_{\eta}(j)}\,j}^{(kL-1)}(x_0)f_j(x_{kL})~
\text{for some }i_{\eta}(j)\in S.
\end{equation}
\subsubsection{Bounds on $\beta$}
\label{subsec:betabound}
Recall (\S\ref{subsec:abs}) that $b_0\in C$. We show that for every $j\in S$
\begin{align}\label{betas}
\beta_j <q^{-1}\Bigl({K\over \delta}\Bigr)^{m}\beta_{b_0}.
\end{align}
Fix $j\in S$ and consider $\alpha_{i_\beta(j)}$ from \eqref{kreeka1}. Let
$v_1,\ldots,v_{m-1}$ be a path that realizes $p_{ij}^{(m-1)}(z_0)$.
Then
$$\beta_j=\alpha_{i_\beta(j)}p_{i_\beta(j)\,v_1}f_{v_1}(z_1)p_{v_1\,v_2}f_{v_2}(z_2)\cdots
p_{v_{m-1}\,j}f_{j}(z_m)<\alpha_{i_\beta(j)}K^m.$$ (The last inequality follows
from the definition of ${\cal Z}$, \S\ref{subsec:Z}.)  Let $w_1,\ldots, w_{m-1}$ be a 
maximum probability path from $i_{\beta(j)}$ to $b_0$ as in \eqref{maxprobpath}. Thus,
$$\beta_{b_0}\geq \alpha_{i_\beta(j)}p_{i_\beta(j)\,b_0}^{(m-1)}(z_0)f_{b_0}(z_m)\geq
\alpha_{i_\beta(j)}p_{i_\beta(j)\,w_1}f_{w_1}(z_1)p_{w_1\,w_2}f_{w_2}(z_2)\cdots
p_{w_{m-1}\,b_0}f_{b_0}(z_m) \geq \alpha_{i_\beta(j)} q \delta^{m}.$$ (The last inequality again 
follows from the definition of ${\cal Z}$, \S\ref{subsec:Z}.) 
Since $q>0$ \eqref{eq:qpositive}, we thus obtain:
\begin{equation*}
\beta_j<\alpha_{i_\beta(j)}K^m\leq \frac{\beta_{b_0}}{q \delta^{m}} K^m,
\end{equation*}
as required.
\subsubsection{Likelihood ratio bounds}
\label{subsec:llbound}
We prove the following claims
\begin{align}\label{krim}
    &p^{(L-1)}_{i1}(x_{lL})\leq p^{(L-1)}_{11}(x_{lL}),\quad
    \forall i\in S,\quad \forall l=0,\ldots,2k-1;\\
    \label{tsiv}
    &{p^{(L-1)}_{ij}(x_{lL})f_j(x_{(l+1)L})\over
    p^{(L-1)}_{11}(x_{lL})f_1(x_{(l+1)L})}<1-\epsilon,\quad
    \forall i,j\in S,\,j\ne 1,\quad \forall l=0,\ldots,2k-1;\\
    \label{sugar}
    &p^{(R-1)}_{ij}(z_m)f_j(x_0)\leq A^R p^{(R-1)}_{b_01}(z_m)f_1(x_0),\quad \forall i,j\in
    S;\\
    \label{p}
    &{p^{(m+P-1)}_{ij}(x_{2kL})\over p^{(m+P-1)}_{1j}(x_{2kL})}\leq
q^{-1}\Bigl({K\over \delta}\Bigr)^{m-1},\quad \forall j\in C,\,\forall i\in S.
\end{align}
If $L=1$, then \eqref{krim} becomes $p_{i\,1}\leq p_{1\,1}$ for all $i\in S$, which is
true by the assumption $p^*_1=p_{1\,1}$ made in the course of constructing the $\mathbf{s}$ sequence (\S\ref{subsec:abs}).
If $L=1$, then \eqref{tsiv} becomes 
$${p_{ij}f_j(x_{l+1})\over p_{11}f_1(x_{l+1})}<1-\epsilon,\quad \forall i,j\in S, j\ne 1,$$
and thus, since $x_{l+1}\in {\cal X}_1,~0\le l<2k$ in this case, \eqref{tsiv} is true by the definition of 
${\cal X}_1$ (\S\ref{subsec:Xl}) (and the fact that $p^*_1=p_{1\,1}$).
Let us next prove \eqref{krim} and \eqref{tsiv} for the case $L>1$. 
Consider any $l=0,1,\ldots,2k-1$. Note that the definitions of the $s$-path \eqref{path}, 
${\cal X}_{s_i}$ (\S\ref{subsec:Xl}), and the fact that $x_{lL+i}\in {\cal X}_{s_i}$ for $1\le i <L$ 
imply that given observations $x_{Ll+1},\ldots, x_{L(l+1)-1}$, the path
$s_1,\ldots,s_{L-1}$ realizes the maximum in
$p^{(L-1)}_{11}(x_{Ll})$, i.e.
\begin{equation}\label{tee}
p_{11}^{(L-1)}(x_{lL})=p_{1\,s_1}f_{s_1}(x_{lL+1})p_{s_1\,s_2}\cdots
p_{s_{L-2}\,s_{L-1}}f_{s_{L-1}}(x_{(l+1)L-1})p_{s_{L-1}\,1}.
\end{equation}
(Indeed, $$p_{1\,s_1}f_{s_1}(x_{lL+1})p_{s_1\,s_2}\cdots
p_{s_{L-2}\,s_{L-1}}f_{s_{L-1}}(x_{(l+1)L-1})p_{s_{L-1}\,1}=
p^*_{s_1}f_{s_1}(x_{lL+1})p^*_{s_2}\cdots
p^*_{s_{L-1}}f_{s_{L-1}}(x_{(l+1)L-1})p^*_{1},$$
and for $i=1,2,\ldots,L-1$, $p^*_{s_i}f_{s_i}(x_{lL+i})\ge p_{hj}f_{j}(x_{lL+i})$ for any $h,j\in S$.)
Suppose $j\neq 1$ and 
$t_1,\ldots,t_{L-1}$ realizes $p_{ij}^{(L-1)}(x_{lL}),$ i.e.
\begin{equation}\label{kohv}
p_{ij}^{(L-1)}(x_{lL})=p_{i\,t_1}f_{t_1}(x_{lL+1})p_{t_1\,t_2}\cdots
p_{t_{L-2}\,t_{L-1}}f_{t_{L-1}}(x_{(l+1)L-1})p_{t_{L-1}\,j}.
\end{equation}
Hence, with $t_0$ and  $t_L$ standing for $i$ and $j$, respectively (and $s_0=s_L=1$),
the left-hand side of \eqref{tsiv} becomes
\begin{align*}
\Bigl({p_{t_0\,t_1}f_{t_1}(x_{lL+1})\over
    p_{s_0\,s_1}f_{s_1}(x_{lL+1})}\Bigr)\Bigl({p_{t_1\,t_2}f_{t_2}(x_{lL+2})\over
    p_{s_1\,s_2}f_{s_2}(x_{lL+2})}\Bigr)\cdots \Bigl({p_{t_{L-2}\,t_{L-1}}f_{t_{L-1}}(x_{(l+1)L-1})
\over
p_{s_{L-2}\,s_{L-1}}f_{s_{L-1}}(x_{(l+1)L-1})}\Bigr)\Bigl({p_{t_{L-1}\,t_L}f_j(x_{(l+1)L})\over
p_{s_{L-1}\,s_L}f_1(x_{(l+1)L})}\Bigr).\end{align*}
For $h=1,\ldots,L$ such that  $t_h\ne s_h$,  
\begin{equation}\label{kuku}
{p_{t_{h-1}\,t_h}f_{t_h}(x_{lL+h})\over
p_{s_{h-1}\,s_h}f_{s_h}(x_{lL+h})}<1-\epsilon,\quad\text{since }x_{lL+h}\in \mathcal{X}_{s_h}. 
\end{equation}
For all other $h$, $s_h=t_h$ and therefore, the left-hand side of \eqref{kuku}
becomes ${p_{t_{h-1}\,t_h}\over p_{s_{h-1}\,s_h}}=\frac{p_{t_{h-1}\,s_h}}{p^*_{s_h}} \leq 1$ (by property \eqref{klass}).
Since the last term of the product above does satisfy \eqref{kuku} ($j\ne 1$), \eqref{tsiv} is thus proved.
Suppose next that $t_1,\ldots,t_{L-1}$ realizes $p_{i1}^{(L-1)}(x_{lL})$. With $s_0=1$ and $t_0=i$, similarly to
the previous arguments, we have
$${p_{i\,1}^{(L-1)}(x_{lL})\over p_{1\,1}^{(L-1)}(x_{lL})}=
\prod_{h=1}^{L-1} \Bigl({p_{t_{h-1}\,t_h}f_{t_h}(x_{lL+h})\over
p_{s_{h-1}\,s_h}f_{s_h}(x_{lL+h})}\Bigr){p_{t_{L-1}\,1}\over
p_{s_{L-1}\,1}}\leq 1,$$ implying \eqref{krim}.\\\\
Let us now prove \eqref{sugar}. To that end, note that for all states
$h,i,j\in S$ such that $p_{jh}>0$, it follows from the definitions $\eqref{eq:maxtransandratio}$ that
\begin{equation}\label{A}
{p_{ih}\over p_{jh}}\leq {p^*_h \over p_{jh}}\leq A.\end{equation}
If $R=1$, then \eqref{sugar} becomes $$p_{ij}f_j(x_0)\leq
Ap_{b_01}f_1(x_0).$$ By the definition of ${\cal X}_1$ (recall that
$x_0\in {\cal X}_1$), we have that for every $i,j\in S$
$p_{ij}f_j(x_0)\le p^*_1f_1(x_0)$. Using \eqref{A} with $h=1$ and $j=b_{0}$, we get
$p^*_1f_1(x_0)\leq Ap_{b_0\,1}f_1(x_0)$ ($p_{b_0\,1}>0$ by the construction of $\mathbf{b}$ 
\S\ref{subsec:abs}). Putting these all together, we obtain
\begin{equation*}
p_{ij}f_j(x_0)<p^*_1f_1(x_0)\leq Ap_{b_0 1}f_1(x_0),~\text{as required.}
\end{equation*}
Consider the case $R>1$. Let $t_1,\ldots,t_{R-1}$ be a path that
realizes $p^{(R-1)}_{ij}(z_m)$, i.e.
$$p^{(R-1)}_{ij}(z_m)=p_{i\,t_1}f_{t_1}(x'_1)p_{t_1\, t_2}f_{t_2}(x'_2)\cdots
p_{t_{R-2}\,t_{R-1}}f_{t_{R-1}}(x'_{R-1})p_{t_{R-1}j}.$$
By 
the definition of $\mathcal{X}_l$ (\S\ref{subsec:Xl}) and
the facts that $x'_r\in \mathcal{X}_{b_r}$, $r=1,2,\ldots,R-1$, and $x_0\in\mathcal{X}_1$, 
we have   
\begin{equation}\label{eq:claim3proof}
p^{(R-1)}_{ij}(z_m)f_j(x_0)\le p^*_{b_1}f_{b_1}(x'_1)p^*_{b_2}f_{b_2}(x'_2)\cdots
p^*_{b_{R-1}}f_{b_{R-1}}(x'_{R-1})p^*_1f_1(x_0).\end{equation} 
Now, by the construction of $\mathbf{b}$ (\S\ref{subsec:abs}), $p_{b_{r-1}\,b_r}>0$ for $r=1,\ldots,R$, ($b_R=1$).
Thus,  the argument behind \eqref{A} applies here to bound the right-hand side of \eqref{eq:claim3proof} from above
by $$Ap_{b_0\,b_1}f_{b_1}(x'_1)Ap_{b_1\,b_2}f_{b_2}(x'_2)\cdots
Ap_{b_{R-2}\,b_{R-1}}f_{b_{R-1}}(x'_{R-1})Ap_{b_{R-1}\,1}f_1(x_0)=A^Rp^{(R-1)}_{b_0\,1}(z_m)f_1(x_0),$$
as required. \\\\
Let us now prove \eqref{p}.
If $m=1$ then \eqref{p} becomes
\begin{equation}\label{teem1}
p^{(P)}_{ij}(x_{2kL})\leq p^{(P)}_{1j}(x_{2kL})q^{-1},\quad \forall
j\in C,\,\forall i\in S.
\end{equation}
If $P=0$, then \eqref{teem1} reduces to $p_{ij}\leq p_{1j}q^{-1}$ which is true, because in this case 
the state $q_1=q_T=1$ belongs to $C$ (\S\ref{subsec:abs}) and
$p_{1j}q^{-1}\geq 1$ (\eqref{maxprobpath}, \eqref{eq:qpositive} with $m=1$). 
To see why \eqref{teem1} is true with $P\geq 1$, note  
that by the same argument as used to prove \eqref{krim} and \eqref{tsiv}, we now get
\begin{equation}\label{tele}
p_{1\,a_P}^{(P-1)}(x_{2kL})f_{a_P}(x^{''}_P)\geq
p_{h',l}^{(P-1)}(x_{2kL})f_{l}(x^{''}_P),\quad \forall h,l\in
S.\end{equation} Also, since $a_P=q_1\in C$ (\S\ref{subsec:abs}),  
$p_{a_P\,j}q^{-1}\geq 1$ (\eqref{maxprobpath}, \eqref{eq:qpositive} with $m=1$). Thus
\begin{eqnarray*}
p_{i\,j}^{(P)}(x_{2kL})&\stackrel{\mathrm{by~\eqref{eq:prrecurse}}}{=}&
\max_{l\in S} p_{i\,l}^{(P-1)}(x_{2kL})f_l(x^{''}_P)p_{l\,j}
                \stackrel{\mathrm{by~\eqref{tele}}}{\le} p^{(P-1)}_{1a_P}(x_{2kL})f_{a_P}(x''_P)\max_{l\in S}p_{l\,j}\le\\
                &\le&p^{(P-1)}_{1\,a_P}(x_{2kL})f_{a_P}(x''_P) \le 
p_{1\,a_P}^{(P-1)}(x_{2kL})f_{a_P}(x^{''}_P)p_{a_P\,j}q^{-1}\stackrel{\mathrm{by~\eqref{eq:prrecurse}}}{\le} 
p^{(P)}_{1\,j}(x_{2kL})q^{-1}.
\end{eqnarray*}
For $m>1$, let $t_1,t_2,\ldots,t_{m-1}$ be a path realizing $p^{(m-1)}_{h\,j}(x^{''}_P)$.
Thus, 
\begin{equation}\label{eq:auxbound1}
p^{(m-1)}_{h\,j}(x^{''}_P)=p_{h\,t_1}f_{t_1}(z'_1)p_{t_1\,t_2}f_{t_2}(z'_2)\cdots
f_{t_{m-1}}(z'_{m-1})p_{t_{m-1}j}< K^{m-1}.
\end{equation}
(This is true since $z'_r\in \mathcal{Z}$ for $r=1,2,\ldots,m-1$ (\S\ref{subsec:Z})
and thus, for $p^{(m-1)}_{h\,j}(x^{''}_P)$ to be positive it is 
necessary that  $t_r\in C$, $r=1,\ldots,m-1$, implying $f_{t_r}(z'_r)<K$.) Now, 
let $t_1,t_2,\ldots,t_{m-1}$ realize $p^{(m-1)}_{a_P\,j}(x^{''}_P)$,
which is clearly positive, with $t_r\in C$, $r=1,\ldots,m-1$ 
($z'_r\in \mathcal{Z}$ for $r=1,2,\ldots,m-1$), and $a_P, j\in C$ 
(recall the positivity assumption on $Q^m$, \S\ref{subsec:k}).  
We thus have $p^{(m-1)}_{a_P\,j}(x^{''}_P)=p_{a_P\,t_1}f_{t_1}(z'_1)p_{t_1\,t_2}f_{t_2}(z'_2)\cdots
f_{t_{m-1}}(z'_{m-1})p_{t_{m-1}j}\ge $
\begin{equation}\label{eq:auxbound2}
\ge q^*_{a_P\,j}f_{t_1}(z'_1)f_{t_2}(z'_2)\cdots
f_{t_{m-1}}(z'_{m-1})>q \delta^{m-1}.
\end{equation}
Combining the  bounds of \eqref{eq:auxbound1} and \eqref{eq:auxbound2} ($q>0$, \eqref{eq:qpositive}), we obtain :
\begin{equation} \label{eq:teem1further}
p^{(m-1)}_{h\,j}(x^{''}_P) < p^{(m-1)}_{a_P\,j}(x^{''}_P)\Bigl({K\over \delta}\Bigr)^{m-1}q^{-1}. 
\end{equation}
Finally,
\begin{align*}
p_{ij}^{(P+m-1)}(x_{2kL})&\stackrel{\mathrm{by~\eqref{eq:prrecurse}}}{=}\max_{l\in S}
p_{il}^{(P-1)}(x_{2kL})f_l(x^{''}_P)p^{(m-1)}_{lj}(x^{''}_P)
\stackrel{\mathrm{by~\eqref{tele},~\eqref{eq:teem1further}}}{<}\\
&\stackrel{\mathrm{by~\eqref{tele},~\eqref{eq:teem1further}}}{<}
p_{1\,a_P}^{(P-1)}(x_{2kL})f_{a_P}(x^{''}_P)p^{(m-1)}_{a_P\,j}(x^{''}_P)\Bigl({K\over
\delta}\Bigr)^{m-1}q^{-1}\stackrel{\mathrm{by~\eqref{eq:prrecurse}}}{\le}\\
&\stackrel{\mathrm{by~\eqref{eq:prrecurse}}}{\le} p_{1j}^{(P+m-1)}(x_{2kL})\Bigl({K\over
\delta}\Bigr)^{m-1}q^{-1}.
\end{align*}
\subsubsection{$\gamma_j\le const \times \gamma_1$}
\label{subsec:gammabound}
Combining \eqref{kreeka2}, 
\eqref{betas}, 
and \eqref{sugar}, we get that
for every state $j\in S$,
\begin{align*}
\gamma_j&\stackrel{\mathrm{by~\eqref{kreeka2}}}{=}\beta_{i_\gamma(j)}p_{i_\gamma(j)\,j}^{(R-1)}(z_m)f_j(x_0)
\stackrel{\mathrm{by~\eqref{sugar}}}{\leq} \beta_{i_\gamma(j)}
p_{b_0\,1}^{(R-1)}(z_m)f_1(x_0)A^R \stackrel{\mathrm{by~\eqref{betas}}}{\leq} \\
&\stackrel{\mathrm{by~\eqref{betas}}}{\leq} q^{-1}\Bigl({K\over
\delta}\Bigr)^{m}A^R\beta_{b_0}p_{b_0\,1}^{(R-1)}(z_m)f_1(x_0) \leq U
\max_{i\in S} \beta_i p_{i\,1}^{(R-1)}(z_m)f_1(x_0)\stackrel{\mathrm{by~\eqref{kreeka2}}}{=}U\gamma_1,
\end{align*}
where
\begin{equation}\label{eq:defU}
U\stackrel{\mathrm{def}}{=}q^{-1}\Bigl({K\over
\delta}\Bigr)^{m}A^R.\end{equation}
Hence
\begin{equation}\label{con}
\gamma_j\leq U\gamma_1,\quad  \forall j\in S. \end{equation}
\subsubsection{Further bounds on likelihoods}
\label{subsec:gammaboundmore}
Let $l\ge 0$ and $n>0$ be integers such that  $l+n \leq 2k$ but arbitrary otherwise.
Expanding $p_{1\,1}^{(nL-1)}(x_{lL})$ recursively according with \eqref{eq:prrecurse}, we obtain
\begin{align}\label{prekohus}
p_{1\,1}^{(nL-1)}(x_{lL})=&\max_{i_1,i_2,\ldots,i_{n-1}\in S} p_{1\,i_1}^{(L-1)}(x_{lL})f_{i_1}(x_{(l+1)L})
p_{i_1\,i_2}^{(L-1)}(x_{(l+1)L})f_{i_2}(x_{(l+2)L}) \cdots\\
\nonumber \cdots
&p_{i_{n-2}\,i_{n-1}}^{(L-1)}(x_{(l+n-2)L})f_{i_{n-1}}(x_{(l+n-1)L})p_{i_{n-1}\,1}^{(L-1)}(x_{(l+n-1)L}).
\end{align}
Since  $p_{1\,i_1}^{(L-1)}(x_{lL})f_{i_1}(x_{(l+1)L})\le p_{1\,1}^{(L-1)}(x_{lL})f_{1}(x_{(l+1)L})$ for any $i_1\in S$,
as well as $$p_{i_{r-1}\,i_r}^{(L-1)}(x_{(l+r-1)L})f_{i_r}(x_{(l+r)L})\le p_{1\,1}^{(L-1)}(x_{(l+r-1)L})f_1(x_{(l+r)L})~
r=2,\ldots,n-1,~
 \text{by~  
\eqref{tsiv}},$$ and since $p_{i_{n-1}\,1}^{(L-1)}(x_{(l+n-1)L})\le p_{1\,1}^{(L-1)}(x_{(l+n-1)L})$ for any $i_{n-1}\in S$
by \eqref{krim}, maximization \eqref{prekohus} above is achieved as in \eqref{kohus} below:
\begin{align}\label{kohus}
p_{1\,1}^{(nL-1)}(x_{lL})=&p_{1\,1}^{(L-1)}(x_{lL})f_1(x_{(l+1)L})p_{11}^{(L-1)}(x_{(l+1)L})f_1
(x_{(l+2)L}) \cdots\\
\nonumber \cdots
&p_{1\,1}^{(L-1)}(x_{(l+n-2)L})f_1(x_{(l+n-1)L})p_{1\,1}^{(L-1)}(x_{(l+n-1)L}).
\end{align}
Now, we replace state $1$ by generic states $i,j\in S$ on the both ends of the paths in \eqref{prekohus}
and repeat the above arguments. Thus, also using \eqref{kohus}, we arrive at bound \eqref{liim} below:
\begin{align}\label{liim}
p_{ij}^{(nL-1)}(x_{lL})f_j(x_{(l+n)L})&\leq \prod_{u=l+1}^{l+n}
p_{11}^{(L-1)}(x_{(u-1)L})f_1(x_{uL})\stackrel{\mathrm{by~\eqref{kohus}}}{=}\\
&\stackrel{\mathrm{by~\eqref{kohus}}}{=}p_{11}^{(nL-1)}(x_{lL})f_1(x_{(l+n)L}),~\forall i,j\in S.
\nonumber
\end{align}

In particular, \eqref{liim} states
\begin{equation}\label{pulk}
p_{ij}^{(kL-1)}(x_0)f_j(x_{kL})\leq
p_{11}^{(kL-1)}(x_0)f_1(x_{kL}),\quad \forall i,j\in S.
\end{equation}
\subsubsection{$\eta_j\le const \times \eta_1$}
\label{subsec:etabound}
In order to see
\begin{equation}\label{gamma}
\eta_j\leq U\eta_1,\quad \forall j\in S,
\end{equation}
note that:
\begin{align*}
\eta_j\stackrel{\mathrm{\eqref{kreeka3}}}{=}&\max_{i\in S} \gamma_i p_{i\,j}^{(kL-1)}(x_0)f_j(x_{kL})
\stackrel{\mathrm{by~\eqref{pulk}}}{\leq}
\max_{i\in S} \gamma_i p_{1\,1}^{(kL-1)}(x_0)f_1(x_{kL})\stackrel{\mathrm{by~\eqref{con}}}{\leq}\\
    \stackrel{\mathrm{by~\eqref{con}}}{\leq}&U\gamma_1 p_{1\,1}^{(kL-1)}(x_0)f_1(x_{kL})
\stackrel{\mathrm{by~\eqref{kreeka3}}}{\le } U\eta_1.
\end{align*}
\subsubsection{A representation of  $\eta_1$}
\label{subsec:etaone}
Recall that $k$, the number of cycles in the $s$-path, was chosen sufficiently large for 
\eqref{eq:largek} to hold (in particular, $k>1$).
We now prove that there exists $\kappa\in \{1,\ldots,k-1\}$ such that
\begin{equation}\label{on}
\eta_1=\delta_{1}(x_{\kappa L})p_{1\,1}^{((k-\kappa)L-1)}(x_{\kappa L})f_1(x_{kL}).
\end{equation}
The relation \eqref{on} states that (given observations $x_1,x_2,\ldots,x_u$) a maximum-likelihood path  
(from time $1$, observation $x_1$) to time $u-m-P-kL$ (observation $x_{kL}$) 
goes through state $1$ at time $u-m-P-2kL+\kappa L$, that is when $x_{\kappa L}$ is observed.\\

To see this, suppose no such $\kappa$ exists to satisfy \eqref{on}. Then, 
applying \eqref{eq:prrecurse} to \eqref{kreeka3} and recalling that
$\delta_{1}(x_{\kappa L})$ is introduced by \eqref{eq:shortscores}, we would have
$$
\eta_1=\gamma_{j_\eta(1)}p_{j_\eta(1)\,j_1}^{(L-1)}(x_0)f_{j_1}(x_L)
p_{j_1\,j_2}^{(L-1)}(x_L)f_{j_2}(x_{2L})p_{j_2\,j_3}^{(L-1)}(x_{2L})\cdots
p_{j_{k-1}\,1}^{(L-1)}(x_{(k-1)L})f_1(x_{kL})
$$
for some $j_1\ne 1,\ldots,j_{k-1}\ne 1$. Furthermore, this would imply 
\begin{align}\label{jalle}
\nonumber
\eta_1&\stackrel{\mathrm{by~\eqref{tsiv},~\eqref{krim}}}{<}\gamma_{j_\eta(1)}(1-\epsilon)^{k-1}
\prod_{i=1}^{k}p_{1\,1}^{(L-1)}(x_{(i-1)L})f_1(x_{iL})\stackrel{\mathrm{by~\eqref{eq:largek}}}{<}\\
\nonumber
& \stackrel{\mathrm{by~\eqref{eq:largek}}}{<}
\gamma_{j_\eta(1)}q^2\left({\delta\over K}\right)^{2m}A^{-R} \prod_{i=1}^{k}p_{1\,1}^{(L-1)}(x_{(i-1)L})f_1(x_{iL})
\stackrel{\mathrm{by~\eqref{con}}}{\le}\\ 
\nonumber
& \stackrel{\mathrm{by~\eqref{con}}}{\le}
\gamma_1Uq^2\left ({\delta\over K}\right )^{2m}A^{-R} \prod_{i=1}^{k}p_{1\,1}^{(L-1)}(x_{(i-1)L})f_1(x_{iL})
\stackrel{\mathrm{by~\eqref{eq:defU}}}{=}\\ & \stackrel{\mathrm{by~\eqref{eq:defU}}}{=}
\gamma_1q\left ({\delta\over K}\right )^m \prod_{i=1}^{k}p_{1\,1}^{(L-1)}(x_{(i-1)L})f_1(x_{iL})
<\gamma_1 \prod_{i=1}^{k}p_{1\,1}^{(L-1)}(x_{(i-1)L})f_1(x_{iL}).
\end{align}
(The last inequality follows from $q\le 1$ \eqref{eq:qpositive} and $\delta<K$, \S\ref{subsec:Z}.) 
On the other hand, by definition \eqref{kreeka3} (and $k-1$-fold application of \eqref{eq:prrecurse}),
$\eta_1\ge \gamma_1 \prod_{i=1}^{k}p_{1\,1}^{(L-1)}(x_{(i-1)L})f_1(x_{iL}),$ which evidently contradicts 
\eqref{jalle} above.
Therefore, $\kappa$ satisfying \eqref{on} and $1\le \kappa< k$, does exist. 
\subsubsection{An implication  of  \eqref{kohus} and \eqref{on} for $\delta_1(x_{lL})$}
\label{subsec:deltaone}
Clearly, the arguments of the previous section (\S\ref{subsec:etaone}) are valid
if $k$ is replaced by any $l\in\{k,\ldots,2k\}$. Hence the following generalization of
\eqref{on}:
\begin{equation}\label{preojakaar}
\delta_1(x_{lL})=\delta_1(x_{\kappa(l)L})p^{((l-\kappa(l))L-1)}_{11}(x_{\kappa(l)L})f_1(x_{lL})
~\text{for some}~\kappa(l)<l.
\end{equation}
We apply \eqref{preojakaar} recursively, starting with  $\kappa^{(0)}:=l$ and returning
$\kappa^{(1)}:=\kappa(l)<l$. If $\kappa^{(1)}\leq k$, we stop, otherwise we substitute $\kappa^{(1)}$
for $l$, and obtain $\kappa^{(2)}:=\kappa(l)<\kappa^{(1)}$, and so, on until $\kappa^{(j)}\le k$ for some $j>0$.
Thus, 
\begin{equation}\label{preojakaar2}\hspace*{-4mm}
\delta_1(x_{lL})=\delta_1(x_{\kappa^{(j)}L})p^{((\kappa^{(j-1)}-\kappa^{(j)})L-1)}_{11}(x_{\kappa^{(j)}L})
f_1(x_{\kappa^{(j-1)}L})\cdots p^{((l-\kappa^{(1)})L-1)}_{11}(x_{\kappa^{(1)}L})
f_1(x_{lL}).
\end{equation}
Applying \eqref{kohus} to the appropriate factors of the right-hand side of \eqref{preojakaar2} above, 
we get:
\begin{align}\label{preojakaar3}
\delta_1(x_{lL})  = &\delta_1(x_{\kappa^{(j)}L})p^{(L-1)}_{11}(x_{\kappa^{(j)}L})f_1(x_{(\kappa^{(j)}+1)L})\cdots 
       p^{(L-1)}_{11}(x_{(k-1)L})f_1(x_{kL})\cdots\\
\nonumber & \cdots p^{(L-1)}_{11}(x_{kL})f_1(x_{(k+1)L})
           \cdots p^{(L-1)}_{11}(x_{(\kappa^{(j-1)}-1)L})f_1(x_{\kappa^{(j-1)}L})\cdots \\ \nonumber
  & \cdots p^{(L-1)}_{11}(x_{(\kappa^{(1)}-1)L})f_1(x_{\kappa^{(1)}L})\cdots p^{(L-1)}_{11}(x_{(l-1)L})f_1(x_{lL}).
\end{align}
Also, according to \eqref{kohus}, 
$$ \delta_1(x_{\kappa^{(j)}L})p^{(L-1)}_{11}(x_{\kappa^{(j)}L})f_1(x_{(\kappa^{(j)}+1)L})\cdots 
p^{(L-1)}_{11}(x_{(k-1)L}) =\delta_1(x_{\kappa^{(j)}L})p^{((k-\kappa^{(j)})L-1)}_{11}(x_{\kappa^{(j)}L}).$$
At the same time, 
\begin{equation}\label{preojakaar4}
\delta_1(x_{\kappa^{(j)}L})p^{((k-\kappa^{(j)})L-1)}_{11}(x_{\kappa^{(j)}L})f_1(x_{kL})
\stackrel{\mathrm{by~\eqref{eq:prrecurse2}}} \le \eta_1.
\end{equation}
However, we cannot have the strict inequality in \eqref{preojakaar4} above since that, via
\eqref{preojakaar3},  would contradict maximality of $\delta_1(x_{lL})$.  
We have thus arrived at 
\begin{equation}\label{ojakaar}
\delta_1(x_{lL})=\eta_1p^{(L-1)}_{11}(x_{{k}L})f_1(x_{(k +1)L})\cdots p^{(L-1)}_{11}(x_{(l-1)L})f_1(x_{lL}).
\end{equation}

In summary, for any $l\ge k$ and $l\le 2k$ there exists a realization of  $\delta_1(x_{lL})$ that goes
through state $1$ every time when $x_{iL}$, $i=k,\ldots, l$, is observed.
\subsubsection{$x_{kL}$ is a $(kL+m+P)$-order 1-node}
\label{subsec:onenode}
When we prove in \S\ref{subsec:auxproof35} that
for any  $i\in S,i\ne 1,$ and any$ j\in C$, 
\begin{equation}\label{lill}
\eta_ip_{ij}^{(kL+m+P-1)}(x_{kL})\leq
\eta_1p_{1j}^{(kL+m+P-1)}(x_{kL}),
\end{equation}
this will immediately imply that $x_{kL}$ is a  1-node of order $kL+m+P$. Indeed, 
let $l\in S$ be arbitrary. Since $f_j(z'_m)=0$ for every $j\in S\setminus C$, any maximum
likelihood path to state $l$ at time $u+1$ (observation $x_{u+1}$) must go
through a state in $C$ at time $u$ (observation $x_u=z'_m$.) Formally,
\begin{align*}
\eta_ip_{il}^{(kL+m+P)}(x_{kL})&= \max_{j\in
S}\eta_ip_{ij}^{(kL+m+P-1)}(x_{kL})f_j(z'_m)p_{jl}= \max_{j\in
C}\eta_ip_{ij}^{(kL+m+P-1)}(x_{kL})f_j(z'_m)p_{jl}    \\
&\stackrel{\mathrm{by~\eqref{lill}}}{\le}  \max_{j\in C}\eta_1p_{1j}^{(kL+m+P-1)}(x_{kL})f_j(z'_m)p_{jl}
\stackrel{\mathrm{by~\eqref{eq:prrecurse}}}{=}
\eta_1p_{1l}^{(kL+m+P)}(x_{kL}).
\end{align*}
Therefore, by Definition~\ref{rnode} $x_{kL}$ is a 1-node of order $kL+m+P$.
\subsubsection{Proof of  \eqref{lill}}
\label{subsec:auxproof35}
Let $i\in S$ and $j\in C$ be arbitrary. Let state $j^*\in S$ be such that
$$p_{i\,j}^{(kL+m+P-1)}(x_{kL})=
p_{i\,j^*}^{(kL-1)}(x_{kL})f_{j^*}(x_{2kL})p^{(m+P-1)}_{j^*\,j}(x_{2kL})=\nu(i,j^*)p^{(m+P-1)}_{j^*\,j}(x_{2kL}),$$
where
$$\nu(i,j)\stackrel{\mathrm{def}}{=}p_{ij}^{(kL-1)}(x_{kL})f_j(x_{2kL}),\quad\text{for all}~i,j\in S.$$
We consider the following two cases separately:
\begin{enumerate}[1.]
\item 
There exists a path realizing $p_{i\,j^*}^{(kL-1)}(x_{kL})$ and going
through state 1 at the time of observing $x_{lL}$ for some 
$l \in \{k,\ldots, 2k\}$.
\begin{equation}\label{ikka}
p_{i\,j^*}^{(kL-1)}(x_{kL})=p^{((l-k) L-1)}_{i\,1}(x_{kL})f_1(x_{lL})p^{((2k-l)L-1)}_{1\,j^*}(x_{lL}).
\end{equation}
Equation \eqref{ikka} above together with the fundamental recursion \eqref{eq:prrecurse2} yields
the following:
\begin{eqnarray}
    \label{eq:ikka2}\nonumber
\eta_ip_{i\,j^*}^{(kL-1)}(x_{kL})& \stackrel{\mathrm{by~\eqref{ikka}}}{=}&
\eta_ip^{((l-k) L-1)}_{i\,1}(x_{kL})f_1(x_{lL})p^{((2k-l)L-1)}_{1\,j^*}(x_{lL})
 \stackrel{\mathrm{by~\eqref{eq:shortscores},~\eqref{eq:prrecurse2}}}{\le}\\ 
& \stackrel{\mathrm{by~\eqref{eq:shortscores},~\eqref{eq:prrecurse2}}}{\le}&
\delta_1(x_{lL})p^{((2k-l)L-1)}_{1\,j}(x_{lL}).   
  \end{eqnarray}
At the same time, the right hand-side of \eqref{eq:ikka2} can be expressed as follows:
\begin{eqnarray}
    \label{eq:ikka3}\nonumber
\delta_1(x_{lL})p^{((2k-l)L-1)}_{1\,j^*}(x_{lL})&\stackrel{\mathrm{by~\eqref{ojakaar}}}{=}
&\eta_1p_{1\,1}^{((l-k)L-1)}(x_{kL})f_1(x_{lL})p^{((2k-l)L-1)}_{1\,j^*}
\stackrel{\mathrm{by~\eqref{kohus}}}{=}\\
&\stackrel{\mathrm{by~\eqref{kohus}}}{=}&\eta_1p_{1\,j^*}^{(kL-1)}(x_{kL}).  
  \end{eqnarray}
Therefore, if there exists $l \in \{k,\ldots, 2k\}$ such that \eqref{ikka} holds, 
we have by virtue of \eqref{eq:ikka2} and \eqref{eq:ikka3}:
\begin{equation}\label{ometigi}
\eta_ip_{i\,j^*}^{(kL-1)}(x_{kL})\leq
\eta_1p_{1\,j^*}^{(kL-1)}(x_{kL}),\quad\text{that is,}\quad \eta_i\nu(i,j^*)\leq
\eta_1\nu(1,j^*).
\end{equation}
Hence,
\begin{eqnarray*}
\eta_ip_{i\,j}^{(kL+m+P-1)}(x_{kL})&\stackrel{\mathrm{by~\eqref{ikka}}}{=}&
\eta_i\nu(i,j^*)p_{j^*\,l}^{(m+P-1)}(x_{2kL})\stackrel{\mathrm{by~\eqref{ometigi}}}{\le}\\
&\stackrel{\mathrm{by~\eqref{ometigi}}}{\le}&
\eta_1\nu(1,j^*)p_{j^*\,j}^{(m+P-1)}(x_{2kL})\stackrel{\mathrm{by~\eqref{eq:prrecurse}}}{\le}
\eta_1p_{1\,j}^{(kL+m+P-1)}(x_{kL})
\end{eqnarray*}
and \eqref{lill} holds.
\item 
Assume now that no path exists to satisfy \eqref{ikka}.
Argue as for \eqref{jalle} to get
\begin{equation}\label{jalle2}
\nu(i,j^*)<(1-\epsilon)^{k-1} \prod_{n=k+1}^{2k}p_{1\,1}^{(L-1)}(x_{(n-1)L})f_1(x_{nL}).
\end{equation} By \ref{kohus}, the (partial likelihood) product in the 
right-hand side of \eqref{jalle2} equals $\nu(1,1)$. Thus,
\begin{align}\label{lopp}\nonumber
\eta_i\nu(i,j^*) p^{(m+P-1)}_{j^*\,j}(x_{2kL})&\stackrel{\mathrm{by~\eqref{jalle2}}}{<}
\eta_{i}(1-\epsilon)^{k-1}\nu(1,1)p^{(m+P-1)}_{j^*\,j}(x_{2kL})
\stackrel{\mathrm{by~\eqref{eq:largek}}}{<} \\ \nonumber
&\stackrel{\mathrm{by~\eqref{eq:largek}}}{<}\eta_{i} q^2\left({\delta\over K}\right)^{2m}A^{-R} 
\nu(1,1)p^{(m+P-1)}_{j^*\,j}(x_{2kL})\stackrel{\mathrm{by~\eqref{eq:defU},~\eqref{gamma}}}{\le}\\
 &\stackrel{\mathrm{by~\eqref{eq:defU},~\eqref{gamma}}}{\le}\eta_1 q\left({\delta\over K}\right)^{m}
\nu(1,1)p^{(m+P-1)}_{j^*\,j}(x_{2kL}).
\end{align}
Hence, for every $j'\in S$,
\begin{align*}
\eta_i\nu(i,j')p^{(m+P-1)}_{j'\,j}(x_{2kL})&\stackrel{\mathrm{by~\eqref{ikka}}}{\le}
\eta_i\nu(i,j^*)p^{(m+P-1)}_{j^*\,j}(x_{2kL})\stackrel{\mathrm{by~\eqref{lopp}}}{<}\\
&\stackrel{\mathrm{by~\eqref{lopp}}}{<}
\eta_1 q\left({\delta\over K}\right)^{m}\nu(1,1)p^{(m+P-1)}_{j^*\,j}(x_{2kL})
\stackrel{\mathrm{by~\eqref{p}}}{\le}\\&\stackrel{\mathrm{by~\eqref{p}}}{\le}
\eta_1 \left({\delta\over K}\right)\nu(1,1)p^{(m+P-1)}_{1\,j}(x_{2kL})<
\eta_1 \nu(1,1)p^{(m+P-1)}_{1\,j}(x_{2kL})\stackrel{\mathrm{by~\eqref{eq:prrecurse}}}{\le}\\
&\stackrel{\mathrm{by~\eqref{eq:prrecurse}}}{\le} \eta_1p_{1\,j}^{(kL+m+P-1)}(x_{kL}),
\end{align*}
\noindent which, by virtue of \eqref{eq:prrecurse}, implies \eqref{lill}.
\end{enumerate}
\subsubsection{Completion of the $s$-path to $q_{1\ldots M}$ and conclusion}
\label{subsec:M}
Finally, let $$M=2m+2Lk+P+R+2,\quad r=kL+P+m,\quad l=1.$$ 
Recall from \S\ref{subsec:abs} that
$b_0\in C$. Since all the entries of $Q^m$ are positive, there exists a path
$v_0,v_1,\ldots,v_{m-1},b_0\in C$ such that $p_{v_i\,v_{i+1}}>0$ and
$p_{v_{m-1}b_0}>0$.
Similarly, there must exist a path $u_1,\ldots,u_{m}\in C$ such that $p_{u_i\,u_{i+1}}>0$
$\forall i=1,\ldots, m-1$ and $p_{a_P\,u_1}>0$ (recall that $a_P\in C$).
Hence, by these and the constructions of \S\ref{subsec:spath}, 
all of the  transitions of the following sequence occur with positive probabilities.
\begin{equation}\label{pathbig}
q_{1\ldots M}\stackrel{\mathrm{def}}{=}v_0,v_1,\ldots,v_{m-1},b_0,b_1,\ldots, b_{R-1},b_R,s_1,\ldots,
s_{2Lk},a_1,\ldots,a_P,u_1,\ldots,u_{m}.\end{equation} 
Clearly, the actual probability of observing $q_{1\ldots M}$ is positive, as required.
By the constructions of \S\S\ref{subsec:Xl}-\ref{subsec:abs}, the conditional
probability of $B$ below, given $q_{1\ldots M}$, is evidently positive, as required.
$$B\stackrel{\mathrm{def}}{=}{\cal Z}^{m+1}\times {\cal X}_{b_1}\times \cdots \times {\cal X}_{b_{R-1}}\times {\cal X}_1\times {\cal X}_{s_1} \times
\cdots \times {\cal X}_{s_{2kL-1}}\times {\cal X}_1\times {\cal
X}_{a_1}\times\cdots \times {\cal X}_{a_{P}}\times {\cal Z}^{m}.$$
Finally, since the sequence \eqref{blokk} below was chosen from $B$ 
arbitrarily (\S\ref{subsec:barrier})
and has been shown to be an $l$-barrier of order $r$, this completes the proof of the Lemma.
\begin{equation*}\hspace*{1.5cm}
z_0,z_1,\ldots,z_{m-1},z_{m},x'_1,\ldots,x'_{R-1},x_0,x_1,\ldots,x_{2Lk},x^{''}_1,\ldots,x^{''}_P,z'_1,\ldots,z'_{m}.
\hspace*{1.5cm}\eqref{blokk}
\end{equation*} \end{proof}
\subsection{Proof of Lemma \ref{separated}}\label{sec:separatedproof}
\begin{proof}
We use the notation of the previous proof in \S\ref{sec:proofneljas}. 
We  deal with the following two different situations: First (\S\ref{subsec:donedeal}), all barriers from $B$
as constructed in  the proof of Lemma \ref{separated} are already separated. Obviously,
there is nothing to do in this case.  The second situation (\S\ref{subsec:needmore}) is complementary,
in which case a simple extension will immediately ensure separation.  

\subsubsection{All $x^b\in B$ are already separated}
\label{subsec:donedeal}
Recall the definition of ${\cal Z}$ from \S\ref{subsec:Z}. 
Consider the two cases in the definition separately.
First, suppose ${\cal Z}=\hat{{\cal Z}}\backslash (\cup_{l\in S}{\cal X}_l)$, 
in which case ${\cal Z}$ and ${\cal X}_l$ are disjoint for every $l\in S$. 
This implies that 
every barrier \eqref{blokk} 
is already separated. Indeed, for any $w$, $1\le w\le r$,
and for any $x^b\in B$, the fact that $x^b_{M-\max(m,w)}\not\in {\cal Z}$, for example, 
makes it impossible for $(x'_{1\ldots w},x^b_{1\ldots M-w})\in B$ for any $x'_{1\ldots w}\in \mathcal{X}^w$. 
Consider now the case when ${\cal Z}=\hat{{\cal Z}}\cap {\cal X}_s$ 
for some $s\in C$. Then $$B\subset {\cal X}_s^{m+1}\times
{\cal X}_{b_1}\times \cdots \times {\cal X}_{b_{R-1}}\times {\cal
X}_1\times {\cal X}_{s_1} \times \cdots \times {\cal
X}_{s_{2kL-1}}\times {\cal X}_1\times {\cal X}_{a_1}\times\cdots
\times {\cal X}_{a_{P-1}}\times {\cal X}_s^{m+1}.$$ 
Let $x^b\in B$ be arbitrary. 
Assume first  $L>1$. By construction (\S\ref{subsec:abs}), the states $s_1,\ldots,s_L$ are all distinct.
We now show that $(x'_{1\ldots w},x^b_{1\ldots M-w})\not \in B$ for any $x'_{1\ldots w}\in \mathcal{X}^w$
when $1\le w\le r$.  Note that the sequence 
$$q_{m+2\ldots m+R+2kL+P+1}=(b_1,\ldots,b_{R-1},1,s_1,\ldots,s_{2kL-1},1,a_1,\ldots,a_{P-1},s)$$
is such that no two consecutive states are equal.  It is  straightforward to 
verify that there exist indices $j$, $0\le j\le m-1$, such that, when shifted $w$ positions to the right,
the pair $x_{j+1\,j+2}\in \mathcal{X}_s^2$ would at the same time have to belong to 
$\mathcal{X}_{q_{j+1+w}}\times \mathcal{X}_{q_{j+2+w}}$ with $m+1\le j+1+w<j+2+w\le m+R+2kL+1+P$. This
is clearly a contradiction since  $\mathcal{X}_{q_{j+1+w}}$ and  $\mathcal{X}_{q_{j+2+w}}$ are disjoint
for that range of indices $j$. A verification of the above fact simply amounts to verifying that
the inequality $\max(0,m-w)\le j\le \min(m-1,m+R+2kL-1+P-w)$ is consistent for any $w$ from the admissible
range:  
\begin{enumerate}[i.)]
\item When $0\ge m-w$, $m-1\le m+R+2kL-1+P-w$ ($m\le w \le \min(r,R+2kL+P)$), $0\le j \le m-1$
is evidently consistent.
\item When $0\ge m-w$, $m-1> m+R+2kL-1+P-w$ ($\max(m,R+2kL+P)\le w \le r$), $0\le j \le m+R+2kL-1+P-w$
is also consistent since $m+R+2kL-1+P-r=R+kL-1\ge 0$.
\item When $0< m-w$, $m-1\le m+R+2kL-1+P-w$ ($1\le w \le \min(m-1,R+2kL+P)$), $m-w\le j \le m-1$ is
consistent since $w\ge 1$.
\item When $0< m-w$, $m-1> m+R+2kL-1+P-w$ ($\max(1,R+2kL+P-1)\le w < m$), $m-w\le j \le m+R+2kL-1+P-w$ is
consistent since $R+2kL-1\ge 0$.
\end{enumerate}

Next consider the case of $L=1$ but $s\neq 1$ (that is, $P>0$). Then
$$B\subset {\cal X}_s^{m+1}\times
{\cal X}_{b_1}\times \cdots \times {\cal X}_{b_{R-1}}\times {\cal
X}_1^{2k+1}\times {\cal X}_{a_{1}}\times \cdots \times {\cal
X}_{a_{P-1}}\times {\cal X}_s^{m+1}.$$   If $s\ne 1$, then also $b_i\ne 1$,
$i=1,\ldots, R-1$ and $a_i\ne 1$, $i=1,\ldots, P-1$. To
see that $y$ is separated in this case, simply note
that $x_{M-max(w,m+1)}\not\in\mathcal{X}_s$ for any admissible $w$.

\subsubsection{Barriers $x^b\in B$ need not be separated}
\label{subsec:needmore}
Finally, we consider the case when $L=1$ and $s=1$ (where $s\in C$ is such that 
$\mathcal{Z}=\hat{\mathcal{Z}}\cap \mathcal{X}_s$). 
This implies that $P=0$, $1\in C$, and $p_{1\,1}>0$, which in turn implies that $R=1$, and
$$B\subset {\cal X}_1^{m+1}\times{\cal X}_1^{2k+1}\times {\cal X}_{1}^{m+1}=
{\cal X}_1^{2m+2k+3}.$$ Clearly, the barriers from
$B$ need not be, and indeed, are not separated.  It is, however, easy to extend them to
separated ones. Indeed, let $q_0\ne 1$ be such that
$p_{q_0\,1}>0$ and redefine $B\stackrel{\mathrm{def}}{=}{\cal X}_{q_0}\times B$. Evidently,
any shift of any $x^b\in B$ by $w$ ($1\le w\le r$) positions to the right makes it impossible 
for $x^b_1$ to be simultaneously in $\mathcal{X}_{q_0}$ and in $\mathcal{X}_1$ (since the latter sets are 
disjoint, \S\ref{subsec:Xl}).
\end{proof}


\subsection{Proof of Proposition \ref{kesk}}\label{subsec:keskproof}
\begin{proof}
Let us additionally define the following non-overlapping block-valued processes
\begin{equation*}
U^b_m\stackrel{\mathrm{def}}{=}(X_{(m-1)M+1},\ldots,X_{mM}),\quad
D^b_m\stackrel{\mathrm{def}}{=}(Y_{(m-1)M+1},\ldots,Y_{mM}),\quad m=1,2,\ldots,
\end{equation*}
and stopping times
\begin{align}\label{stop-nub}
&\nu^b_0\stackrel{\mathrm{def}}{=}\min\{m\geq 1: U^b_m\in B,D^b_m=q\},\\
&\nonumber
\nu^b_i\stackrel{\mathrm{def}}{=}\min\{m>\nu^b_{i-1}: U^b_m\in B,D^b_m=q\};\\
&\nonumber \\
\label{stop-Rb}
&R^b_0\stackrel{\mathrm{def}}{=}\min\{m>1:\,\, D^b_m=q\},\\
& \nonumber R^b_i\stackrel{\mathrm{def}}{=}\min\{m>R^b_{i-1}:\,\, D^b_m=q\}.
\end{align}
The process $D^b$ is clearly a time homogeneous, finite state Markov chain. 
Since $Y$ is aperiodic and irreducible, so is $D^b$. Hence $(D^b,U^b)$ is also an HMM.\\
Since $Y$ is stationary (under $\pi$), $q$ occurs in every interval of length $M$ with the same positive 
probability (Lemma \ref{separated}). 
In particular, $q$ belongs to the state space of $D^b$. 
Since $D^b$ is irreducible and its state space is finite, all of its states, including $q$, 
are positive recurrent. Hence $E_{\pi'}(R^b_0)<\infty$ and 
$E_{\pi'}(R^b_1-R^b_0)<\infty$ and recall that $R^b_i-R^b_{i-1}$, $i=1,2,\ldots$ are i.i.d. 
(These and the statements below hold for any initial distribution $\pi'$.)
 The following two equalities are straightforward to verify and ultimately yield the second statement: 
$E_{\pi'}(\nu_1-\nu_0)\leq E_{\pi'}(\nu^b_1-\nu^b_0)=\frac{1}{\gamma^*}E_{\pi'}(R^b_1-R^b_0)<\infty$.
The second equality above is also a simple extension of the Wald's equation (for a general
reference see, for example, \cite{grimmet-stirzaker}).

It can similarly be verified that $E_{\pi'}\nu^b_0= \gamma^*E_{\pi'}R^b_0 +
\frac{1-\gamma^*}{\gamma^*}E_{\pi'}(R^b_1-R^b_0)$, which is again  finite.
Finally, $E_{\pi'}\nu_0\leq M(E_{\pi'}\nu^b_0-1)+1<\infty$.
\end{proof}
\subsection{Proof of Proposition \ref{reg}}\label{subsec:reg}
\begin{proof}
 Recall \eqref{renewal}, the definition of stopping times $\tau$, according
to which, for each $i=1,2,\ldots$ the underlying Markov chain satisfies
$Y_{\tau_i}=l$. Hence, the behavior of $X$ after $\tau_i$ does not
depend of the behavior of $X$ up to $\tau_i$. Together with the fact that
$T_i$ are renewal, this establishes regenerativity of $X$. 
Next, to every $\tau_i$ there corresponds a $r$-order $l$-node and
$\tau_i$ is always $u_j$ for some $j>i$. This means that all the
nodes corresponding to $\tau_i$'s  are also used to define the
alignment as in Definition \ref{l6pmatu}.
 Therefore, the
 alignment up to $\tau_i$ does not depend on the alignment after $\tau_i$.
 In other words, the segment of the alignment process
 that corresponds to  $T_i$ is a function of the segment of $X$ corresponding to the same
 $T_i$. Formally
 $$(V_s:s\in \tau_{i-1}+1,\ldots,\tau_{i})=v_{(l)}(X_s: s\in \tau_{i-1}+1,\ldots,\tau_{i}).$$
Thus, the process $Z$ is regenerative with respect to $\tau$.
\end{proof}
\subsection{Proof of Theorem \ref{saddam}}\label{subsec:proofofsaddam}
\begin{proof}
First note that the right-hand side of \eqref{jaotus} does define a measure.\\
The proof below uses regenerativity of $Z$ in a standard way.
For every $n\geq \tau_0$ and $A\in {\cal B}$, and for every $l\in
S$, we have
\begin{equation*}
{1\over n}\sum_{i=1}^n I_{A\times l}(Z_i)={1\over
n}\sum_{i=1}^{\tau_o} I_{A\times l}(Z_i)+{1\over
n}\sum_{i=\tau_o+1}^{{\tau_{k(n)}}} I_{A\times l}(Z_i)+ {1\over
n}\sum_{i=\tau_{k(n)}+1}^n I_{A\times l}(Z_i)
\end{equation*}
where $k(n)=\max\{k:\tau_k\leq n\}$ stands for the renewal
process. Now, since $\tau_0<\infty$ a.s., we have
$${1\over n}\sum_{i=1}^{\tau_0} I_{A\times l}(Z_i)
\leq {\tau_0\over n}\to 0,\quad {\rm a.s.}.$$
Let $\mu\stackrel{\mathrm{def}}{=}E\tau_0^r$. By \eqref{lopp0},
 $\mu<\infty$. Then
$${n-\tau_{k(n)}\over n}\leq {T_{k(n)+1}\over n}\to 0,\quad{\rm a.s.}$$
Finally, since $Z$ is regenerative with respect to 
$\tau_0,\tau_1,\ldots $, we have
$${1\over n}\sum_{i=\tau_0+1}^{\tau_{k(n)}} I_{A\times l}(Z_i)
= {k(n)\over n}{1\over k(n)}\sum_{k=1}^{{k(n)}} \xi_k,$$ where
$$\xi_k\stackrel{\mathrm{def}}{=}\sum_{i=\tau_{k-1}+1}^{\tau_k}I_{A\times l}(Z_i),\quad k=1,2,\ldots$$
are i.i.d. random variables. Let $m_l(A)$ stand for $E\xi_k$. Thus, $m_l(A)\leq \mu<\infty$. Then, as $n\to
\infty$, we have
$${n\over k(n)}\to {\mu}\quad {\rm and}\quad {1\over k(n)}
\sum_{k=1}^{{k(n)}} \xi_k\to m_l(A),\quad{\rm a.s.}$$
Let us calculate $m_l(A)$. Clearly,
$$m_l(A)=E\sum_{i=1}^{\tau_0^r} I_{A\times l}(Z^r_i).$$
Now
\begin{align*}
m_l(A)&=E\sum_{i=1}^{\tau_0^r} I_{A\times l}(Z^r_i)=
\sum_{j=1}^{\infty}E(\sum_{i=1}^j I_{A\times
l}(Z^r_i)|\tau_0^r=j){\bf P}(\tau_0^r=j)\\
&=\sum_{j=1}^{\infty}\sum_{i=1}^j{\bf P}(Z^r_i\in A\times l|\tau_0^r=j){\bf P}(\tau_0^r=j)\\
&= \sum_{j=1}^{\infty} {\bf P}(Z^r_1\in A\times l|\tau_0^r=j){\bf
P}(\tau_0^r=j)+
\sum_{j=2}^{\infty} {\bf P}(Z^r_2\in A\times l|\tau_0^r=j){\bf P}(\tau_0^r=j)+\cdots \\
&= {\bf P}(Z^r_1\in A\times l, \tau_0^r\geq 1)+{\bf P}
(Z^r_2\in A\times l, \tau_0^r\geq 2)+\cdots\\
&= \sum_{i=1}^{\infty} {\bf P}(Z^r_i \in A\times l, \tau_0^r\geq
i)\leq \sum_{i=1}^{\infty} {\bf P}( \tau_0^r\geq i)=\mu<\infty
\end{align*}
Similarly,
\begin{equation}\label{viterbi}
{1\over n}\sum_{i=1}^n I_{l}(V^r_i)\to {w_l\over \mu}
\leq 1,\quad {\rm a.s.},
\end{equation}
where $w_l\stackrel{\mathrm{def}}{=}\sum_{i=1}^{\infty} P(V^r_i=l, \tau_0^r\geq i)$. 
Hence, we have shown that for each $l\in S$ and for every $A\in
{\cal B}$
\begin{equation}\label{Qkuju}
Q_l^n(A)\to {m_l(A)\over w_l}={\sum_{i=1}^{\infty}{\bf P}(Z^r_i
\in A\times l, \tau_0^r\geq i)\over \sum_{i=1}^{\infty} {\bf
P}(V^r_i=l,\tau_0^r\geq i)},\quad \rm{a.s.}
\end{equation}
Recalling that ${\cal X}$ is a separable metric space and envoking the theory of 
weak convergence of measures  now establishes 
$Q_l^n\Rightarrow Q_l,\quad \text{a.s.}.$
\\\\
It remains to show that for all $l\in S$ and $A\in {\cal B}$
\begin{equation}\label{kill}
P_l^n(A)\to {m_l(A)\over w_l},\quad{\rm a.s.}.
\end{equation}
To see this, consider $\sum_{i=1}^n I_{A\times
l}(X_i,V'_i)$. Since $V'_i=V_i$, if $i\leq \tau_{k(n)}$, we obtain
\begin{equation*}
{1\over n}\sum_{i=1}^n I_{A\times l}(X_i,V'_i)={1\over
n}\sum_{i=1}^{\tau_0} I_{A\times l}(Z_i)+{1\over
n}\sum_{i=\tau_0+1}^{\tau_{k(n)}} I_{A\times l}(Z_i)+ {1\over
n}\sum_{i=\tau_{k(n)}+1}^n I_{A\times l}(X_i,V'_i)\to {m_l(A)\over
\mu}\quad \text{a.s.}
\end{equation*}
Similarly,
\begin{equation}\label{viterbi2}
{1\over n}\sum_{i=1}^n I_{l}(V'_i)\to {1\over
\mu}\sum_{i=1}^{\infty} P_1(V_i=l, \tau_0^r\geq i)={w_l\over
\mu},\quad{\rm a.s.}.
\end{equation}
These convergences prove \eqref{kill}.
\end{proof}

\section*{Acknowledgement}
The first author has been supported by the 
Estonian Science Foundation Grant 5694 at the later stages of the
work. The authors are also thankful to {\em Eurandom} (The Netherlands) 
and Professors R. Gill and A. van der Vaart for their support.  
\bibliography{../ref}

\begin{thebibliography}{10}

\bibitem{baumhmm}
L.~Baum and T.~Petrie.
\newblock Statistical inference for probabilistic functions of finite state
  {M}arkov chains.
\newblock {\em Ann. Math. Stat.}, 37:1554--1563, 1966.

\bibitem{emtutorial}
J.~Bilmes.
\newblock A gentle tutorial of the {EM} algorithm and its application to
  parameter estimation for {G}aussian mixture and hidden {M}arkov models.
\newblock Technical Report 97--021, International Computer Science Institute,
  Berkeley, CA, USA, 1998.

\bibitem{bremaud}
P.~Br\'{e}maud.
\newblock {\em {M}arkov {C}hains: {G}ibbs fields, {M}onte {C}arlo simulation,
  and {Q}ueues}.
\newblock Springer, 1999.

\bibitem{gray}
P.~Chou, T.~Lookbaugh, and R.~Gray.
\newblock Entropy-{C}onstrained {V}ector {Q}uantization.
\newblock {\em IEEE Trans. Acoust. Speech Signal Process.}, 37(1):31--42, 1989.

\bibitem{dna2}
G.~Ehret, P.~Reichenbach, U.~Schindler, and {\em et.~al.}
\newblock D{N}{A} {B}inding {S}pecificity of {D}ifferent {S}{T}{A}{T}
  {P}roteins.
\newblock {\em J. Biol. Chem.}, 276(9):6675--6688, 2001.

\bibitem{gray4}
R.~Gray, T.~Linder, and J.~Li.
\newblock A {L}agrangian formulation of {Z}ador's entropy-constrained
  quantization theorem.
\newblock {\em IEEE Trans. Inf. Theory}, 48(3):695--707, 2000.

\bibitem{grimmet-stirzaker}
G.~Grimmet and D.~Stirzaker.
\newblock {\em Probability and Random Processes}.
\newblock Oxford University Press Inc., 2 edition, 1995.

\bibitem{vanaraamat}
X.~Huang, Y.~Ariki, and M.~Jack.
\newblock {\em Hidden {M}arkov models for speech recognition}.
\newblock Edinburgh University Press, Edinburgh, UK, 1990.

\bibitem{jelinek}
F.~Jelinek.
\newblock {\em Statistical methods for speech recognition}.
\newblock The MIT Press, Cambridge, MA, USA, 2001.

\bibitem{Rabiner90}
B.-H. Juang and L.R. Rabiner.
\newblock The segmental {K}-means algorithm for estimating parameters of hidden
  {M}arkov models.
\newblock {\em IEEE Trans. Acoust. Speech Signal Proc.}, 38(9):1639--1641,
  1990.

\bibitem{AVT3}
A.~Koloydenko, M.~K\"a\"arik, and J.~Lember.
\newblock \href{http://www.springerlink.com/content/t5882055g807t574/} {On
  {A}djusted {V}iterbi training}.
\newblock {\em {A}cta {A}ppl. {M}ath.}, 96(1---3):309---326, 2007.

\bibitem{AVT1}
J.~Lember and A.~Koloydenko.
\newblock
  \href{http://journals.cambridge.org/action/displayAbstract?fromPage=online&a%
id=1293756&fulltextType=RA&fileId=S0269964807000083} {Adjusted {V}iterbi
  training: {A} proof of concept}.
\newblock {\em Probab. Eng. Inf. Sci.}, 21(3):451---475, 2007.

\bibitem{AVT4}
J.~Lember and A.~Koloydenko.
\newblock \href{http://www.maths.nottingham.ac.uk/personal/pmzaak/VA/AVT4.pdf}
  {The {A}djusted {V}iterbi training for hidden {M}arkov models}.
\newblock {\em Bernoulli}, 2007.
\newblock Invited revision submitted.

\bibitem{gray2}
J.~Li, R.~Gray, and R.~Olshen.
\newblock Multiresolution image classification by hierarchical modeling with
  two dimensional hidden {M}arkov models.
\newblock {\em IEEE Trans. Inf. Theory}, 46(5):1826--1841, 2000.

\bibitem{philips}
H.~Ney, V.~Steinbiss, R.~Haeb-Umbach, and {\em et.~al.}
\newblock An overview of the {P}hilips research system for large vocabulary
  continuous speech recognition.
\newblock {\em Int. J. Pattern Recognit. Artif. Intell.}, 8(1):33--70, 1994.

\bibitem{ochney}
F.~Och and H.~Ney.
\newblock Improved {S}tatistical {A}lignment {M}odels.
\newblock In {\em Proceedings of 38th Annual Meeting of the Association for
  Computational Linguistics}, 2000.

\bibitem{dna1}
U.~Ohler, H.~Niemann, G.~Liao, and G.~Rubin.
\newblock Joint modeling of {D}{N}{A} sequence and physical properties to
  improve eukaryotic promoter recognition.
\newblock {\em Bioinformatics}, 17(Suppl. 1):S199--S206, 2001.

\bibitem{tutorial}
L.~Rabiner.
\newblock A tutorial on {H}idden {M}arkov models and selected applications in
  speech recognition.
\newblock {\em Proc. IEEE}, 77(2):257--286, 1989.

\bibitem{raamat}
L.~Rabiner and B.~Juang.
\newblock {\em Fundamentals of speech recognition}.
\newblock Prentice-Hall, Inc., Upper Saddle River, NJ, USA, 1993.

\bibitem{Rabiner86}
L.~Rabiner, J.~Wilpon, and B.~Juang.
\newblock A segmental {K}-means training procedure for connected word
  recognition.
\newblock {\em AT\&T {T}ech. J.}, 64(3):21--40, 1986.

\bibitem{philips2}
V.~Steinbiss, H.~Ney, X.~Aubert, and {\em et.~al.}
\newblock The {P}hilips research system for continuous-speech recognition.
\newblock {\em Philips J. Res.}, 49:317--352, 1995.

\bibitem{strom}
N.~Str\"{o}m, L.~Hetherington, T.~Hazen, E.~Sandness, and J.~Glass.
\newblock Acoustic {M}odeling {I}mprovements in a {S}egment-{B}ased {S}peech
  {R}ecognizer.
\newblock In {\em Proc. IEEE ASRU Workshop Keystone, CO, USA}, 1999.

\end{thebibliography}
\end{document}